\documentclass[12pt]{article}
\usepackage{amsmath,amsthm,amssymb,latexsym,color}
\usepackage{bbm}

\usepackage{graphicx}

\usepackage{hyperref}

\date{\today}

\theoremstyle{plain}

\newtheorem*{theorem*}{Theorem}

\def\kkk{\null\hfill $\Box $ \\}

\newtheorem{theor}{Theorem}[section]
\newtheorem*{theor*}{Theorem}
\newtheorem*{conj*}{Conjecture}
\newtheorem{conj}[theor]{Conjecture}

\newtheorem{prop}[theor]{Proposition}

\newtheorem{cor}[theor]{Corollary}
\newtheorem{lemma}[theor]{Lemma}
\newtheorem{rem}[theor]{Remark}
\theoremstyle{remark}

\newtheorem{Problem}[theor]{Problem}

\def\ST{{\mathcal S}_n}
\def\Rv{{\mathcal R}}
\def\R{{\mathbb R}}
\def\N{{\mathbb N}}
\def\Z{{\mathbb Z}}
\def\C{\mathbb{C}}

\newcommand{\lc}{\lceil}
\newcommand{\rc}{\rceil}
\def\la{\left\langle}
\def\r{\right}
\def\ra{\r\rangle}

\def\Bern{Bernoulli($p_n$) }
\def\Ber{Bernoulli($p$) }
\def\Berq{Bernoulli($q$) }

\def\Prob{{\mathbb P}}
\def\Exp{{\mathbb E}}
\def\E{{\mathbb E}}
\def\Event{{\mathcal E}}
\def\cf{\mathcal{Q}}
\newcommand{\p}{\mathbb{P}}

\newcommand{\Col}{\col}

\def\bal{{\bf UD}}
\def\dist{{\rm dist}}

\def\col{{\bf C}}
\def\row{{\bf R}}

\def\spn{{\rm span}\,}
\def\supp{{\rm supp\, }}

\def\Net{{\mathcal N}}

\def\gncvectors{{\mathcal V}}

\def\normr{{\Upsilon}}

\newcommand{\lam}{\lambda}
\newcommand{\eps}{\varepsilon}

\def\AA{{\mathcal A}_N}
\def\MSet{{\mathcal M}}

\def\Mc{{\mathcal{M}_{n}}}

\def\Om0{\Omega_0}

\def\edv{{\bf e}}
\def\qand{\quad \mbox{ and } \quad}

\def\nx{\| x \|}

\newcommand{\st}{\mathcal{T}}
\newcommand{\ct}{C_\tau}
\newcommand{\stt}{\mathcal{B}}
\newcommand{\il}{I^\ell}
\newcommand{\ir}{I^r}

\def\f{{\mathcal F}}

\def\ii{{\mathcal I}}

\def\nn{n_{s_0+3}}

\def\BB{\mathcal{AC}}

\def\aaa{r}

\def\gfn{{\bf g}}

\title{Singularity of sparse Bernoulli matrices}
\author{Alexander E. Litvak
and Konstantin E. Tikhomirov}

\newcommand\address{\noindent\leavevmode

\medskip
\noindent
Alexander E. Litvak\\
Dept.~of Math.~and Stat.~Sciences,\\
University of Alberta, \\
Edmonton, AB, Canada, T6G 2G1.\\
\texttt{\small
e-mail:  aelitvak@gmail.com}\\

\medskip

\noindent
Konstantin E. Tikhomirov\\
School~of Math., GeorgiaTech,\\
686 Cherry street,\\
Atlanta, GA 30332, USA.\\
\texttt{\small
e-mail:   ktikhomirov6@gatech.edu}\\
}

\begin{document}

\maketitle

\begin{abstract}
Let $M_n$ be an $n\times n$ random matrix with i.i.d.\ \Ber entries. We show that there is a universal constant
$C\geq 1$ such that, whenever $p$ and $n$ satisfy $C\log n/n\leq p\leq C^{-1}$,
\begin{align*}
\Prob\big\{\mbox{$M_n$ is singular}\big\}&=(1+o_n(1))\Prob\big\{\mbox{$M_n$ contains a zero row or column}\big\}\\
&=(2+o_n(1))n\,(1-p)^n,
\end{align*}
where $o_n(1)$ denotes a quantity  which converges to zero as $n\to\infty$. 
We provide the corresponding upper and lower bounds on the smallest singular value of $M_n$ as well.
\end{abstract}

\bigskip

{\small
}

\tableofcontents

\section{Introduction}
\label{s:intro}

Invertibility of discrete random matrices attracts considerable attention in the literature.
The classical problem in this direction --- estimating the singularity probability of a square
random matrix $B_n$ with i.i.d.\ $\pm 1$ entries --- was first addressed by Koml\'os
in the 1960-es. Koml\'os \cite{Komlos} showed that $\Prob\{\mbox{$B_n$ is singular}\}$
decays to zero as the dimension grows to infinity. A breakthrough result of Kahn--Koml\' os--Szemer\' edi
\cite{KKS95} confirmed that the singularity probability of $B_n$ is exponentially small in dimension.
Further improvements on the singularity probability were obtained by Tao--Vu \cite{TV disc1, TV bernoulli}
and Bourgain--Vu--Wood \cite{BVW}.
An old conjecture states that $\Prob\{\mbox{$B_n$ is singular}\}=\big(\frac{1}{2}+o_n(1)\big)^n$.
The conjecture was resolved in \cite{KT-last}.

Other models of non-symmetric discrete random matrices considered in the literature include
adjacency matrices of $d$-regular digraphs,
as well as the closely related model of sums of independent uniform permutation matrices
\cite{LSY, Cook adjacency, Cook circular, LLTTY-cras, LLTTY:15, LLTTY first part, LLTTY third part,
LLTTY rank n-1, BCZ}.
In particular, the recent breakthrough works \cite{Huang, Mez, NW} confirmed that the adjacency
matrix of a uniform random $d$--regular digraph of a constant degree $d\geq 3$ is non-singular with probability
decaying to zero as the number of vertices of the graph grows to infinity.
A closely related line of research deals with the rank of random matrices over finite fields.
We refer to \cite{LMN} for some recent results and further references.

The development of the {\it Littlewood--Offord theory}
and a set of techniques of geometric functional analysis reworked in the random matrix context,
produced strong invertibility results for a broad class of distributions. Following works \cite{TV ann math,R-ann}
of Tao--Vu and Rudelson,
the paper \cite{RV} of Rudelson and Vershynin established optimal small ball probability estimates for the
smallest singular value in the class of square matrices with i.i.d.\ subgaussian entries,
namely, it was shown that any $n\times n$ matrix $A$ with i.i.d.\ subgaussian entries
of zero mean and unit variance satisfies $\Prob\{s_{\min}(A)\leq t\,n^{-1/2}\}\leq Ct+2\exp(-cn)$
for all $t>0$ and some $C,c>0$ depending only on the subgaussian moment.
The assumptions of identical distribution of entries and of bounded subgaussian moment
were removed in subsequent works \cite{RT,Livshyts,LTV}.
This line of research lead to positive solution of the Bernoulli matrix conjecture mentioned in the first paragraph.
Let us state the result of \cite{KT-last} for future reference.
\begin{theor*}[Invertibility of dense Bernoulli matrices, \cite{KT-last}]\hspace{0cm}
\begin{itemize}

\item For each $n$, let $B_n$ be the $n\times n$ random matrix with i.i.d.\ $\pm 1$ entries.
Then for any $\varepsilon>0$ there is $C$ depending only on $\varepsilon$ such that
the smallest singular value $s_{\min}(B_n)$ satisfies
$$
\Prob\big\{s_{\min}(B_n)\leq t n^{-1/2}\big\}\leq Ct+C(1/2+\varepsilon)^n,\quad t>0.
$$
In particular, $\Prob\big\{\mbox{$B_n$ is singular}\big\}=(1/2+o_n(1))^n$,
where the quantity $o_n(1)$ tends to zero as $n$ grows to infinity.

\item
For each $\varepsilon>0$ and $p\in(0,1/2]$ there is $C>0$ depending on $\varepsilon$ and $p$
such that for any $n$ and for random $n\times n$ matrix $M_n$ with i.i.d.\ \Ber entries,
$$
\Prob\big\{s_{\min}(M_n)\leq t n^{-1/2}\big\}\leq Ct+C(1-p+\varepsilon)^n,\quad t>0.
$$
In particular, for a fixed $p\in(0,1/2]$, we have $\Prob\big\{\mbox{$M_n$ is singular}\big\}=(1-p+o_n(1))^n$.
\end{itemize}
\end{theor*}

\bigskip

Sparse analogs of the Rudelson--Vershynin invertibility theorem \cite{RV} were obtained, in particular, in works
\cite{TV sparse, GT, LiRi,  BasRud, BasRud circ, BasRud-sharp}, with the strongest small ball probability estimates
in the i.i.d.\ subgaussian setting available in \cite{BasRud, BasRud circ, BasRud-sharp}.
Here, we state a result of Basak--Rudelson \cite{BasRud} for
Bernoulli($p_n$) random matrices.
\begin{theor*}[Invertibility of sparse Bernoulli matrices, \cite{BasRud}]
There are universal constants $C,c>0$ with the following property. Let $n\in\N$ and let
$p_n\in(0,1)$ satisfy $C\log n/n\leq p_n\leq 1/2$. Further, let $M_n$ be the
random $n\times n$ matrix with i.i.d. Bernoulli($p_n$) entries (that is, $0/1$ random variables
with expectation $p$). Then
\begin{align*}
\Prob\big\{s_{\min}(M_n)\leq t\,\exp\big(-C\log(1/p_n)/\log(np_n)\big)\,\sqrt{p_n/n}\big\}\leq Ct+2\exp(-cnp_n),\quad t>0.
\end{align*}
\end{theor*}

\medskip

The singularity probabilities implied by the results \cite{KT-last,BasRud} 
may be regarded as suboptimal
in a certain respect. Indeed, while \cite{KT-last} produced an
asymptotically sharp base of the power in the singularity probability
of $B_n$,
the estimate of \cite{KT-last}
is off by a factor $(1+o_n(1))^n$ which may (and in fact does, as analysis of the proof shows)
grow to infinity with $n$ superpolynomially fast.
Further, the upper bound on the singularity probability of sparse Bernoulli matrices implied by \cite{BasRud}
captures an exponential dependence on $n p_n$, but does not recover an asymptotically optimal base of the power.

A folklore conjecture for matrices $B_n$ asserts that
$\Prob\{\mbox{$B_n$ is singular}\}=(1+o_n(1))n^2 2^{1-n}$,
where the right hand side of the expression is the probability that two rows or two columns
of the matrix $B_n$ are equal up to a sign (see, for example, \cite{KKS95}).
This conjecture can be naturally extended to the model with Bernoulli($p_n$) ($0/1$) entries as follows.
\begin{conj}[Stronger singularity conjecture for Bernoulli matrices]
For each $n$, let $p_n\in (0,1/2]$, and let $M_n$ be the $n\times n$
matrix with i.i.d.\ Bernoulli($p_n$) entries. Then
\begin{align*}
\Prob\{&\mbox{$M_n$ is singular}\}\\
&=(1+o_n(1))
\Prob\big\{\mbox{a row or a column of $M_n$ equals zero, or two rows or columns are equal}\big\}.
\end{align*}
In particular, if $\limsup p_n<1/2$ then
\begin{align*}
\Prob\{\mbox{$M_n$ is singular}\}&=(1+o_n(1))
\Prob\big\{\mbox{either a row or a column of $M_n$ equals zero}\big\}.
\end{align*}
\end{conj}
Conceptually, the above conjecture asserts that the main causes for singularity are local in the sense
that the linear dependencies typically appear within small subsets of rows or columns.
In a special regime $n p_n\leq \ln n+o_n(\ln \ln n)$, the conjecture was positively
resolved in \cite{BasRud-sharp}  (note that if $n p_n\leq \ln n$ then the matrix
has a zero row with probability at least $1-1/e-o_n(1)$). However,
the regime $\liminf (np_n/\log n)> 1$ was not covered in \cite{BasRud-sharp}.

The main purpose of our paper is to develop methods capable of capturing the singularity probability with a
sufficient precision to answer the above question. Interestingly, this appears to be more accessible
in the sparse regime, when $p_n$ is bounded above by a small universal constant
(we  discuss this in the next section in more detail).
It is not difficult to show that when $\liminf (np_n/\ln n)>1$,
the events that a given row or a given column equals zero, almost do not intersect, so that
$$
\Prob\big\{\mbox{either a row or a column of $M_n$ equals zero}\big\}=(2+o_n(1))n\,(1-p_n)^n.
$$
Our  main result  can be formulated as follows.
\begin{theor}\label{th: main}
There are universal constants $C,\widetilde C\geq 1$ with the following property.
Let $n\geq 1$ and let $M_n$ be an $n\times n$ random matrix such that
\begin{equation}\tag{{\bf{}A}}\label{eq: assumptions}
\mbox{The entries of $M_n$ are i.i.d.\ \Ber, with $p=p_n$ satisfying $C\ln n\leq np\leq C^{-1}$.}
\end{equation}
Then
$$\Prob\big\{\mbox{$M_n$ is singular}\big\}=(2+o_n(1))n\,(1-p)^n, $$
where $o_n(1)$ is a quantity which tends to zero as $n\to\infty$.
Moreover, for every $t>0$,
$$
\Prob\big\{s_{\min}(M_n)\leq t\, \exp(-3\ln^2(2n))\big\}\leq t+(1+o_n(1))\Prob\big\{\mbox{$M_n$ is singular}\big\}
=t+(2+o_n(1))n\,(1-p)^n.
$$
\end{theor}
In fact, our approach gives much better estimates on $s_{\min}$ in the regime when $p_n$ is constant,
see Theorem~\ref{const-p-th} below.
At the same time, we note that obtaining small ball probability estimates for $s_{\min}$
was not the main objective of this paper, and the argument was not fully optimized in that respect.


Geometrically, the main result of our work asserts that (under appropriate assumptions on $p_n$)
the probability that
a collection of $n$ independent random vectors $X_1^{(n)},\dots,X_n^{(n)}$ in $\R^n$,
with i.i.d Bernoulli($p_n$) components is linearly dependent is equal (up to $(1+o_n(1))$ factor)
to probability of the event that either $X_i^{(n)}$ is zero for some $i\leq n$ or $X_1^{(n)},\dots,X_n^{(n)}$ are contained
in the same coordinate hyperplane:
\begin{align*}
&\Prob\big\{\mbox{$X_1^{(n)},\dots,X_n^{(n)}$ are linearly dependent}\big\}
=(1+o_n(1))\,\Prob\big\{X_i^{(n)}={\bf 0}\mbox{ for some }i\leq n\big\}\\
&\hspace{1cm}+(1+o_n(1))\,\Prob\big\{\exists\,
\mbox{ a coordinate hyperplane $H$ such that }X_i^{(n)}\in H\mbox{ for all }i\leq n\big\}.
\end{align*}
Thus, the linear dependencies between the vectors, when they appear,
typically have the prescribed structure, falling into one of the two categories described above with the (conditional)
probability $\frac{1}{2}+o_n(1)$.


The paper is organized as follows.
In the next section, we give an overview of the proof of the main result.
In Section~\ref{preliminaries}, we gather some preliminary facts and important notions to be used later.
In Section~\ref{s: unstructured}, we consider new anti-concentration inequalities for random $0/1$ vectors with prescribed
number of non-zero components, and introduce a functional ({\it u-degree} of a vector)
which enables us to classify vectors on the sphere according to anti-concentration properties of
inner products with the random $0/1$ vectors.
In the same section, we prove a key technical result --- Theorem~\ref{th: gradual} --- which states, roughly speaking,
that with very high probability a random unit vector orthogonal to $n-1$ columns of $M_n$ is either close
to being sparse or to being a constant multiple of $(1,1,\dots,1)$, or the vector is {\it very unstructured,}
i.e., has a very large u-degree.

In Section~\ref{steep:constant p}, we consider a special regime of constant probability of success $p$.
In this regime, estimating the event that $M_n$ has an almost null vector which is either close to sparse or almost constant,
is relatively simple. The reader who is interested only in the regime of constant $p$ can thus skip
the more technical Section~\ref{s: steep} and have the proof of the main result as a combination of
the theorems in Sections~\ref{s: unstructured} and~\ref{steep:constant p}. In Section~\ref{s: steep}, we consider the entire range for $p$.
Here, the treatment of almost constant and close to sparse null vectors is much more challenging and involves
a careful analysis of multiple cases. Finally, in Section~\ref{s: main th} we establish an {\it invertibility via distance}
lemma and prove the main result of the paper. Some open questions are discussed in Section~\ref{s: further}.

\section{Overview of the proof}\label{s: overview}

In this section, we provide a high-level overview of the proof; technical details will be discussed further in the text.
The proof utilizes some known approaches to the matrix invertibility, which involve, in particular,
a decomposition of the space into {\it structured}
and {\it  unstructured} parts, a form of {\it invertibility via distance} argument, small ball probability estimates based on the Esseen
lemma, and various forms of the $\varepsilon$--net argument.
The novel elements of the proof are anti-concentration inequalities for random vectors with a prescribed
cardinality of the support, a structural theorem for normals to random hyperplanes spanned by vectors with
i.i.d. \Ber components, and a sharp analysis of the matrix invertibility over the set of structured vectors.
We will start the description with our use of the partitioning trick, followed by a modified
invertibility-via-distance lemma, and then consider the anti-concentration inequality
and the theorem for normals (Subsection~\ref{ss: anti-c})
as well as invertibility over the structured vectors (Subsection~\ref{ss: steep overview}).

The use of decompositions of the
space $\R^n$ into structured and unstructured vectors has become rather standard in the literature.
A common idea behind such partitions is to apply the Littlewood--Offord theory to analyse the unstructured vectors
and to construct a form of the $\varepsilon$--net argument to treat the structured part.
Various definitions of structured and unstructured have been used in works dealing with the matrix invertibility.
One of such decomposition was introduced in \cite{LPRT} and further developed in \cite{RV}. In this splitting
the structured vectors are {\it compressible}, having a relatively small
Euclidean distance to the set of {\it sparse} vectors, while the vectors
in the complement are {\it incompressible}, having a large distance to sparse vectors and, as a consequence, many components
of roughly comparable magnitudes.
In our work, the decomposition of $\R^n$ is closer to the one introduced in \cite{LLTTY first part,LLTTY-TAMS}.

Let $x^*$ denote a non-increasing rearrangement of absolute values of components of
a vector $x$, and let $r,\delta,\rho\in(0,1)$ be some parameters.
Further, let $\gfn$ be a non-decreasing function from $[1,\infty)$ into $[1,\infty)$;
we shall call it {\it the growth function}. At this moment, the choice of the growth function is not important;
we can assume that $\gfn(t)$ grows roughly as $t^{\ln t}$.
Define the set of {\it gradual non-constant vectors} as
\begin{align}
\label{eq: gnc def}
\gncvectors_n=\gncvectors_n(r,\gfn,\delta,\rho)
&:=\big\{x\in\R^n\, :\, x^*_{\lfloor rn\rfloor}=1,
\;x^*_i\leq \gfn(n/i)\mbox{ for all $i\leq n$}, \,\,\, \mbox{ and } \nonumber
\\&\exists\,
Q_1,Q_2\subset[n]  \quad \mbox{ such that } \quad
 |Q_1|,|Q_2|\geq \delta n \qand  \max\limits_{i\in Q_2} x_i\leq \min\limits_{i\in Q_1}x_i- \rho
\big\}.
\end{align}
In a sense, constant multiples of the gradual non-constant vectors occupy most of the space $\R^n$,
they play role of the unstructured vectors in our argument. By negation, the structured vectors,
\begin{equation}\label{strvect}
 \ST=\ST(r,\gfn,\delta,\rho):=\R^n\setminus \bigcup_{\lambda\geq 0}(\lambda \,
\gncvectors_n(r,\gfn,\delta,\rho)),
\end{equation}
 are either almost constant (with most of components nearly equal)
or have a very large ratio of $x^*_i$ and $x^*_{\lfloor rn\rfloor}$ for some $i< rn$.

For simplicity, we only discuss the problem of singularity at this moment.
As $M_n$ and $M_n^\top$ are equidistributed, to show that
$\Prob\big\{\mbox{$M_n$ is singular}\big\}=(2+o_n(1))n\,(1-p)^n,$ it is sufficient to verify that
\begin{equation}\label{eq: aux 209582705982}
\begin{split}
\Prob\Big(&\big\{M_n x=0\mbox{ for some }x\in \gncvectors_n\big\}
\cap \Big\{M_n^\top x\neq 0\mbox{ for all }x\in \ST \Big\}\Big)=o_n(n)\,(1-p)^n,
\end{split}
\end{equation}
and
$$
\Prob\Big\{M_n x=0\mbox{ for some }x\in \ST
\Big\}=(1+o_n(1))n\,(1-p)^n.
$$
The first relation is dealt with by using a variation of the {\it invertibility via distance} argument
which was introduced in  \cite{RV} to obtain sharp small ball probability estimates for the smallest singular value.
 In the form given in \cite{RV}, the argument reduces the problem of invertibility over unstructured
vectors to estimating distances of the form $\dist(\col_i(M_n),H_i(M_n))$, where
$\col_i(M_n)$ is the $i$--th column of $M_n$, and
$H_i(M_n)$ is the linear
span of columns of $M_n$ except for the $i$--th. In our setting, however, the argument needs to be modified to pass to
estimating the distance {\it conditioned} on the size of the support of the column,
as this allows using much stronger anti-concentration inequalities (see the following subsection).
By the invariance of the distribution of $M_n$ under permutation of columns, it can be shown that in order to prove the
relation \eqref{eq: aux 209582705982}, it is enough to verify that
\begin{equation}\label{eq: aux -8572059873}
\Prob\big\{|\supp\Col_1(M_n)|\in [\mbox{$\frac{pn}{8}$}, 8pn]\mbox{ and }
\langle{\bf Y},\Col_1(M_n)\rangle=0\mbox{ and }
{\bf Y}/{\bf Y}^*_{\lfloor r n\rfloor}\in \gncvectors_n
\big\}=o_n(n)\,(1-p)^n,
\end{equation}
where ${\bf Y}$ is a non-zero random vector orthogonal to and measurable with respect to $H_1(M_n)$
(see Lemma~\ref{l: inv via dist} and the beginning of the proof of Theorem~\ref{th: main}).
In this form, the question can be reduced to studying the anti-concentration of the linear
combinations $\sum_{i=1}^n{\bf Y}_i b_i$, where the Bernoulli random variables $b_1,\dots,b_n$ are mutually independent with ${\bf Y}$
and conditioned to sum up to a fixed number in $[pn/8, 8pn]$.
This intermediate problem is discussed in the next subsection.

\subsection{New anti-concentration inequalities for random vectors with prescribed support cardinality}
\label{ss: anti-c}

The {\it Littlewood--Offord theory} --- the study of anti-concentration properties of random variables
--- has been a crucial ingredient of many recent results on invertibility of random matrices, starting
with the work of Tao--Vu \cite{TV ann math}.
In particular, the breakthrough result \cite{RV} of Rudelson--Vershynin mentioned in the introduction, is largely based on studying
the L\'evy function
$\cf(\langle \col_1(A),{\bf Y}\rangle,t)$, with $\col_1(A)$ being the first column of the random matrix $A$ and
${\bf Y}$ --- a random unit vector orthogonal to the remaining columns of $A$.

We recall that given a random vector $X$ taking values in $\R^n$, the {\it L\' evy concentration function}
$\cf(X,t)$ is defined by
$$
\cf(X,t):=\sup\limits_{y\in\R^n}\Prob\big\{\|X-y\|\leq t\big\},\quad t\geq 0;
$$
in particular for a scalar random variable $\xi$ we have $\cf(\xi,t):=\sup\limits_{\lambda\in\R}\Prob\{|\xi-\lambda|\leq t\}$.
A common approach is to determine structural properties of a fixed vector which would
imply desired upper bounds on the L\'evy function of its scalar product with a random vector (say, a matrix' column).
The classical result of Erd\H os--Littlewood--Offord \cite{Erdos,LO lemma} asserts that whenever $X$
is a vector in $\R^n$ with i.i.d.\ $\pm 1$ components, and $y=(y_1,\dots,y_n)\in\R^n$ is such that
$|y_i|\geq 1$ for all $i$, we have
$$
\cf(\langle X,y\rangle,t)\leq Ct\,n^{-1/2}+Cn^{-1/2},
$$
where $C>0$ is a universal constant.
It can be further deduced from the L\'evy--Kolmogorov--Rogozin inequality \cite{Rog} that the above assertion remains true
whenever $X$ is a random vector with independent components $X_i$ satisfying $\cf(X_i,c)\leq 1-c$
for some constant $c>0$.
More delicate structural properties, based on whether components of $y$ can be embedded into a generalized
arithmetic progression with prescribed parameters were employed in \cite{TV ann math} to prove
superpolynomially small upper bounds on the singularity probability of discrete random matrices.

The Least Common Denominator (LCD) of a unit vector introduced in \cite{RV} played a central role in establishing
the exponential upper bounds on the matrix singularity under more general assumptions on the entries' distributions.
We recall that the LCD of a unit vector $y$ in $\R^n$ can be defined as
\begin{equation}\label{eq: LCD overview}
{\rm LCD}(y):=\inf\big\{\theta>0:\;\dist(\theta y,\Z^n)\leq\min(c_1\|\theta y\|,c_2\sqrt{n})\big\}
\end{equation}
for some parameters $c_1,c_2\in(0,1)$. The small ball probability theorem of Rudelson and Vershynin \cite{RV}
states that given a vector $X$ with i.i.d.\ components of zero mean and unit variance satisfying some additional mild assumptions,
$$
\cf(\langle X,y\rangle,t)\leq Ct+\frac{C'}{{\rm LCD}(y)}+2e^{-c' n}
$$
for some constants $C,C',c'>0$ (see \cite{RV09} for a generalization of the statement).
The LCD, or its relatives, were subsequently used in studying
invertibility of non-Hermitian square matrices under broader assumptions \cite{RT,Livshyts,LTV},
and delocalization of eigenvectors of non-Hermitian random matrices
\cite{RV no gaps,LOR,LyTi}, among many other works.

\medskip

Anti-concentration properties of random linear combinations naturally play a central role in the current work, however,
the measures of unstructuredness of vectors existing in the literature do not allow to obtain the precise
estimates we are aiming for. Here, we develop a new functional for dealing with linear combinations of
{\it dependent} Bernoulli variables.

Given $n\in\N$, $1\leq m\leq n/2$, a vector $y\in\R^n$ and parameters $K_1,K_2\geq 1$, we define the {\it degree
of unstructuredness (u-degree)} of vector $y$ by
\begin{align}\label{udeg}
\bal_n(y,m,K_1,K_2):=\sup\bigg\{&t>0:\; A_{nm}\sum\limits_{S_1,\dots,S_m}\;
\int\limits_{-t}^t \prod\limits_{i=1}^{m}\psi_{K_2}\big(
\big|\Exp\exp\big(2\pi{\bf i}\,y_{\eta[S_i]}\,m^{-1/2} s\big)\big|\big)\,ds\leq K_1\bigg\},
\end{align}
where the sum is taken over all sequences $(S_i)_{i=1}^m$ of disjoint subsets $S_1,\dots,S_m\subset[n]$,
each of cardinality $\lfloor n/m\rfloor$ and
\begin{align}\label{anm}
  A_{nm} = \frac{\big((\lfloor n/m\rfloor)!\big)^m\,(n-m\lfloor n/m\rfloor)!}{n!}\cdot
\end{align}
Here $\eta[S_i]$, $i\leq m$, denote mutually independent integer random variables uniformly distributed on respective $S_i$'s.
The function $\psi_{K_2}$ in the definition acts as a smoothing of $\max(\frac{1}{K_2},t)$,
with $\psi_{K_2}(t)=\frac{1}{K_2}$ for all $t\leq \frac{1}{2K_2}$ and $\psi_{K_2}(t)= t$ for all $t\geq \frac{1}{K_2}$
(we prefer to skip discussion of this purely technical element of the proof in this section, and
refer to the beginning of Section~\ref{s: unstructured} for the full list of conditions imposed on $\psi_{K_2}$).

The functional $\bal_n(y,m,K_1,K_2)$ can be understood as follows.
The expression inside the supremum is the average value of the integral
$$
\int\limits_{-t}^t \prod\limits_{i=1}^{m}\psi_{K_2}\big(
\big|\Exp\exp\big(2\pi{\bf i}\,y_{\eta[S_i]}\,m^{-1/2} s\big)\big|\big)\,ds,
$$
with the average taken over all choices of sequences $(S_i)_{i=1}^m$.
The function under the integral, disregarding the smoothing $\psi_{K_2}$, is the
absolute value of the characteristic function of the random variable $\langle y,Z\rangle$, where
$Z$ is a random $0/1$--vector with exactly $m$ ones, and with the $i$-th one distributed uniformly on
$S_i$. A relation between the magnitude of the characteristic function and anti-concentration properties of a random variable (the  Esseen lemma (Lemma~\ref{ess} below)
has been commonly used in works on the matrix invertibility (see, for example, \cite{Rud13}),
and determines the shape of the functional $\bal_n(\cdot)$.
The definition of the u-degree is designed
specifically to work with random $0/1$--vectors having a fixed sum (equal to $m$).
The next statement follows from the definition of $\bal_n(\cdot)$
and the Esseen lemma.

\begin{theor}[A Littlewood--Offord-type inequality in terms of the u-degree]\label{p: cf est}
Let $m,n$ be positive integers with $m\leq n/2$, and let $K_1,K_2\geq 1$.
Further, let $v\in\R^n$, and let $X=(X_1,\dots,X_n)$ be a random $0/1$--vector in $\R^n$
uniformly distributed on the set of vectors with $m$ ones and $n-m$ zeros.
Then
$$
\cf\Big(\sum\limits_{i=1}^n v_i X_i,\sqrt{m}\,\tau\Big)
\leq C_{\text{\tiny\ref{p: cf est}}}\,\big(\tau+\bal_n(v,m,K_1,K_2)^{-1}\big)\quad \mbox{for all $\tau>0$},
$$
where $C_{\text{\tiny\ref{p: cf est}}}>0$ may only depend on $K_1$.
\end{theor}
The principal difference of the u-degree and the above
theorem from the notion of the LCD and \eqref{eq: LCD overview}
is that the former allow to obtain stronger anti-concentration inequalities in the same regime of sparsity,
assuming that the coefficient vector $y$ is sufficiently unstructured.
In fact, under certain conditions, {\it sparse random $0/1$ vectors with prescribed support cardinality
admit stronger anti-concentration inequalities compared to the i.i.d.\ model.}

The last principle can be illustrated by taking the coefficient vector $y$
as a ``typical'' vector on the sphere $S^{n-1}$.
First, assume that $b_1,\dots,b_n$ are i.i.d.\ \Ber, with $p< 1/2$. Then it is easy to see that
for almost all (with respect to normalized Lebesgue measure) vectors $y\in S^{n-1}$,
$$
\cf\Big(\sum_{i=1}^n y_i b_i,0\Big)=(1-p)^n.
$$
In words, for a typical coefficient vector $y$ on the sphere, the linear combination $\sum_{i=1}^n y_i b_i$
takes distinct values for any two distinct realizations of $(b_1,\dots,b_n)$, and thus the L\'evy
function at zero is equal to the probability measure of the largest atom of the distribution of $\sum_{i=1}^n y_i b_i$
which corresponds to all $b_i$ equal to zero.
In contrast, if the vector $(b_1,\dots,b_n)$ is uniformly distributed on the
set of $0/1$--vectors with support of size $d=pn$, then for almost all $y\in S^{n-1}$,
the random sum $\sum_{i=1}^n y_i b_i$ takes ${n\choose d}$ distinct values.
Thus,
$$
\cf\Big(\sum_{i=1}^n v_i b_i,0\Big)={n\choose np}^{-1},
$$
where ${n\choose np}^{-1}\ll (1-p)^n$ for small $p$.

The above example provides only qualitative estimates
and does not give an information on the location of the atoms of the distribution of $\sum_{i=1}^n y_i b_i$.
The notion of the u-degree addresses this problem.
The following theorem, which is the main result of Section~\ref{s: unstructured},
asserts that with a very large probability the normal vector to the (say, last) $n-1$ columns of our matrix $M_n$
is either very structured or has a very large u-degree, much greater than
the critical value $(1-p)^{-n}$.
\begin{theor}
\label{th: gradual}
Let $r,\delta,\rho\in(0,1)$, $s>0$, $R\geq 1$, and let $K_3\geq 1$. Then there are $n_0\in\N$, $C\geq 1$
and $K_1\geq 1$, $K_2\geq 4$ depending on $r,\delta,\rho,R,s,K_3$ such that the following holds.
Let $n\geq n_0$, $p\leq C^{-1}$, and  $s\ln n\leq pn$.
Let $\gfn \, : [1,\infty) \to [1,\infty)$ be an increasing (growth) function satisfying
\begin{align}\label{gfncond}
\forall a\geq 2\,\,\forall t\geq 1:\,\,\,\,  \gfn(a\,t)\geq \gfn(t)+a
\quad \quad \qand \quad\quad
\prod_{j=1}^\infty \gfn(2^j)^{j\,2^{-j}}\leq K_3.
\end{align}
Assume that $M_n$ is
an $n\times n$ \Ber random matrix. Then
with probability at least $1-\exp(-R pn)$ one has
\begin{align*}
&\{\mbox{Set of normal vectors to $\col_2(M_n),\dots,\col_{n}(M_n)\}\cap\gncvectors_n(r,\gfn,\delta,\rho)\subset$}\\
&\hspace{3cm}\mbox{$\{x\in\R^n:\;x^*_{\lfloor rn\rfloor}=1,\,\, \,
\bal_n(x,m,K_1,K_2)\geq \exp(Rpn)\,\,$ for all $\,\, pn/8\leq m\leq 8pn\}$.}
\end{align*}
\end{theor}
We would like to emphasize that the parameter $s$ in this theorem
can take values less than one, in the regime when the matrix $M_n$ typically has null rows and columns.
In this respect, the restriction $p\geq C\ln n/n$ in the main theorem comes from the treatment of structured vectors.

The proof of Theorem~\ref{th: gradual} is rather involved, and is based on a double counting argument and
specially constructed lattice approximations of the normal vectors.
We refer to Section~\ref{s: unstructured} for details.
Here, we only note that, by taking $R$ as a sufficiently large constant, the theorem implies the
relation \eqref{eq: aux -8572059873}, hence, accomplishes the treatment of unstructured vectors.

\subsection{Almost constant, steep and $\Rv$-vectors}
\label{ss: steep overview}

In this subsection we discuss our treatment of the set of structured vectors, $\ST$.
In the proof we partition the set $\ST$ into several subsets and work with them separately.
In a simplistic form, the structured vectors are dealt with in two ways:
either by constructing discretizations and taking the union bound (variations of the $\varepsilon$--net argument),
or via deterministic estimates
in the case when there are very few very large components in the vector.
We note here that the discretization procedure has to take into account the non-centeredness
of our random matrix model: while in case of centered matrices with i.i.d. components (and under appropriate moment conditions)
the norm of the matrix is typically of order $\sqrt{n}$ times the standard deviation
of an entry, for our \Ber model it has order $pn$ (i.e., roughly $\sqrt{p}n$ times the standard deviation
of an entry), which makes a direct application of the $\varepsilon$--net
argument impossible. Fortunately, this
   large norm is attained only in one direction --- the direction of the vector ${\bf 1}=(1, 1, \dots, 1)$ while
  on the orthogonal complement of ${\bf 1}$ the typical norm is $\sqrt{pn}$. Therefore it is enough
    to take a standard net in the Euclidean norm and to make it denser in that one direction, which almost does not
    affect the cardinality of the net.
    We refer to Section~\ref{net}  for details.

Let us first describe our approach in the (simpler) case when $p\in (q, c)$, where $c$ is a small enough absolute
constant and $q\in (0,c)$ is a fixed parameter (independent of $n$). We introduce four auxiliary sets and
show that the set of unit structured vectors, $\ST\cap S^{n-1}$, is contained in the closure of their union. 

The first set, $\stt_1$, consists of unit vectors close to vectors of
the canonical basis, specifically, unit vectors $x$ satisfying $x_1^* > 6pn  x_2^*$,
where $x^*$ denotes the non-inreasing rearrangement of the vector $(|x_i|)_{i\leq n}$. For any such vector $x$
the individual bound is rather straightforward --- conditioned on the event that there are no zero columns in our matrix $M$,
and that the Euclidean norms of the matrix rows are not too large, we get
$Mx\neq 0$. This class is the main contributor
to the bound $(1+o_n(1))n(1-p)^n$ for non-invertibility over the structured vectors $\ST$.

For the other three sets we
use anti-concentration probability estimates and discretizations. An application of Rogozin's lemma (Proposition~\ref{rog})
implies that probability of having small inner product of a given row of our matrix with $x$ is small,
provided that there is a subset $A\subset [n]$ such that the maximal coordinate of $P_A x$ is
bounded above by $c\sqrt{p}\|P_A x\|$, where $\|\cdot\|$ denotes the
standard Euclidean norm and $P_A$ is the coordinate projection onto $\R^A$. 
Combined with tensorization Lemma~\ref{tens} this implies exponentially (in $n$)
small probability of the event that $\|Mx\|$ is close to zero --- see Proposition~\ref{rogozin} below.
Specifically, we define
$\stt_2$ as the set of unit vectors satisfying the above condition with $A=[n]$, that is,
satisfying $x_1^*\leq c\sqrt{p}$, and for $\stt_3$ we take all unit vectors satisfying the condition
with $A=\sigma_x ([2, n])$, that is, satisfying $x_2^*\leq c\sqrt{p} \|P_{\sigma_x ([2, n])}x\|$,
where 
$\sigma_x$ is a permutation satisfying $x^*_i=|x_{\sigma_x(i)}|$, $i\leq n$. 
For vectors from these two sets we have very good individual probability estimates, but, unfortunately,
the complexity of both sets is large --- they don't admit nets of small cardinality. To overcome this issue, we
have to redefine these sets by intersecting them with specially chosen sets of vectors having many almost equal coordinates.
For the precise definition of such sets, denoted by $U(m, \gamma)$, see Subsection~\ref{net}. A set $U(m, \gamma)$ is a variant
of the class of {\it almost constant} vectors, $\BB(\rho)$ (see (\ref{acv}) below), introduced to deal with general $p$.
Having a large part of coordinates of a vector almost equal to each other reduces the complexity of the set making
possible to construct a net of small cardinality. This resolves the problem and allows us to deal with these
two classes of sets. The remaining class of vectors, $\stt_4$, consists of vectors $x$
with $x_1^*\geq x_2^*\geq c\sqrt{p} \|P_{\sigma_x ([2, n])}x\|$,
i.e., vectors with relatively big two largest components.
For such vectors we produce needed anti-concentration estimates for the matrix-vector products
by using only those two components, i.e.,
we consider anti-concentration for the vector $P_A x$, where $A=\sigma_x(\{1, 2\})$. Since the Rogozin lemma is
not suitable for this case, we compute the anti-concentration directly in Proposition~\ref{anti2}.
As for the classes $\stt_2,\stt_3$, we
actually intersect the fourth class with appropriately chosen sets of almost constant vectors
in order to control cardinalities of the nets.
The final step is to show that the set $\ST$ is contained in
the union of four sets described here. Careful analysis of this approach shows that the result can be proved with
all constants and parameters $r, \delta, \rho$ depending only on $q$. Thus, it works for $p$ being between the two constants
$q$ and $c$.

 The case of small $p$, that is, the case $C(\ln n)/n \leq p \leq c$, requires a more sophisticated splitting of $\ST$ ---
 we split it into  {\it steep vectors} and {\it $\Rv$-vectors}. The definition and the treatment of steep vectors
 essentially  follows \cite{LLTTY first part, LLTTY-TAMS}, with corresponding adjustments for our model. The set
 of {\it steep} vectors consists of vectors having a large jump between order statistics measured at certain indices.
The first subclass of steep
 vectors, $\st_0$, is the same as the class $\stt_1$ described above --- vectors having very large maximal coordinate --- and is
 treated as  $\stt_1$. Similarly to the case of constant $p$, this class is the main contributor
to the bound $(1+o_n(1))n(1-p)^n$ for non-invertibility
over structured vectors. Next we fix certain $m\approx 1/p$ and consider a sequence
$n_0=2$, $n_{j+1}/n_j = \ell_0$, $j\leq s_0-1$, $n_{s_0+1}=m$ for some
specially chosen parameters $\ell_0$ and $s_0$ depending on $p$ and $n$.
The class $\st_1$ will be defined as the class of vectors such that there exists $j$ with
$x^*_{n_{j+1}}> 6pn x^*_{n_j}$. To work with vectors from this class, we first show that for a given
$j$ the event that for every choice  of two disjoint sets $|J_1|=n_j$ and $|J_2|=n_{j+1}-n_j$,
a random \Ber matrix has a row with exactly one $1$ in components indexed by $J_1$ and no $1$'s among components
indexed by $J_2$, holds with a very high probability. Then, conditioned on this event, for every $x\in \st_1$,
we choose $J_1$ corresponding to $x_i^*$, $i\leq n_j$, and $J_2$ corresponding to
$x_i^*$, $n_{j} \leq i\leq n_{j+1}$, and the corresponding row. Then the inner product of this row
with $x$ will be large in absolute value due to the jump (see Lemma~\ref{l:T0} for the details).
Thus, conditioned on the described event, for every $x\in \st_1$ we have a good lower bound on $\|Mx\|$.
Then next two classes of steep vectors, $\st_2$ and $\st_3$,
consist of vectors having a jump of order $C\sqrt{pn}$, namely, vectors in $\st_2$ satisfy
$x_m^* > C\sqrt{pn}  x_k^*$ and  vectors in $\st_3$ satisfy
$x_k^* > C\sqrt{pn}  x_\ell^*$, where $k\approx \sqrt{n/p}$ and $\ell = \lfloor rn\rfloor$ ($r$ is the parameter
from the definition of $\gncvectors_n(r,\gfn,\delta,\rho)$). Trying to apply the same idea for these two subclasses
one sees that the size of corresponding sets $J_1$ and $J_2$ is too large to
have exactly one $1$ among a row's components indexed by $J_1\cup J_2$ with a high probability.
Therefore the proof of individual probability bounds is more delicate and technical as
well as a construction of corresponding nets for $\st_2,\st_3$.  We discuss the details in Subsection~\ref{subs: nets}.

The class of $\Rv$-vectors  consists of non-steep  vectors to which Rogozin's lemma (Proposition~\ref{rog}) can be applied
when we project a vector on $n-k$ smallest coordinates with $m<k\leq n/\ln^2(pn)$, thus vectors from this class
satisfy  $\|P_A x\|\leq c\sqrt{p}\|P_A x\|_\infty$ for $A=\sigma_x([k, n]$ (we will take union over all choices of
integer $k$ in the interval $(m, n/\ln^2(pn)]$). Thus, the individual probability bounds for $\Rv$-vectors
will follow from Rogozin's lemma together with tensorization lemma as for classes $\stt_2$, $\stt_3$, described above.
Thus the remaining part is to construct a good net for $\Rv$-vectors. For simplicity, dealing with such vectors, we fix
the normalization $x^*_{\lfloor rn\rfloor}=1$. Since vectors are non-steep, we have a certain control of largest coordinates and, thus,
on the Euclidean norm of a vector. The upper bound on $k$ is chosen in such a way that the cardinality
on a net corresponding to largest coordinates of a vector is relatively small (it lies in $n/\ln^2(pn)$-dimensional
subspace). For the purpose of constructing of a net of small cardinality, we need to control the Euclidean norm of $P_A x$
for an $\Rv$-vector. Therefore we split $\Rv$-vectors into level sets according to the value of $\|P_A x\|$. There will be
two different types of level sets --- vectors with relatively large Euclidean norm of $P_A x$ and vectors with small $\|P_Ax\|$.
A net for level sets with large $\|P_A x\|$ is easier to construct, since we can zero all coordinates starting with
$x^*_{\lfloor rn\rfloor}=1$. If the Euclidean norm is small, we cannot do this, so we intersect this subclass with almost
constant vectors (in fact we incorporate this intersection into the definition of $\Rv$-vectors), defined by
\begin{equation}
\label{acv}
 \BB(\rho) :=\{ x\in \R^n\, : \, \exists \lambda \in \R \,  \mbox{ s. t. }\,
 |\lam|= x^*_{\lfloor rn\rfloor}  \,  \mbox{ and }\, |\{ i\leq n \, : \,
 |x_i - \lambda|\leq \rho|\lam|\}| >n- \lfloor rn\rfloor\}.
\end{equation}
 As in the
case of constant $p$, this essentially reduces the dimension corresponding to almost constant part to
one and therefore reduce the cardinality of a net. The rather technical construction of nets is presented in
Subsection~\ref{sub-nets}. In some aspects the construction follows ideas developed in \cite{LLTTY first part}.

\section{Preliminaries}
\label{preliminaries}

\subsection{General notation}
\label{gen-not}

By {\it universal} or {\it absolute} constants we always mean numbers independent of all involved parameters,
in particular independent of $p$ and $n$. Given positive integers $\ell<k$ we denote sets
$\{1, 2, \ldots , \ell\}$ and $\{\ell, \ell + 1,  \ldots , k\}$ by $[\ell]$ and $[\ell, k]$ correspondingly.
Having two functions $f$ and $g$ we write $f\approx g$ if there are two absolute positive constants
$c$ and $C$ such that $cf\le g\le Cf$. As usual, $\Pi_n$ denotes the permutation group
on $[n]$.

For every vector $x=(x_i)_{i=1}^n\in \R^n$,
 by $(x_i^*)_{i=1}^n$ we denote the non-increasing rearrangement of the sequence $(|x_i|)_{i=1}^n$
 and we fix one permutation $\sigma_x$ satisfying $|x_{\sigma_x(i)}|= x_i^*$, $i\leq n$.
We use $\la\cdot, \cdot \ra$ for the standard inner product on $\R^n$, that is
$\la x, y \ra = \sum _{i= 1}^{n} x_i  y_i$.
Further, we write $\|x\|_{\infty}=\max_i |x_i|$ for the $\ell_{\infty}$-norm of $x$.
We also denote ${\bf 1}=(1, 1, \dots, 1)$.


\subsection{Lower bound on the singularity probability}\label{subs: lower b}

Here, we provide a simple argument showing that for the sequence of random \Bern matrices $(M_n)$,
with $p_n$ satisfying  $(n p_n-\ln n)\longrightarrow \infty$ as $n\to \infty$, we have
$$
\Prob\big\{\mbox{$M_n$ contains a zero row or column}\big\}\geq (2-o_n(1))n\,(1-p)^n.
$$
Our approach is similar to that applied  in \cite{BasRud-sharp} in the related context.

Fix $n> 1$ and write $p=p_n$.
Let ${\bf 1}_R$ be the indicator of the event that there is zero row in the matrix $M_n$, and, similarly,
let ${\bf 1}_C$ be the indicator of the event that $M_n$ has a zero column.
Then, obviously,
$$
\Exp\,{\bf 1}_R=\Exp\,{\bf 1}_C=1-\big(1-(1-p)^n\big)^n,
$$
hence,
$$
\Exp({\bf 1}_R+{\bf 1}_C)^2\geq 2-2\big(1-(1-p)^n\big)^n.
$$
On the other hand,
$$
\Exp \,{\bf 1}_R\,{\bf 1}_C\leq \sum_{i=1}^n\sum_{j=1}^n\Prob\big\{\mbox{$i$--th row and
$j$--th column of $M_n$ are zero}\big\}= n^2(1-p)^{2n-1},
$$
implying
$$
\Exp({\bf 1}_R+{\bf 1}_C)^2=\Prob\big\{{\bf 1}_R+{\bf 1}_C=1\big\}+4\,\Prob\big\{{\bf 1}_R\,{\bf 1}_C=1\big\}
\leq \Prob\big\{{\bf 1}_R+{\bf 1}_C=1\big\}+4n^2(1-p)^{2n-1}.
$$
Therefore,
\begin{align*}
\Prob\big\{\mbox{$M_n$ contains a zero row or column}\big\}&\geq
\Prob\big\{{\bf 1}_R+{\bf 1}_C=1\big\}\\
&\geq \Exp({\bf 1}_R+{\bf 1}_C)^2-4n^2(1-p)^{2n-1}\\
&\geq 2-2\big(1-(1-p)^n\big)^n-4n^2(1-p)^{2n-1}.
\end{align*}
It remains to note that, with our assumption on the growth rate of $p=p_n$, we have
$n(1-p)^n\longrightarrow 0$, which implies
$$
\frac{1}{n(1-p)^n}\big(2-2\big(1-(1-p)^n\big)^n-4n^2(1-p)^{2n-1}\big)\longrightarrow 2.
$$

\subsection{Gradual  non-constant vectors}
\label{gradnac}

For any $r\in(0,1)$, we define $\normr_n(r)$ as the set of all vectors $x$ in $\R^n$ with $x^*_{\lfloor rn\rfloor}=1$.
We will call these vectors {\it $r$-normalized}.
By a {\it growth function} $\gfn$ we mean any non-decreasing function from $[1,\infty)$ into $[1,\infty)$.

Let $\gfn$ be an arbitrary growth function.
We will say that a vector $x\in \normr_n(r)$ is {\it gradual} (with respect to the function $\gfn$) if $x^*_i\leq \gfn(n/i)$
for all $i\leq n$.
Further, if $x\in \normr_n(r)$ satisfies
\begin{align}
\label{cond2}
\exists\,
Q_1,Q_2\subset[n]  \quad \mbox{ such that } \quad
 |Q_1|,|Q_2|\geq \delta n \qand  \max\limits_{i\in Q_2} x_i\leq \min\limits_{i\in Q_1}x_i- \rho
\end{align}
then we say that the vector $x$ is {\it essentially non-constant} or just {\it non-constant} (with  parameters $\delta,\rho$).
Recall that the set $\gncvectors_n =\gncvectors_n(r,\gfn,\delta,\rho)$ was defined in (\ref{eq: gnc def}) as
$$
\big\{x\in\normr_n(r):\,\, x \, \mbox{ is gradual with $\gfn$ and satisfies \eqref{cond2}}\big\}.
$$
Vectors from this set we call  {\it gradual non-constant vectors}.

\medskip

Recall that the set $\ST=\ST(r,\gfn,\delta,\rho)$ of structured vectors was defined in (\ref{strvect})
as the complement of scalar multiples of $\gncvectors_n(r,\gfn,\delta,\rho)$.
The next simple lemma will allow us to reduce analysis of $\{x/\|x\|:\;x\in\ST\}$ to the treatment of the set
$\{x/\|x\|:\;x\in \normr_n(r)\setminus \gncvectors_n\}$.
\begin{lemma}\label{l:closure}
For any choice of parameters $r,\gfn,\delta,\rho$,
the set $\{x/\|x\|:\;x\in\ST\}$ is contained in the closure of the set $\{x/\|x\|:\;x\in \normr_n(r)\setminus \gncvectors_n\}$.
\end{lemma}
\begin{proof}
Let $y$ be a unit vector such that $y=x/\|x\|$ for some $x\in \ST$.
If $x^*_{\lfloor rn\rfloor}\ne 0$ then $y=z/\|z\|$, where
$z=x/x^*_{\lfloor rn\rfloor}\in \normr_n(r)\setminus \gncvectors_n$.
If $x^*_{\lfloor rn\rfloor}= 0$, we can consider a sequence of vectors $(x(j))_{j\geq 1}$ in $\R^n$
defined by $x(j)_i=x_j$ for $i\ne \lfloor rn\rfloor$ and $x(j)_{\lfloor rn\rfloor}=1/j$. Let
$$
y(j):=x(j)/x(j)^*_{\lfloor rn\rfloor}\in \normr_n(r),\quad j\geq 1.
$$
Clearly, $y(j)^*_1\longrightarrow\infty$, so for all sufficiently large $j$ we have $y(j)\notin \gncvectors_n$.
Thus, for all large $j$,
$$
y(j)/\|y(j)\|\in \{x'/\|x'\|:\;x'\in \normr_n(r)\setminus \gncvectors_n\},
$$
whereas $y(j)/\|y(j)\|=x(j)/\|x(j)\|\longrightarrow x/\|x\|$. This implies the desired result.
\end{proof}


We will need two following lemmas. The first one states that vectors which do not satisfy
(\ref{cond2}) are almost constant (that is, have large part of coordinates nearly equal to each other).
The second one is a simple combinatorial estimate, so we omit
its proof.

\begin{lemma}
\label{a-c-cond2}
Let $n\geq 1$, $\delta,\rho, r\in(0,1)$.
Denote $k=\lceil \delta n\rceil$ and $m=\lfloor rn \rfloor$ and assume
$n\geq 2m> 4k$. Assume $x\in\normr_n(r)$ does not satisfy  (\ref{cond2}).
Then there exist $A\subset [n]$ of cardinality $|A|> n-m$ and $\lambda$ with
$|\lambda|=1$ such that $|x_i-\lam|<\rho$ for every $i\in A$.
\end{lemma}

\begin{proof}
By $(x_i^{\#})_i$ denote the non-increasing  rearrangement
of $(x_i)_i$ (we would like to emphasize that  we do not take absolute values).
Note that there are two subsets $Q_1, Q_2\subset[n]$  with $|Q_1|,|Q_2|\geq k$
satisfying $\max_{i\in Q_2} x_i\leq \min_{i\in Q_1}x_i- \rho$ if and only if
$x_k^{\#}-x_{n-k+1}^{\#}\geq \rho$. Therefore, using that $x$ does not satisfy  (\ref{cond2}),
we observe $x_k^{\#}-x_{n-k+1}^{\#}< \rho$. Next consider the set
$$A:=\{x_i^{\#}\,\, :\,\, k<i\leq n-k\}.$$
Then $|A|=n-2k> n-m$.
Since  $x^*_m=1$ we obtain that
$$
  |\{i\, :\, |x_i|>1\}| <m\leq n-m \qand |\{i\, :\, |x_i|<1\}|\leq n-m.
$$
Therefore, there exists an index $j\in A$ such that $|x_j|=1$. Taking $\lam=x_j$,
we observe that  for every $i\in A$, $|x_i-\lam|<\rho$. This completes the proof.
\end{proof}

\begin{lemma}\label{l: aux 2498276098059385-}
For any $\delta\in(0,1]$ there are $n_\delta\in\N$, $c_\delta>0$ and $C_\delta\geq 1$ depending only on $\delta$ with the following property.
Let $n\geq n_\delta$ and let $m\in\N$ satisfy $n/m\geq C_\delta$.
Denote by $\mathcal S$ the collection of sequences $(S_1,\dots,S_m)\subset[n]$
with $|S_i|=\lfloor n/m\rfloor$ and $S_i\cap S_j=\emptyset\mbox{ for all }i\neq j$.
Let $A_{nm}$ be as in (\ref{anm}).
Then for any pair $Q_1,Q_2$ of disjoint subsets of $[n]$ of cardinality at least $\delta n$ each,
 one has
\begin{align*}
\Big|\Big\{&(S_1,\dots,S_m)\in \mathcal S:\;
\min(|S_i\cap Q_1|,|S_i\cap Q_2|)\geq \frac{\delta}{2}\lfloor n/m\rfloor
\mbox{ for at most $c_\delta m$ indices $i$}
\Big\}\Big|
\leq e^{-c_\delta n} A_{nm}^{-1}.
\end{align*}
\end{lemma}

\subsection{Auxiliary results for Bernoulli r.v. and random matrices }
\label{ben-ineq}

Let $p\in (0,1)$, $\delta$ is Bernoulli random variable taking
value $1$ with probability $p$ and $0$ with probability $1-p$.
We say that $\delta$ is a \Ber random variable.
A random matrix with i.i.d.\ entries distributed as $\delta$ will be
called {\it \Ber random matrix}.

Here we provide four lemmas needed below. We start with
notations for random matrices used throughout the paper.
The class of all $n\times n$ matrices having $0/1$ entries we
denote by $\Mc$.
We will consider a probability measure on $\Mc$ induced by the distribution of an $n\times n$
\Ber random matrix. We will use the same notation $\Prob$ for this probability measure;
the parameter $p$ will always be clear from the context.
Let $M=\{\mu_{ij}\}\in \Mc$. By $\row_i=\row_i(M)$
we denote the $i$-th row of $M$, and by $\col_i(M)$ --- the $i$-th column, $i\leq n$. By $\|M\|$ we always denote
the operator norm of $M$ acting as an operator $\ell_2\to \ell_2$. This
norm is also called  spectral norm and equals the largest singular number.

We will need the following form of Bennett's inequality.

\begin{lemma}
\label{bennett}
Let $n\geq 1$, $0<q<  1$, and $\delta$ be a \Berq random variable.
Let $\delta_i$ and $\delta_{ij}$, $i,j\leq n$, be independent copies of $\delta$.
Define the function $h(u):=(1+u)\ln (1+u) -u$, $u\geq 0$.
Then for every $t>0$,
\begin{align*}
  \max\left(\p\left(\sum_{i=1}^n \delta_i  >qn+t \right),
 \p\left(\sum_{i=1}^n \delta_i  < qn-t \right) \right)
 &\leq
 \exp\left(-\frac{nq(1-q)}{\max^2(q, 1-q)}  \,  h\left(\frac{t\max(q, 1-q) }{nq(1-q)}\right)
 \right).
\end{align*}
In particular, for $0<\eps \leq  q\leq 1/2$,
\begin{align*}
  \max\left(\p\left(\sum_{i=1}^n \delta_i  >(q+ \eps) n \right),
 \p\left(\sum_{i=1}^n \delta_i  <(q- \eps) n \right) \right)
 &\leq
   \exp\left(-\frac{n\eps^2}{2q(1-q)}  \,
 \left(1-\frac{\eps}{3q}\right)\right),
\end{align*}
and for $q\leq 1/2$, $\tau >e$,
\begin{align*}
  \p\left(\sum_{i=1}^n \delta_i  >(\tau +1) qn \right)
 &\leq
   \exp\left(-    \tau \ln(\tau/e)qn \right).
\end{align*}
Furthermore, for $50/n\leq q\leq 0.1$,
\begin{align*}
  \p\Big(qn/8\leq \sum_{i=1}^n \delta_i  \leq 8 qn \Big)
   &\geq 1-   (1-q)^{n/2}.
\end{align*}
Moreover, if $n\geq 30$ and $p=q \geq (4\ln n)/n$ then denoting
$$
 \Event_{sum}  :=\Big\{M=\{\delta_{ij}\}_{i,j\leq n}\in \Mc \, :\,
 \sum _{j=1}^n \delta_{ij} \leq 3.5 pn   \quad
  \mbox{ for every }\,\,\, i\leq n\Big\}
$$
we have $\p(\Event_{sum} )\geq 1- \exp(-1.5 np)$.
\end{lemma}

\begin{proof}
Recall that Bennett's
 inequality states that for
mean zero independent random variables $\xi_1$, \dots, $\xi_n$
satisfying $\xi_i \leq \rho$ (for a certain fixed $\rho>0$)
almost surely for $i\leq n$, one has for every $t>0$,
$$
 \p\left(\sum_{i=1}^n \xi_i  > t \right) \leq \exp\left(-\frac{\sigma^2}{\rho^2}\,
  h\left(\frac{\rho t}{\sigma^2}\right)\right),
$$
where $\sigma^2 = \sum_{i=1}^n\E \xi_i^2$
(see e.g. Theorem 1.2.1 on p. 28 in \cite{lama} or Exercise 2.2
on p. 11 in \cite{DevLug} or Theorem 2.9 in \cite{BLM}).
Take $\xi_i = \delta_i -q$, $\xi_i'=-\xi_i$, $i\leq n$.  Then for every $i\leq n$,
$\xi_i'$ and $\xi_i$ are
centered, $|\xi _i'|= |\xi _i|=\max(q, 1-q)$,  and $\sigma^2= n q(1-q)$. Applying the Bennett inequality
with  $\rho=\max(q, 1-q)$ twice --- to $\xi_i$ and $\xi _i'$,  we observe the first inequality.
To prove the second inequality, we take $t=\eps n$ and use that
$h(\cdot)$ is an increasing function satisfying $h(u)\geq  u^2/2-u^3/6$
on $\R^+$. The third inequality follows by taking $t=\tau qn$ and using
$h(u)\geq u\ln (u/e)$.

For the ``furthermore" part, we apply the third inequality with $\tau=7$, to get
$$
\Prob\Big\{\sum_{i=1}^n \delta_i> 8qn\Big\}\leq \exp(-6qn).
$$
On the other hand, using $q\leq 0.1$,
\begin{align*}
   \Prob\Big\{\sum_{i=1}^n \delta_i< qn/8\Big\}
   &= \sum_{i=0}^{\lfloor qn/8\rfloor}{n\choose i}q^i(1-q)^{n-i}
   \leq (1-q)^n+\sum_{i=1}^{\lfloor qn/8\rfloor}\bigg(\frac{enq}{i(1-q)}\bigg)^i\,(1-q)^n\\
&\leq (1-q)^n+\frac{qn}{8}\bigg(\frac{8e}{1-q}\bigg)^{qn/8}\,(1-q)^n\leq
(1-q)^n+\frac{qn}{8}\bigg(\frac{80e}{9}\bigg)^{qn/8}\,(1-q)^n
\end{align*}
Since $(80 e/9)^{1/8}\leq e^{0.4}$, $(1-q)^n\leq \exp(-qn)$, $qn\geq 50$, and $\ln x \leq x/e$ on $[0, \infty)$,
 this implies
\begin{align*}
  \p\Big(qn/8\leq \sum_{i=1}^n \delta_i  <  qn /8\Big)
 &\leq  \exp(-6qn)+ (1+\exp(0.45 qn)) (1-q)^n \leq (1-q)^{n/2} .
\end{align*}

Finally, to get the last inequality, we take $t=2.5 qn=2.5pn$, then
$$
 \p\left(\sum _{j=1}^n \delta_{ij}   > 3.5 p n\right) \leq
 \exp\left(-\frac{n p}{1-p}\,
  h\left(2.5\right)\right) \leq \exp\left(- n p  \,
 \left(3.5 \ln 3.5 -2.5 \right)\right)\leq \exp\left(- 1.8 n p \r).
$$
Since under our assumptions, $n\exp\left(- 1.8 n p \r)\leq \exp\left(- 1.5 n p \r)$,
the bound on $\p(\Event_{sum} )$  follows by the union bound.
\end{proof}

We need the following simple corollary of the Bennet's lemma.

\begin{lemma}\label{l: column supports}
For any $R\geq 1$ there is $C_{\text{\tiny\ref{l: column supports}}}
=C_{\text{\tiny\ref{l: column supports}}}(R)\geq 1$ with the following property.
Let $n\geq 1$ and $p\in(0,1)$ satisfy $C_{\text{\tiny\ref{l: column supports}}}p\leq 1$
and $C_{\text{\tiny\ref{l: column supports}}}\leq p n$.
Further, let $M$ be an $n\times n$ be \Ber random matrix.
Then with probability at least $1-\exp(-n/C_{\text{\tiny\ref{l: column supports}}})$
one has
$$
 8pn\geq |\supp \col_i(M)|\geq pn/8 \quad \mbox{for all but $\,\,\, \lfloor(pR)^{-1}\rfloor\,\,\,$
 indices $\, i\in [n]\setminus\{1\}$.}
$$
\end{lemma}
\begin{proof}
For each $i\in [n]\setminus\{1\}$, let $\xi_i$ be the indicator of the event
$$\big\{8pn< |\supp \col_i(M)|\quad\quad \mbox{ or }\quad \quad|\supp \col_i(M)|< pn/8\big\}.$$
By Lemma~\ref{bennett},  $\Exp\,\xi_i\leq e^{-pn/2}$. Since $\xi_i$'s are independent, by the Markov inequality,
\begin{align*}
\Prob\Big\{\sum_{i=2}^n \xi_i\geq \frac{1}{p R}\Big\}
\leq {n-1 \choose \lfloor(p R)^{-1}\rfloor}\big(e^{-pn/2}\big)^{\lfloor(p R)^{-1}\rfloor}
\leq {n-1 \choose \lfloor(p R)^{-1}\rfloor}\,e^{- n/(4R)}.
\end{align*}
The result follows.
\end{proof}

The following lemma provides a  bound on the norm of a random Bernoulli
matrix. It is similar to \cite[Theorem~1.14]{BasRud-sharp}, where the case of
symmetric matrices was treated. For the sake of completeness 
we sketch its proof.


\begin{lemma}
\label{mnorm}
Let $n$ be large enough and
$(4\ln n)/n \leq p\leq 1/4$. Let
  $M=(\delta_{ij})_{i,j}$ be a \Ber random matrix.
Then for every $t\geq 30$ one has
\begin{equation*}\label{bdd}
  \p\big\{\|M - \E M\|\geq 2 t \sqrt{n p} \big\}\leq  4 e^{-t^2 pn/4}
   \quad \quad \mbox{ and } \quad \quad
   \p\big\{\|M\|\geq 2 t \sqrt{n p} + pn \big\}\leq  4 e^{-t^2 pn/4}.
\end{equation*}
In particular, taking $t=\sqrt{pn}$,
\begin{equation}\label{normofone}
    \p\left( \|M {\bf 1}\| \geq 3 p n^{3/2} \r)\leq  4 \exp(- n^2 p^2/4).
\end{equation}
\end{lemma}

\begin{proof}
Given an $n\times n$ random matrix $T=(t_{ij})_{i,j}$ with independent
entries taking values in $[0,1]$.
we consider it as a vector in
$\R^m$ with $m={n^2}$. Then the Hilbert--Schmidt norm of $T$ is the standard
Euclidean norm on $\R^m$. Let $f$ be any function in $\R^m$
which is convex and is $1$-Lipschitz with respect to the standard Euclidean norm.
Then the Talagrand inequality (see e.g. Corollary~4.10 and Proposition~1.8
in \cite{Ledoux}) gives that for every $s>0$,
$$
  \p\left(  f(T)  \geq \E  f(T) + s +4\sqrt{\pi} \r) \leq 4\exp (-s^2/4) .
$$
We apply this inequality twice, first with the function $f(T):=\|T\|$ to the matrix
$T:=M-\E M$.  At the end of this proof we show that $\E\|M-\E M\|\leq 20\sqrt{pn}$.
Therefore, taking $s=t\sqrt{pn}$ with $t\geq 30$, we obtain the first bound.
For the second bound, note that all entries of $\E M$ equal $p$,
hence $\|\E M\| = pn$. Thus, the second bound follows by  the triangle inequality.


It remains to prove that $\E\|M-\E M\|\leq 20 \sqrt{pn}$.
Recall that $\delta_{ij}$ are the entries of $M$.
Let $\delta_{ij}'$, $i, j\leq n$ be independent
copies of $\delta_{ij}$ and set $M':= (\delta_{ij}')_{i,j}$.
Denote by $r_{ij}$  independent Rademacher random variables
and by $g_{ij}$ independent standard Gaussian random variables.
We assume that all our variables are mutually independent and set
$\xi_{ij}:= \delta_{ij}-\delta_{ij}'$.
Since for every $i,j\leq n$,
 $\xi_{ij}$ is symmetric, it has the same distribution as
 $|\xi_{ij}| r_{ij}$ and the same as $\sqrt{2/\pi}|\xi_{ij}| r_{ij} \E |g_{ij}|$.
Then we have
$$
   \E_{\delta} \|M-\E M\| = \E_{\delta}\|M-\E_{\delta'} M' \|\leq
    \E_{\delta} \E_{\delta'}\|M- M' \|=
   \E_{\xi} \| (\xi_{ij})_{i,j}\| =
    \sqrt{2/\pi}\, \E_{\xi, r} \| (\xi_{ij}r_{ij} \E_g |g_{ij}|)_{i,j}\|
$$
$$
    \leq\sqrt{2/\pi}\,  \E_{\xi, r, g}  \| (\xi_{ij}r_{ij}  |g_{ij}|)_{i,j}\| =
     \sqrt{2/\pi}\,  \E_{\xi} \E_g \| (\xi_{ij} |g_{ij}|)_{i,j}\| .
$$
Applying a result of Bandeira and Van Handel (see the beginning of Section~3.1 in \cite{BVH}),
we obtain
$$
  \E_{\delta} \|M-\E M\| \leq \E_{\xi} (4 \max(\sigma_1, \sigma_2) + 15\sigma_* \sqrt{\ln (2n)}),
$$
where
$$
  \sigma_1 =\max_{i\leq n} \sqrt{\sum_{j=1}^n\xi_{ij}^2}, \quad
  \sigma_2 =\max_{j\leq n} \sqrt{\sum_{i=1}^n\xi_{ij}^2}, \qand
  \sigma_*=\max_{i, j\leq n} |\xi_{ij}|\leq 1.
$$
Note that $\xi_{ij}^2$ are \Berq\, random variables with $q=2p(1-p)$. Therefore,
using $(4\ln n)/n\leq p\leq 1/2$ and
applying the ``moreover part" of Lemma~\ref{bennett}, we obtain that
$\max(\sigma_1, \sigma_2)> 2\sqrt{pn}$ with probability at most
$2\exp(-1.5 nq)\leq 2/n^6$.
Moreover, since $\xi_{ij}^2\leq 1$, we have $\max(\sigma_1, \sigma_2)\leq \sqrt n$.
Therefore,
$$
  \E_{\xi} (4 \max(\sigma_1, \sigma_2) + 15\sigma_* \sqrt{\ln (2n)})
  \leq 8\sqrt{pn} + 4/n^5 + 15\sqrt{\ln (2n)}\leq 20\sqrt{pn}.
$$
\end{proof}
%
%

As an elementary corollary of the above lemma, we have the following statement where the restriction $pn\geq 4\ln n$ is removed.

\begin{cor}\label{cor: norm of centered}
For every $s>0$ and $R\geq 1$ there is $C_{\text{\tiny\ref{cor: norm of centered}}}\geq 1$
depending on $s,R$ with the following property.
Let $n\geq 16/s$ be large enough and let $p\in(0,1/4]$ satisfy $s\ln n\leq pn$.
Let $M_n$ be an $n\times n$  \Ber random matrix.
Then
$$
\Prob\big\{\|M_n-\E M_n\|\leq C_{\text{\tiny\ref{cor: norm of centered}}}\sqrt{pn}\big\}\geq 1-\exp(-Rpn).
$$
\end{cor}
\begin{proof}
Let $w:=\max(1,\lceil 8/s\rceil)$, $\widetilde n:=w\,n$, and let $\widetilde M_n$ be $\widetilde n\times \widetilde n$
\Ber matrix. Assuming that $n$ is sufficiently large, we get
$$
p\,\widetilde n = wpn \geq s\max(1,\lceil 8/s\rceil) \ln n \geq 4\ln \widetilde n.
$$
Thus, the previous lemma is applicable, and we get
$$
\Prob\big\{\|\widetilde M_n-\Exp \widetilde M_n\|\leq C_{\text{\tiny\ref{cor: norm of centered}}}\sqrt{pn}\big\}\geq 1-\exp(-Rpn),
$$
for some $C_{\text{\tiny\ref{cor: norm of centered}}}>0$ depending only on $s,R$.
Since the norm of a matrix is not less than the norm of any of its submatrices, and because any $n\times n$ submatrix of
$\widetilde M_n$ is equidistributed with $M_n$, we get the result.
\end{proof}

\subsection{Anti-concentration}

In this subsection we combine anti-concentration inequalities with the
following tensorization lemma (see Lemma~3.2 in \cite{KT-last}, Lemma~2.2
in \cite{RV} and Lemma 5.4 in \cite{R-ann}).
We also provide Esseen's lemma.

\begin{lemma}[{Tensorization lemma}]
\label{tens}
\label{l: tensor}
Let $\lam, \gamma>0$.
Let $\xi_1, \xi_2, \ldots, \xi_m$ be independent random variables.
Assume that for all $j\leq m$, $\p(|\xi_j|\leq \lam)\leq \gamma$. Then for every $\eps\in (0,1)$
one has
$$
 \p(\|(\xi_1, \xi_2, ...,\xi_m)\|\leq \lam\sqrt{\eps m} )\leq (e/\eps)^{\eps m} \gamma^{m(1-\eps)}.
$$
Moreover, if there exists $\eps_0>0$ and $K>0$ such that for every  $\eps\geq \eps_0$ and
for all $j\leq m$ one has  $\p(|\xi_j|\leq \eps)\leq K \eps$ then
there exists an absolute constant $C_{\text{\tiny\ref{l: tensor}}}>0$ such that for every  $\eps\geq \eps_0$,
$$
  \p(\|(\xi_1, \xi_2, ...,\xi_m)\|\leq \eps\sqrt{m} )\leq (C_{\text{\tiny\ref{l: tensor}}} K\eps)^{ m}.
$$
\end{lemma}

Recall that for a real-valued random variable $\xi$ its {\it L\'evy concentration function} $\cf(\xi,t)$ is defined as
$$
  \cf(\xi,t):=\sup\limits_{\lambda\in\R}\Prob\bigl\{|\xi-\lambda|\leq t\bigr\},\;\;t>0.
$$
We will need bounds on the L\'evy concentration function of sums of independent random variables.
Such inequalities were investigated in many works, starting with L\'evi, Doeblin, Kolmogorov, Rogozin.
We quote here a result due to  Kesten \cite{kesten}, who improved Rogozin's estimate \cite{Rog}.


\begin{prop}\label{rog}\label{prop: esseen}
Let $\xi_1, \xi_2, \ldots, \xi_m$ be independent random variables and  $\lam, \lam _1, ...,\lam _m>0$
satisfy $\lam \geq \max_{i\leq m} \lam _i$. Then there exists an absolute positive  constant $C$ such that
$$
 \cf\Bigl(\sum_{i=1}^m \xi_i, \lam \Bigr)\leq \frac{C\,  \lam \, \max_{i\leq m} \cf(\xi_i, \lam)
 }{ \sqrt{\sum_{i=1}^m \lam_i^2(1-\cf(\xi_i, \lam _i))}} .
$$
\end{prop}

This proposition together with Lemma~\ref{tens} immediately implies the following consequence,
in which, given $A\subset [m]$ and $x\in \R^m$,
$x_A$ denotes coordinate projection of $x$ on $\R^A$.

\begin{prop}\label{rogozin}
There exists and absolute constant $C_0\geq 1$ such that the following holds.
Let $p\in (0, 1/2)$.  Let $\delta$ be a \Ber random variable. Let
 $\delta_j$, $j\leq n$,
 and $\delta_{ij}$, $i,j\leq n$,
be independent copies of $\delta$. Let  $M=(\delta_{ij})_{ij}$.
 Let $A\subset [n]$ and $x\in \R^n$ be such that $\|x_A\|_\infty \leq   C_0^{-1}\sqrt{p}\, \|x_A\|$.
 Then
$$
   \p\Bigl(\|Mx\| \leq \frac{\sqrt{pn}}{3\sqrt{2} C_0} \|x_A\|  \Bigr)\leq  e^{-3n}.
$$
Moreover, if $\lam := \frac{\sqrt{p}\, \|x_A\|}{ 3 C_0} \leq 1/3$ then
$
   \cf\Bigl(\sum_{j=1}^n \delta_j x_j, \lam \Bigr)\leq  e^{-8}.
$
\end{prop}

\begin{proof} We start with the ``moreover" part. Assume ${\sqrt{p}\, \|x_A\|} \leq C_0$.
Let $\lam _j=|x_j|/3$.  Clearly, for every $j\leq n$, $\cf(x_j \delta _j,|x_j|/3)= \cf(\delta _j, 1/3) = 1-p$.
Proposition~\ref{rog} implies that for every $\lam$ satisfying  $\max_{j\in A} \lam _j\leq \lam \leq 1/3$
one has
$$
  \cf\Bigl(\sum_{j=1}^n x_j\delta_j, \lam \Bigr)\leq
   \cf\Bigl(\sum_{j\in A} x_j\delta_j, \lam \Bigr)\leq
   \frac{C\,  \lam
 }{ \sqrt{\sum_{j\in A} \lam_j^2 \,  p}} =
  \frac{3 C\,  \lam
 }{ \sqrt{ p}\, \|x_A\|}.
$$
Choosing $C_0=C e^8$ and $\lam =   \sqrt{ p}\, \|x_A\|/ (3C_0)$
(note that the assumption on $\|x_A\|_\infty$
ensures that $\lambda\geq \lambda_j$ for all $j\in A$) we obtain the ``moreover" part.

\smallskip

Now apply Lemma~\ref{tens} with $\xi _i= (Mx)_i = \sum_{j=1}^n x_j\delta_{ij}$, $\eps =1/2$,
$\gamma = e^{-8}$, $m=n$. We have
$$
 \p\Bigl(\|Mx\| \leq \lam \sqrt{n/2}
 \Bigr)\leq (2e)^{n/2} \exp(-4 n )\leq \exp(-3 n ).
$$
This implies the bound under assumption  ${\sqrt{p}\, \|x_A\|} \leq C_0$, which can be removed
by normalizing $x$.
\end{proof}

 We also will need the following combination of a simple anti-concentration fact
 with Lemma~\ref{tens}.

\begin{prop}\label{anti2}
Let $p\in (0, 1/20)$ and $\alpha >0$.  Let $\delta$ be a \Ber random variable. Let
 $\delta_j$, $j\leq n$,
 and $\delta_{ij}$, $i,j\leq n$,
be independent copies of $\delta$. Let  $M=(\delta_{ij})_{ij}$.
Let $x\in \R^n$ be such that $x_2^*\geq \alpha$.
Then
$$
   \cf\Bigl(\sum_{j=1}^n x_j \delta_j, \alpha/2.1 \Bigr)\leq  \exp(-1.9 p)
   \quad \qand \quad
   \p\Bigl(\|Mx\| \leq \frac{\alpha \sqrt{pn}}{7\sqrt{\ln(e/p)}}  \Bigr)\leq  \exp( - 1.6np).
$$
\end{prop}

\begin{proof}
Without loss of generality we assume that $x_1^* = |x_1|$ and $x_2^*=|x_2|$. Note that
$x_1\delta_1+x_2 \delta_2$ takes value in $E_1:=\{0, x_1+x_2\}$ with probability $(1-p)^2+p^2\leq 1-1.9p$
and in  $E_2:=\{x_1, x_2\}$ with probability $2p(1-p)\leq 1-1.9p$. Since the distance between $E_1$ and $E_2$
is $\min(|x_1|, |x_2|) =|x_2|$ and since  $\cf\bigl(\sum_{j=1}^n x_j\delta_j, \lam \bigr)\leq
 \cf\bigl(\sum_{j=1}^2 x_j\delta_j, \lam \bigr)$, the first inequality follows.

 Now apply Lemma~\ref{tens} with $\xi _i= (Mx)_i = \sum_{j=1}^n x_j\delta_{ij}$, $\eps =p/(10\ln (e/p))$,
$\gamma = e^{-1.9 p}$, $m=n$. We note that then $\eps\ln (e/\eps) \leq p/4$ and therefore we  have
$$
 \p\Bigl(\|Mx\| \leq \frac{\alpha \sqrt{pn}}{2.1\sqrt{10\ln(e/p)}}
 \Bigr)\leq (e/\eps)^{\eps n} \exp(- 1.9 p n(1-\eps) )
$$
$$
\leq
 \exp(pn/4 - 1.9np(1-\eps)) \leq \exp( - 1.6np),
 $$
 which completes the proof.
\end{proof}

Finally we provide Esseen's lemma  \cite{Esseen-type}, needed to prove Theorem~\ref{p: cf est}.

\begin{lemma}[Esseen] \label{ess}
There exists an absolute constant $C>0$ such that the following holds.
Let $\xi_i$, $i\leq m$ be independent random variables. Then for every $\tau>0$,
\begin{align*}
\cf\Big(\sum\limits_{i=1}^m \xi_i,\tau\Big)
&\leq C\int\limits_{-1}^1 \prod\limits_{i=1}^m|\Exp\exp(2\pi{\bf i}\xi_i s/\tau)|\,ds.
\end{align*}
\end{lemma}

\subsection{Net argument}
\label{net}

Here we discuss special nets that will be used and corresponding approximations.
We fix the following notations. Let $\edv={\bf 1} /\sqrt{n}$ be the unit vector in the direction
of ${\bf 1}$. Let $P_\edv$ be the projection on $\edv^\perp$ and
$P_\edv^\perp$ be the projection on $\edv$, that is $P_\edv^\perp = \la \cdot , \edv\ra \edv$.
Similarly, for $j\leq n$, let $P_j$ be the projection on $e_j^\perp$ and
$P_j^\perp$ be the projection on $e_j$. Recall that for $x\in \R^n$, the permutation $\sigma_x$
satisfies $|x_{\sigma_x(i)}|= x_i^*$, $i\leq n$. Define a (non-linear)
operator $Q:\R^n\to \R^n$ by $Qx = P_{F(x)} x$ --- the coordinate projection on $\R^{F(x)}$,
 where $F(x)=\sigma _x ([2, n])$, in other words $Q$ annihilates the largest coordinate
 of a vector.
Consider the triple norm on $\R^n$ defined by
$$
  ||| x |||^2 := \|P_\edv x\|^2 + p n \|P_\edv^\perp x\|^2
$$
(note that $\|P_\edv^\perp x\| = |\la x , \edv\ra|$).
We will use the following notion of shifted sparse vectors
(recall here that $\sigma_x$ is the permutation responsible for
the non-increasing rearrangement). Given $m\leq n$ and a parameter $\gamma >0$, define
$$
  U(m, \gamma):=\Big\{ x\in \R^n \,\,  :  \, \, \exists A\subset [n], |A|=n- m,\,\, \exists |\lam|\leq \frac{2}{\sqrt m} \, \,
  \forall i \in A
  \,\, \mbox{ one has }
   \,  \,  |x_i- \lam | \leq \frac{\gamma}{\sqrt n} \Big\}.
$$
Further, given another parameter $\beta >0$, define the set
$$
   V(\beta):=\{ x\in \R^n \,  :  \, \|x\|_\infty \leq 1 \, \,
   \mbox{ and }  \, \, \|Qx\|\leq \beta \}.
$$


\begin{lemma}\label{cardnet}
Let $0<8\gamma\leq \eps \leq \beta$ and $1\leq m\leq n$.
Then there exists an $\eps$-net 
in $V(\beta)\cap U(m, \gamma)$
with respect to $|||\cdot|||$ of cardinality at most
$$
   \frac{ 2^{10} \sqrt{p}\, n^2 }{\eps^2\, \sqrt{m}} \left(\frac{9 \beta}{\eps} \r) ^{m} {n\choose m}.
$$
\end{lemma}

\begin{proof}
Denote $V:=V(\beta)\cap U(m, \gamma)$.
For each $x\in V$ let $A(x)$ be a set from the definition of $U(m, \gamma)$
(if the choice of $A(x)$ is not unique, we fix one of them).

Fix $E\subset [n]$ of cardinality $m$.
We first consider vectors   $x\in V$ satisfying $A(x)=E^c$.
Fix $j\leq n$ and denote
$$
   V_j=V_j(E):= \{ x\in V \,  :  \, j=\sigma_x(1)\,\,  \mbox{ and }  \, \, A(x)=E^c\}
$$
(thus $x_1^*=|x_j|$ on $V_j$).
We now construct a net for $V_j$. It will be obtained as the sum of four nets, where the first one
deals with just one coordinate, $j$, ``killing" the maximal coordinate; the second one deals with
non-constant part of the vector, consisting of at most $m$ coordinates (excluding $x_1^*$);
the third one deals with almost constant coordinates (corresponding to $A(x)$); and the fourth net deals with the direction of the constant vector.
This way, three
of our four nets are $1$-dimensional.
Let $P_W$ be the coordinate projection onto $\R^{W}$, where $W=E\setminus\{j\}$.
Note that the definition of $V(\beta)$ implies that $\|P_W(x)\|\leq \beta$ for every $x\in V_j$.
Let, as before, $P_j^\perp$ be the projection onto $e_j$.

Let $\mathcal{N}_1$ be an $\eps/4$-net in
$P_j^\perp (V_j)\subset [-1, 1]e_j$
of cardinality at most $8/\eps$.
Let $\mathcal{N}_2$ be an $\eps/4$-net (with respect to the Euclidean metric)
in $P_F( V_j)$ of cardinality
at most
$
 \left(1+8 \beta/\eps\r)^{m}.
$

Further, let $\mathcal{N}_3$ be an $\eps/(8\sqrt{n})$-net in the segment
$[-2/\sqrt{m}, 2/\sqrt{m}]\sum_{i\in E^c\setminus\{j\}} e_i$
with cardinality at most $16\sqrt{n}/(\eps\sqrt{m})$.
Then by the construction of the nets and by the definition of $U(m, \gamma)$
for every $x\in V_j$ there exist $y^i_x\in \mathcal{N}_i$, $i\leq 3$,
such that for $y_x=y^1_x+y^2_x+y^3_x$,
$$
  \| x - y_x \| ^2\leq \frac{\eps^2}{16}+\frac{\eps^2}{16}+\sum_{i\in E^c\setminus\{j\}}
  \left(\frac{\gamma}{\sqrt{n}} + \frac{\eps}{8\sqrt{n}}\r)^2\leq
  \frac{3\eps^2}{16};
$$
in particular, $\|P_\edv (x-y_x)\|\leq \sqrt{3/16}\eps$.
Finally, let $\mathcal{N}_4$ be an $\eps/(4\sqrt{pn})$-net in the segment
$(\eps/2)[-\edv, \edv]$ with cardinality at most $8\sqrt{pn}$.
Then for every $x\in V_j$ there exists $y_x$ as above and $y_x^4\in \mathcal{N}_4$
with
$$
  ||| x - y_x -y^4_x||| ^2=
   ||| P_\edv (x - y_x) +P_\edv^\perp (x-y_x)  -y^4_x|||^2=
   \| P_\edv (x - y_x)\|^2 +pn \|P_\edv^\perp (x-y_x)  -y^4_x\|^2\leq \eps^2/4.
$$
Thus the set
$
  \mathcal{N}_{E,j} = \mathcal{N}_1 + \mathcal{N}_2 + \mathcal{N}_3+ \mathcal{N}_4
$
is an $(\eps/2)$-net for $V_j$ with respect to $|||\cdot|||$ and its cardinality is bounded by
$$
\frac{2^{10} \sqrt{p}\, n}{\eps^2\sqrt{m}}
 \left(1+\frac{8 \beta}{\eps}\r)^{m}.
$$
Taking union of such nets over all choices of $E\subset [n]$ and all $j\leq n$ we obtain an $(\eps/2)$-net  $\mathcal{N}_0$
in $|||\cdot|||$ for $V$ of desired cardinality.
Using standard argument, we pass to an  $\eps$-net $\mathcal{N} \subset V$ for $V$.
\end{proof}

Later we apply  Lemma~\ref{cardnet} with the following proposition.

\begin{prop}\label{nettri}
Let $n$ be large enough and $(4\ln n)/n\leq p<1/2$, and $\eps >0$.
Denote
$$
  \Event_{nrm}:= \{ M\in \Mc \, : \, \|M - p {\bf 1}{\bf 1}^\top\|\leq 60 \sqrt{n p}  \qand
  \|M {\bf 1}\| \leq 3 p n^{3/2} \}.
$$
Then for every $x\in \R^n$ satisfying $|||x|||\leq \eps$ and every $M\in \Event_{nrm}$ one has
$   \|M x\| \leq  100\sqrt{ p n} \eps .$
\end{prop}

\begin{proof}
Let $w=P_\edv^\perp x$. Then, by the definition of the triple norm,
$\|w\|\leq |||x|||/\sqrt{pn}\leq \eps /\sqrt{pn}$. 
Clearly,
$$
 ( p{\bf 1}{\bf 1}^\top)  (x-w)= ( p{\bf 1}{\bf 1}^\top)P_\edv x = 0.
$$
Therefore, using that $M\in \Event_{nrm}$, we get
$$
  \|  M  (x-w)\| =\|  (M-p{\bf 1}{\bf 1}^\top)  (x-w)\|\leq  60 \sqrt{p n}\|x-w\|\leq 70 \sqrt{p n} \eps.
$$
Since $w={\bf 1} \|w\|/\sqrt{n}$ and $\|w\|\leq \eps /\sqrt{pn}$, using again that
$M\in \Event_{nrm}$,
we observe that
$$
  \|Mw\| \leq \frac{\eps}{\sqrt{p}\, n} \|M {\bf 1} \| \leq
  3 \sqrt{pn}  \eps .
$$
The proposition follows by the triangle inequality.
\end{proof}

\section{Unstructured vectors}\label{s: unstructured}

The goal of this section is to prove Theorem~\ref{th: gradual}.

Recall that given growth function $\gfn$ and parameters $r,\delta,\rho\in (0,1)$,
 the set of vectors $\gncvectors_n=\gncvectors_n(r,\gfn,\delta,\rho)$ was defined in (\ref{eq: gnc def}).
In the next two sections (dealing with invertibility over structured vectors) we   work with two
different growth functions; one will be applied to the case of constant $p$ and the other
one (giving a worse final estimate) is
suitable in the general case. For this reason, and to increase flexibility of our argument,
rather than fixing a specific growth function here, we will work with an arbitrary
non-decreasing function  $\gfn\, : \, [1,\infty)\to [1,\infty)$
satisfying the  additional assumption (\ref{gfncond})
with a ``global'' parameter $K_3\geq 1$.

\subsection{Degree of unstructuredness: definition and basic properties}

Below, for any non-empty finite integer subset $S$, we denote by $\eta[S]$ a random variable uniformly
distributed on $S$.
Additionally, for any $K_2\geq 1$, we fix a smooth version of $\max(\frac{1}{K_2},t)$.
More precisely,  let us fix a function $\psi_{K_2}:\R_+\to\R_+$ satisfying
\begin{itemize}

\item The function $\psi_{K_2}$ is twice continuously differentiable, with
$\|\psi_{K_2}'\|_\infty =1$ and $\|\psi_{K_2}''\|_\infty <\infty$;

\item $\psi_{K_2}(t)=\frac{1}{K_2}$ for all $t\leq \frac{1}{2K_2}$;

\item $\frac{1}{K_2}\geq \psi_{K_2}(t)\geq t$ for all $\frac{1}{K_2}\geq t\geq \frac{1}{2K_2}$;

\item $\psi_{K_2}(t)= t$ for all $t\geq \frac{1}{K_2}$.

\end{itemize}
In what follows, we view the maximum of the second derivative of $\psi_{K_2}$
as a function of $K_2$ (the nature of this function is completely irrelevant as we do not attempt to track
magnitudes of constants involved in our arguments).

\medskip

Fix an integer $n\geq 1$ and an integer $m\leq n/2$. Recall that
given a vector $v\in\R^n$ and parameters $K_1,K_2\geq 1$,  the {\it degree of unstructuredness (u-degree)}
$\bal_n = \bal_n(v,m,K_1,K_2)$ of $v$ was defined in (\ref{udeg}).
The quantity $\bal_n$ will serve as a measure of unstructuredness of the vector $v$
and in its spirit is similar to the notion of the essential least common denominator introduced
earlier by Rudelson and Vershynin \cite{RV}.
Here {\it unstructuredness} refers to the uniformity in the locations of components of $v$
on the real line.
The larger the degree is, the better
 anti-concentration properties of an associated random linear combination are.
The functions $\psi_{K_2}$ employed in the definition will be important when discussing certain
stability properties of $\bal_n$.

We start with a proof of Theorem~\ref{p: cf est} which connects the definition of the u-degree with
anti-concentra\-tion properties.

\begin{proof}[Proof of  Theorem~\ref{p: cf est}]
For any sequence of disjoint subsets $S_1,\dots,S_m$ of $[n]$ of cardinality $\lfloor n/m\rfloor$ each, set
$$
\Event_{S_1,\dots,S_m}:=\big\{\supp X\cap S_i=1\mbox{ for all $i\leq m$}\big\}.
$$
Note that each point $\omega$ of the probability space belongs to the same number of events from the collection
$\{\Event_{S_1,\dots,S_m}\}_{S_1,\dots,S_m}$, therefore, for $A_{nm}$ defined in (\ref{anm}) we have
 for any $\lambda\in\R$ and $\tau>0$,
\begin{equation}\label{eq: aux-92-502}
\begin{split}
\Prob\Big\{&\Big|\sum\limits_{i=1}^n v_i X_i-\lambda\Big|\leq\tau\Big\}=
A_{nm} \,
\sum\limits_{S_1,\dots,S_m}\Prob\Big\{\Big|\sum\limits_{i=1}^n v_i X_i-\lambda\Big|\leq\tau\;\big|\;\Event_{S_1,\dots,S_m}\Big\}.
\end{split}
\end{equation}
Further, conditioned on an event $\Event_{S_1,\dots,S_m}$, the random sum $\sum\limits_{i=1}^n v_i X_i$
is equidistributed with $\sum\limits_{i=1}^m v_{\eta[S_i]}$
(where we assume that $\eta[S_1],\dots,\eta[S_m]$ are jointly independent with $\Event_{S_1,\dots,S_m}$).
On the other hand, applying Lemma~\ref{ess}, we observe that for every $\tau>0$,
\begin{align*}
\cf\Big(\sum\limits_{i=1}^m v_{\eta[S_i]},\tau\Big)
&\leq C'\int\limits_{-1}^1 \prod\limits_{i=1}^m|\Exp\exp(2\pi{\bf i}v_{\eta[S_i]} s/\tau)|\,ds\\
&= C'\,m^{-1/2}\,\tau\int\limits_{-\sqrt{m}/\tau}^{\sqrt{m}/\tau} \prod\limits_{i=1}^m|\Exp\exp(2\pi{\bf i}v_{\eta[S_i]}\,m^{-1/2} s)|\,ds,
\end{align*}
for a universal constant $C'>0$.
Combining this with \eqref{eq: aux-92-502}, we get  for every $\tau>0$,
\begin{align*}
\cf\Big(\sum\limits_{i=1}^n v_i X_i,\tau\Big)
&\leq A_{nm}\,
\sum\limits_{S_1,\dots,S_m}\cf\Big(\sum\limits_{i=1}^n v_i X_i,\tau\;\big\vert\;\Event_{S_1,\dots,S_m}\Big)\\
&\leq \frac{C'\tau
A_{nm}}{\sqrt{m}}
\sum\limits_{S_1,\dots,S_m}
\int\limits_{-\sqrt{m}/\tau}^{\sqrt{m}/\tau} \prod\limits_{i=1}^m|\Exp\exp(2\pi{\bf i}v_{\eta[S_i]} \,m^{-1/2}s)|\,ds.
\end{align*}
Setting $\tau:=\sqrt{m}/\bal_n$, where $\bal_n=\bal_n(v,m,K_1,K_2)$, we obtain
\begin{align*}
\cf&\Big(\sum\limits_{i=1}^n v_i X_i,\sqrt{m}/\bal_n\Big)
\leq \frac{C'
A_{nm}}{\bal_n}\,\sum\limits_{S_1,\dots,S_m}
\int\limits_{-\bal_n}^{\bal_n} \prod\limits_{i=1}^m|\Exp\exp(2\pi{\bf i}v_{\eta[S_i]}\,m^{-1/2} s)|\,ds
\leq \frac{C' K_1}{\bal_n},
\end{align*}
in view of the definition of $\bal_n(v,m,K_1,K_2)$.
The result follows.
\end{proof}

For the future use we state an immediate consequence
of Theorem~\ref{p: cf est} and  Lemma~\ref{l: tensor}.
\begin{cor}\label{cor: cf tensor}
Let $n,\ell\in\N$, let $m_1,\dots,m_\ell$ be integers with $m_i\leq n/2$ for all $i$,
and let $K_1,K_2\geq 1$.
Further, let $v\in\R^n$, and let $B$ be an $\ell\times n$ random matrix with independent
rows such that the $i$-th row is
uniformly distributed on the set of vectors with $m_i$ ones and $n-m_i$ zeros.
Then for any non-random vector $Z\in\R^\ell$ we have
$$
\Prob\big\{\| Bv-Z\|\leq \sqrt{\ell}\, t\big\}\leq
\Big(2C_{\text{\tiny\ref{l: tensor}}}C_{\text{\tiny\ref{p: cf est}}} t/\sqrt{\min\limits_i m_i}\Big)^\ell
\quad \mbox{for all }t\geq \max\limits_i\frac{\sqrt{m_i}}{\bal_n(v,m_i,K_1,K_2)}.
$$
\end{cor}

The parameter $K_2$
which did not participate in any way in the proof of  Theorem~\ref{p: cf est} is needed to guarantee
a certain stability property of $\bal_n(v,m,K_1,K_2)$.
We would like to emphasize that the use of functions $\psi_{K_2}$ is a technical element of the argument.
\begin{prop}[Stability of the u-degree]\label{l: stability of bal}
For any $K_2\geq 1$ there are $c_{\text{\tiny\ref{l: stability of bal}}},c_{\text{\tiny\ref{l: stability of bal}}}'>0$
depending only on $K_2$ with the following property.
Let $K_1\geq 1$, $v\in\R^n$, $k\in\N$, $m\leq n/2$, and assume that $\bal_n(v,m,K_1,K_2)\leq c_{\text{\tiny\ref{l: stability of bal}}}' k$.
Then there is a vector $y\in \big(\frac{1}{k}\Z\big)^n$ such that $\|v-y\|_\infty\leq \frac{1}{k}$, and such that
$$
\bal_n(y,m,c_{\text{\tiny\ref{l: stability of bal}}} K_1,K_2)\leq \bal_n(v,m,K_1,K_2)\leq
 \bal_n(y,m,c_{\text{\tiny\ref{l: stability of bal}}}^{-1} K_1,K_2)
$$
\end{prop}

To prove the proposition we need two auxiliary lemmas.
$u>0$, $\varepsilon\in(0,u/2]$,

\begin{lemma}\label{l: aux 0948523059873205}
Let $0\ne z\in \C$, $\eps \in [0, |z|/2]$ and let $W$ be a random vector in $\C$ with
$\Exp W=0$ and with  $|W|\leq \varepsilon$ everywhere on the probability space.
Then
$$
  \big|\Exp |z+W|-|z|\big|\leq \frac{\varepsilon^2}{|z|}.
$$
\end{lemma}
\begin{proof}
We can view both $z$ and $W$ as vectors in $\R^2$, and can assume without loss of generality that $z=(z_1,0)$, with $z_1=|z|$. Then  $|z_1 + W_1|= z_1 + W_1$ and
$$
   z_1+W_1 \leq |z+W|=\sqrt{(z_1+W_1)^2+W_2^2}
\leq (z_1+W_1)+\frac{W_2^2}{2|z_1+W_1|}\leq  (z_1+W_1)+\frac{\varepsilon^2}{2(|z|-\varepsilon)}.
$$
Hence,
$$
  |z|=z_1  = \Exp (z_1+W_1) \leq \Exp |z+W| \leq \Exp (z_1+W_1) + \frac{\varepsilon^2}{|z|}
  =|z|+\frac{\varepsilon^2}{|z|},
$$
which implies the desired estimate.
\end{proof}

\begin{lemma}\label{l: aux 29853205983475938}
Let $\lam, \mu\in\R$, and let $\xi$ be a random variable in $\R$ with $\Exp\xi=\mu$ and with
$|\xi-\mu|\leq \lam$ everywhere on the probability space.
Then for any $s\in\R$ we have
$$
\big|\Exp\exp\big(2\pi{\bf i}\,\xi\,s\big)-\exp\big(2\pi{\bf i}\,\mu\,s\big)\big|\leq (2\pi \lam  s)^2.
$$
\end{lemma}

\begin{proof} Denote $\xi'=\xi-\mu$. Then $\Exp\xi=0$ and $|\xi'|\leq \lam$. Therefore, using that
$|\sin x|\leq |x|$ and $|\sin x - x |\leq x^2/2$ for every $x\in \R$, we obtain
\begin{align*}
 \big|\Exp\exp&\big(2\pi{\bf i}\,\xi\,s\big)-\exp\big(2\pi{\bf i}\,\mu\,s\big)\big|=
 \big|\Exp\exp\big(2\pi{\bf i}\,\xi' s\big)-1\big|=
 \big|\Exp \cos\big(2\pi\xi's\big)-1 + {\bf i}\, \Exp \sin\big(2\pi\xi' s\big)\big|
 \\ &=
 \big|-2 \Exp \sin^2 \big(\pi\xi's\big) + {\bf i}\, \Exp \big(\sin\big(2\pi\xi's\big)- 2\pi\xi's\big)\big|
 \leq 2 (\pi\lam s)^2 + (2\pi\lam s)^2/2 = (2\pi\lam s)^2.
\end{align*}

\end{proof}

\begin{proof}[Proof of Proposition~\ref{l: stability of bal}]
To prove the proposition, we will use the {\it randomized rounding} which is a well known
notion in computer science, and was recently applied in the random matrix context in
\cite{Livshyts} (see also \cite{KT-last,LTV}).
Define a random vector $Y$ in $\big(\frac{1}{k}\Z\big)^n$ with independent components $Y_1,\dots,Y_n$
such that each component $Y_i$ has distribution
$$
Y_i=\begin{cases}
\frac{1}{k}\lfloor k v_i\rfloor,&\mbox{ with probability $\lfloor k v_i\rfloor-kv_i+1$},\\
\frac{1}{k}\lfloor k v_i\rfloor+\frac{1}{k},&\mbox{ with probability $kv_i-\lfloor k v_i\rfloor$}.
\end{cases}
$$
Then $\Exp Y_i=v_i$, $i\leq n$ and,  deterministically, $\|v-Y\|_\infty\leq1/k$.

Fix for a moment a number $s\in(0,k/(14\pi K_2)]$ and a subset $S\subset[n]$ of cardinality $\lfloor n/m\rfloor$.
Our intermediate goal is to estimate the quantity
$$
\Exp\,\psi_{K_2}\Big(\Big|\frac{1}{\lfloor n/m\rfloor}\sum_{j\in S}\exp\big(2\pi{\bf i}\,Y_{j} \,s\big)\Big|
\Big).
$$
Denote
$$
  V=V_S:=\left|\frac{1}{\lfloor n/m\rfloor}\sum_{j\in S}\exp\big(2\pi{\bf i}\,v_j \,s\big)\r|
  = \big|\Exp\exp\big(2\pi{\bf i}\,v_{\eta[S]} \,s\big)\big|
$$
and consider two cases.

\smallskip

\noindent
{\it Case 1. }
 $V\leq \frac{1}{2K_2}-\frac{2\pi\,s}{k}$.
Using that $|e^{{\bf i} x}-1|\leq |x|$ for every $x\in \R$, we observe that deterministically
\begin{equation}\label{ineqktwo}
  |\exp\big(2\pi{\bf i}\,v_j \,s\big)-\exp\big(2\pi{\bf i}\,Y_j \,s\big)|\leq 2\pi s/k.
\end{equation}
 Therefore, by the definition of the function $\psi_{K_2}$, in this case we have on the entire probability space
$$
\psi_{K_2}\Big(\Big|\frac{1}{\lfloor n/m\rfloor}\sum_{j\in S}\exp\big(2\pi{\bf i}\,Y_{j} \,s\big)\Big|
\Big)=\psi_{K_2}(V)=\frac{1}{K_2}.
$$

\smallskip

\noindent
{\it Case 2. }
$V> \frac{1}{2K_2}-\frac{2\pi\,s}{k} \geq \frac{1}{4K_2}$.
Set
$$
z:=\frac{1}{\lfloor n/m\rfloor}\Exp\,\sum_{j\in S}\exp\big(2\pi{\bf i}\,Y_{j} \,s\big) \qand
W:=\frac{1}{\lfloor n/m\rfloor}\sum_{j\in S}\exp\big(2\pi{\bf i}\,Y_{j} \,s\big)-z.
$$
Then $\Exp W=0$ and, using again $|e^{{\bf i} x}-1|\leq |x|$, we see that $|W|\leq 2\pi s/k$ everywhere.
By Lemma~\ref{l: aux 29853205983475938},
$|z-V|\leq (2\pi s /k)^2$, in particular, $z\geq V-(2\pi s /k)^2\geq 1/(3K_2) \geq 4\pi s /k \geq |W|/2$.
Therefore we may apply Lemma~\ref{l: aux 0948523059873205}, to obtain
$$
 \big|\Exp|W+z|-|z|\big|\leq \frac{4\pi^2 s^2}{|z|k^2}\leq \frac{12\pi^2 K_2 s^2}{k^2}.
$$
This implies,
\begin{align}
\Big|\Exp\Big|&\frac{1}{\lfloor n/m\rfloor}\sum_{j\in S}\exp\big(2\pi{\bf i}\,Y_{j} \,s\big)\Big|
-V\Big|
=
\Big|\Exp|W+z|-|z|+|z| -V\Big|
\leq \frac{16\pi^2 K_2 s^2}{k^2}.\label{eq: aux 5987325938}
\end{align}
To convert the last relation to estimating $\psi_{K_2}(\cdot)$, we will use the
assumption that the second derivative of $\psi_{K_2}$ is uniformly bounded.
Applying Taylor's expansion around the point $V$,
we get
\begin{align*}
\Exp\,\psi_{K_2}\Big(\Big|\frac{1}{\lfloor n/m\rfloor}\sum_{j\in S}\exp\big(2\pi{\bf i}\,Y_{j} \,s\big)\Big|
\Big)
=
\psi_{K_2}\big(V\big) &+\Exp\Big(\Big|\frac{1}{\lfloor n/m\rfloor}\sum_{j\in S}\exp\big(2\pi{\bf i}\,Y_{j} \,s\big)\Big|-V\Big)\,\psi_{K_2}'(V)
\\&+C''\Big\|\,\,\Big|\frac{1}{\lfloor n/m\rfloor}\sum_{j\in S}\exp\big(2\pi{\bf i}\,Y_{j} \,s\big)\Big|-V\Big\|_\infty^2,
\end{align*}
for some $C''>0$ which may only depend on $K_2$.
Here, $\|\cdot\|_\infty$ denotes the essential supremum of the random variable,
and is bounded above by $2\pi  s/k$ by \ref{ineqktwo}.
Together with \eqref{eq: aux 5987325938} and with $\|\psi_{K_2}'\|_\infty\leq 1$, this gives
$$
\Big|\Exp\,\psi_{K_2}\Big(\Big|\frac{1}{\lfloor n/m\rfloor}\sum_{j\in S}\exp\big(2\pi{\bf i}\,Y_{j} \,s\big)\Big|\Big)-\psi_{K_2}(V)
\Big|\leq \frac{\bar C\,s^2}{k^2},
$$
where $\bar C$ depends only on $K_2$.

Since $\psi_{K_2}'\geq 1/(2K_2)$, in both cases we obtain for some $\hat C>0$ depending only  on $K_2$,
\begin{align*}
\Big|&\Exp\,\psi_{K_2}\Big(\Big|\frac{1}{\lfloor n/m\rfloor}\sum_{j\in S}\exp\big(2\pi{\bf i}\,Y_{j} \,s\big)\Big|\Big)
-\psi_{K_2}(V)\Big| \leq \frac{\hat C\,s^2}{k^2}\psi_{K_2}(V).
\end{align*}

 Using this inequality together with definition of $V=V_S$, integrating over $s$, and summing over all choices of disjoint subsets $S_1,\dots,S_m$ of cardinality $\lfloor n/m\rfloor$, for every $t\in(0,k/(14\pi K_2)]$ we get the relation
\begin{align*}
\sum\limits_{S_1,\dots,S_m}\;
&\int\limits_{-t}^t \max\bigg(0,1-\frac{c_0\,s^2}{k^2}\bigg)^m\prod\limits_{i=1}^{m}\psi_{K_2}\big(
\big|\Exp\exp\big(2\pi{\bf i}\,v_{\eta[S_i]} \,s\big)\big|\big)\,ds\\
&\leq
\sum\limits_{S_1,\dots,S_m}\;
\int\limits_{-t}^t \prod\limits_{i=1}^{m}
\Exp_Y\,\psi_{K_2}\Big(\Big|\frac{1}{\lfloor n/m\rfloor}\sum_{j\in S_i}\exp\big(2\pi{\bf i}\,Y_{j} \,s\big)\Big|\Big)\,ds\\
&\leq
\sum\limits_{S_1,\dots,S_m}\;
\int\limits_{-t}^t \bigg(1+\frac{C_0\,s^2}{k^2}\bigg)^m\prod\limits_{i=1}^{m}\psi_{K_2}\big(
\big|\Exp\exp\big(2\pi{\bf i}\,v_{\eta[S_i]} \,s\big)\big|\big)\,ds,
\end{align*}
 where $C_0,c_0>7\pi K_2$ are constants that may only depend on $K_2$.
Using independence of the components of $Y$, we can take the expectation with respect to $Y$ out of the integral.

Given a vector $Q=(q_1,\dots,q_n)\in\R^n$ and $t\in(0,k/(14\pi K_2)]$, denote
$$
g_t(Q):=
\sum\limits_{S_1,\dots,S_m}\;
\int\limits_{-t}^t \prod\limits_{i=1}^{m}
\,\psi_{K_2}\Big(\Big|\frac{1}{\lfloor n/m\rfloor}\sum_{j\in S_i}\exp\big(2\pi{\bf i}\,q_{j} \,s\big)\Big|\Big)\,ds.
$$
The above relation implies that there are
two (non-random) realizations $Y'$ and $Y''$ of $Y$ such that for
\begin{align*}
g_t(Y') &\geq  I_1:= \max\bigg(0,1-\frac{c_0\,t^2}{k^2}\bigg)^m\sum\limits_{S_1,\dots,S_m}\;
\int\limits_{-t}^t\prod\limits_{i=1}^{m}\psi_{K_2}\big(
\big|\Exp\exp\big(2\pi{\bf i}\,v_{\eta[S_i]} \,s\big)\big|\big)\,ds
\end{align*}
and
\begin{align*}
g_t(Y'')\leq I_2:=\bigg(1+\frac{C_0\,t^2}{k^2}\bigg)^m\sum\limits_{S_1,\dots,S_m}\;
\int\limits_{-t}^t\prod\limits_{i=1}^{m}\psi_{K_2}\big(
\big|\Exp\exp\big(2\pi{\bf i}\,v_{\eta[S_i]} \,s\big)\big|\big)\,ds.
\end{align*}
Using properties of the function $\psi_{K_2}$, we note
that for any two non-random vectors $\widetilde Y$ and $\hat Y$ in the range of $Y$
such that they differ on a single coordinate, one has
$
g_t(\widetilde Y)\leq 4K_2\, g_t(\hat Y).
$
Consider a path $Y^{(1)}=Y',Y^{(2)},Y^{(3)},\dots,Y''$
from $Y'$ to $Y''$ consisting of a sequence of non-random vectors in
the range of $Y$ such that each adjacent pair $Y^{(i)},Y^{(i+1)}$ differs on a single coordinate and
let
$$S := \{i\, : \, g_t(Y^{(i)})> 4K_2  I_2\}\subset [1, n-1].$$
 If
$S=\emptyset$, take ${\bf Y}=Y^{(1)}$. Otherwise, let
$\ell = \max\{i\, : \, g_t(Y^{(i)})> 4K_2  I_2\}$. Then take ${\bf Y}=Y^{(\ell+1)}$ and note
$g_t(Y^{(\ell+1)}) \geq g_t(Y^{(\ell)})/(4K_2)\geq I_2\geq I_1$.
Thus the vector ${\bf Y}$ is in the range of $Y$ and
\begin{align*}
     I_1\leq    g_t({\bf Y})&\leq 4K_2 I_2.
\end{align*}
Making substitutions $s'=\sqrt{m} s$, $t'=\sqrt{m} t$ in the integrals in $I_1, I_2$, and assuming that
$t'\leq k/\max(2C_0,2c_0)$ (in this case the condition $t\leq k/(14\pi K_2)$ is satisfied),
we can rewrite the last inequalities as
\begin{multline*}
\frac{1}{2}\sum\limits_{S_1,\dots,S_m}\;
\int\limits_{-t'}^{t'} \prod\limits_{i=1}^{m}\psi_{K_2}\big(
\big|\Exp\exp\big(2\pi{\bf i}\,v_{\eta[S_i]} m^{-1/2}\,s\big)\big|\big)\,ds\\
\leq
\sum\limits_{S_1,\dots,S_m}\;
\int\limits_{-t'}^{t'} \prod\limits_{i=1}^{m}
\,\psi_{K_2}\Big(\Big|\frac{1}{\lfloor n/m\rfloor}\sum_{j\in S_i}\exp\big(2\pi{\bf i}\,{\bf Y}_{j} m^{-1/2}\,s\big)\Big|\Big)\,ds\\
\leq 6K_2 \,\sum\limits_{S_1,\dots,S_m}\;
\int\limits_{-t'}^{t'} \prod\limits_{i=1}^{m}\psi_{K_2}\big(
\big|\Exp\exp\big(2\pi{\bf i}\,v_{\eta[S_i]} m^{-1/2}\,s\big)\big|\big)\,ds.
\end{multline*}
The result follows by the definition of $\bal_n(\cdot)$.
\end{proof}

The last statement to be considered in this subsection asserts that the u-degree of  any
vector from $\gncvectors_n(r,\gfn,\delta,\rho)$ is at least of order $\sqrt{m}$.
\begin{prop}[Lower bound on the u-degree]\label{p: low bound on bal}
For any $r,\delta,\rho$ there is
$C_{\text{\tiny\ref{p: low bound on bal}}}>0$ 
depending only on $r,\delta,\rho$ with the following property.
Let $K_2\geq 2$, $1\leq m\leq n/C_{\text{\tiny\ref{p: low bound on bal}}}$,
$K_1\geq C_{\text{\tiny\ref{p: low bound on bal}}}$
and let $x\in \gncvectors_n(r,\gfn,\delta,\rho)$.
Then $$\bal_n(x,m,K_1,K_2)\geq \sqrt{m}.$$
\end{prop}

\begin{lemma}\label{l: bal lower aux}
For any $\rho>0$ and $\kappa\in (0, 1/2]$ there is a constant $\widetilde C>0$ depending only on $\rho$ and $\kappa$
with the following property. Let $S\neq \emptyset$ be a finite subset of $\Z$, and let
$(y_w)_{w\in S}$ be a real vector (indexed by $S$).
 Assume further that $S_1,S_2$ are two disjoint subsets of $S$, each of cardinality
at least $\kappa|S|$ such that $\min\limits_{w\in S_1}y_w\geq \max\limits_{w\in S_2}y_w+\rho$. Let $K_2\geq 2$ and
$f$ be a function on $[0,1]$ defined by
$$
 f(t):=\psi_{K_2}\Big(\Big|\frac{1}{|S|}\sum_{w\in S}\exp(2\pi{\bf i}\,y_w\,t)\Big|\Big),\quad t\in[0,1].
$$
Then for every $b>0$ one has
$$
  \big|\big\{t\in[0,1]:\;f(t)\geq 1-b^2\big\}\big|\leq \widetilde C b.
$$
\end{lemma}
\begin{proof}
Clearly we may assume that $b\leq 1/\sqrt{2}$.
Denote $m= \lceil \kappa|S|\rceil$ and
$$
g(t):=\Big|\sum_{w\in S}\exp(2\pi{\bf i}\,y_w\,t)\Big|,\quad t\in\R.
$$
Let $T\subset S_1\times S_2$ be of cardinality $T=m$ and such that for all $(q,j),(q',j')\in T$ with
$(q,j)\neq (q',j')$ one has $q\neq q'$ and $j\neq j'$.
Then for all $t\in\R$,
$$
 g(t)= \Big|\sum_{w\in S_1\cup S_2}\exp(2\pi{\bf i}\,y_w\,t)+\sum_{w\notin S_1\cup S_2}\exp(2\pi{\bf i}\,y_w\,t)
 \Big| \leq \sum_{(q,j)\in T}\big|1+\exp(2\pi{\bf i}\,(y_j-y_q)\,t)\big|+|S|-2m.
$$
Further, take any $u\in (0, 1/\sqrt{2\kappa})$ and observe that for each $(q,j)\in T$, since $|y_j-y_q|\geq \rho$,
we have
$$
\big|\big\{t\in[0,1]:\;\big|1+\exp(2\pi{\bf i}\,(y_j-y_q)\,t)\big|\geq 2-2u^2\big\}\big|\leq C'u,
$$
where $C'>0$ may only depend on $\rho$.
This implies that
$$
\Big|\Big\{t\in[0,1]:\;\big|1+\exp(2\pi{\bf i}\,(y_j-y_q)\,t)\big|\geq 2-2u^2\mbox{ for at least $m/2$
pairs $(q,j)\in T$}\Big\}\Big|\leq 2C'u.
$$
On the other hand, whenever $t\in[0,1]$ is such that $\big|1+\exp(2\pi{\bf i}\,(y_j-y_q)\,t)\big|\geq 2-2u^2$
for at most $m/2$ pairs $(q,j)\in T$, we have
$$
g(t)\leq \frac{m}{2}(2-2u^2)+\frac{m}{2}\cdot 2+|S|-2m
= |S|-m u^2 \leq |S|(1-\kappa u^2),
$$
whence $f(t)\leq \max\big(\frac{1}{K_2}, 1-\kappa u^2\big) = 1-\kappa u^2$.
Taking $u=\frac{b}{\sqrt{\kappa}}$ we obtain the
desired result with  $\widetilde C= \frac{2C'}{\sqrt{\kappa}}$.
\end{proof}

\begin{proof}[Proof of Proposition~\ref{p: low bound on bal}]
Let $A_{nm}$ be defined as in (\ref{anm}) and $n_\delta$, $C_\delta$, $\mathcal S$ be from
Lemma~\ref{l: aux 2498276098059385-}. We assume that $n\geq n_\delta$
and $n/m\geq C_\delta$. For every $i\leq m$ denote
$$
  f_i (s)  = \psi_{K_2}\big(
\big|\Exp\exp\big(2\pi{\bf i}\,x_{\eta[S_i]}\,m^{-1/2} s\big)\big|\big).
$$
Further, let subsets $Q_1$ and $Q_2$ be taken from the definition of non-constant vectors applied to $x$.
Then by Lemma~\ref{l: aux 2498276098059385-} and since $\psi_{K_2}(1)\leq 1$,
\begin{align*}
&A_{nm}\,
\sum\limits_{(S_1,\dots,S_m)\in\mathcal S}\;
\int\limits_{-\sqrt{m}}^{\sqrt{m}} \prod\limits_{i=1}^{m}f_i\,ds
\leq \,e^{-c_\delta n}\,2\sqrt{m}
+
A_{nm}\,
\sum\limits_{(S_1,\dots,S_m)\in\mathcal S'}\;
\int\limits_{-\sqrt{m}}^{\sqrt{m}} \prod\limits_{i=1}^{m}f_i\,ds,
\end{align*}
where $\mathcal S'$ is set of all sequences $(S_1,\dots,S_m)\in\mathcal S$ such that
is the subset of $S$ such that
\begin{equation}\label{manyind}
\min(|S_i\cap Q_1|,|S_i\cap Q_2|)\geq \frac{\delta}{2}\lfloor n/m\rfloor\,\,
\mbox{ for at least $\,\, c_\delta m\,\,$ indices $\,i$.}
\end{equation}
Take any $(S_1,\dots,S_m)\in\mathcal S'$ and denote $m_0:=\lceil c_\delta m\rceil$.
 Without loss of generality
we  assume that (\ref{manyind}) holds for all $i\leq m_0$.
Applying Lemma~\ref{l: bal lower aux}
with $\kappa:=\delta/2$ and $b=\sqrt{1-u}$, we get for all $u\in (0,1]$ and $i\leq m_0$,
$$
 \mu (u):= \Big|\Big\{s\in[-\sqrt{m},\sqrt{m}]:\;
  f_i \geq u\Big\}\Big|\leq \widetilde C\sqrt{m}\sqrt{1-u},
$$
where $\widetilde C>0$ depends only on $\delta$ and $\rho$.
This estimate implies that for $i\leq m_0$,
$$
\int\limits_{-\sqrt{m}}^{\sqrt{m}}
(f_i(s))^{m_0}\,ds  = \int\limits_{0}^{1}  m_0\, u^{m_0-1}\, \mu_u \,ds  \leq \widetilde C\sqrt{m}\, m_0\,
B(3/2, m_0 )\leq C_2,
$$
where $B$ denotes the Beta-function and  $C_2>0$ is a constant depending  only on $\rho$ and $\delta$.
Applying H\"older's inequality, we  obtain
\begin{align*}
\int\limits_{-\sqrt{m}}^{\sqrt{m}} \prod\limits_{i=1}^{m}\psi_{K_2}\big(
\big|\Exp\exp\big(2\pi{\bf i}\,x_{\eta[S_i]}\,m^{-1/2} s\big)\big|\big)\,ds
&\leq \int\limits_{-\sqrt{m}}^{\sqrt{m}} \prod\limits_{i=1}^{ m_0}\psi_{K_2}\big(
\big|\Exp\exp\big(2\pi{\bf i}\,x_{\eta[S_i]}\,m^{-1/2} s\big)\big|\big)\,ds
\leq C_2,
\end{align*}
which impies the desired result.
\end{proof}

\subsection{No moderately unstructured normal vectors}\label{s: no mod}

Let $M_n$ be an $n\times n$ \Ber random matrix..
For each $i\leq n$, denote by $H_i=H_i(M_n)$ the span of columns $\col_j(M_n)$, $j\neq i$.
The goal of this subsection is to prove Theorem~\ref{th: gradual}, which asserts that
under appropriate restrictions on $n$ and $p$
with a very large probability (say, at least $1-2e^{-2pn}$), the subspace $H_i^\perp$ is either
structured or very unstructured.
%
%
%
%
%
The main ingredient of the proof --- Proposition~\ref{prop: 09582593852} --- will be considered in the next subsection.
Here, we will only state the proposition to be used as a blackbox
and for this we need to introduce an additional product structure, which, in  a sense, replaces
the set $\gncvectors_n(r,\gfn,\delta,\rho)$.

\bigskip

Fix a permutation $\sigma\in\Pi_n$, two disjoint subsets $Q_1,Q_2$ of cardinality $\lceil\delta n\rceil$ each,
and a number $h\in\R$ such that
\begin{equation}\label{eq: h admissible}
\forall i\in Q_1: \, \, h+2 \leq \gfn(n/\sigma^{-1}(i)) \quad \qand \quad
\forall i\in Q_2: \, \, -\gfn(n/\sigma^{-1}(i))\leq h-\rho-2.
\end{equation}
Define the sets $\Lambda_n=\Lambda_n(k,\gfn,Q_1,Q_2,\rho,\sigma,h)$ by
\begin{equation}\label{eq: param l def}
\begin{split}
\Lambda_n :=\bigg\{x\in\frac{1}{k}\Z^n:\;&|x_{\sigma(i)}|\leq \gfn(n/i)
\;\;\mbox{for all }i\leq n,\;
\;\;\min\limits_{i\in Q_1}x_i\geq h,\,\,\, \mbox{ and }\,\,\,\max\limits_{i\in Q_2} x_i\leq h-\rho\bigg\}.
\end{split}
\end{equation}
In what follows, we adopt the convention that $\Lambda_n=\emptyset$
whenever $h$ does not satisfy \eqref{eq: h admissible}.

\smallskip

\begin{lemma}\label{l: permut}
There exists an absolute constant $C_{\text{\tiny\ref{l: permut}}}\geq 1$ such that for every
 $n\geq 1$ there is a subset $\bar \Pi_n\subset\Pi_n$ of cardinality at most
 $\exp({C_{\text{\tiny\ref{l: permut}}}n})$ with the following property.
For any two partitions $(S_i)_{i=1}^m$ and $(S_i')_{i=1}^m$ of $[n]$ with $2^{-i+1}n\geq |S_i|=|S_i'|$, $i\leq m$,
there is $\sigma\in \bar \Pi_n$ such that $\sigma(S_i)=S_i'$, $i\leq m$.
\end{lemma}

This lemma immediately follows from the fact that the total number of partitions $(S_i)_{i=1}^m$ of $[n]$
satisfying $2^{-i+1}n\geq |S_i|$, $i\leq m$, is exponential in $n$ (one can take $C_{\text{\tiny\ref{l: permut}}}=23$).
Using Lemma~\ref{l: permut}, we provide an efficient approximation of $\gncvectors_n(r,\gfn,\delta,\rho)$.

\begin{lemma}\label{l: disc with l}
For any $x\in\gncvectors_n=\gncvectors_n(r,\gfn,\delta,\rho)$, $k\geq 4/\rho$,
and any $y\in\frac{1}{k}\Z^n$ with $\|x-y\|_\infty\leq 1/k$ one has
$$
y \in
\bigcup\limits_{q=\lfloor-4\gfn(6 n)/\rho\rfloor}^{
\lceil4\gfn(6 n)/\rho \rceil}\;
\bigcup\limits_{\bar\sigma\in\bar\Pi_n}\bigcup\limits_{|Q_1|,|Q_2|=\lceil\delta n\rceil}
\Lambda_n(k,\gfn(6\, \cdot),Q_1,Q_2,\rho/4,\bar\sigma,\rho q/4),
$$
where the set of permutations $\bar\Pi_n$ is taken from Lemma~\ref{l: permut}.
\end{lemma}

\begin{proof}
Let $x\in \gncvectors_n$, and assume that $y\in\frac{1}{k}\Z^n$ satisfies
$\|x-y\|_\infty\leq 1/k$.
Then, by the definition of $\gncvectors_n$, there exist sets $Q_1,Q_2\subset[n]$, each of cardinality
$\lceil \delta n\rceil$, satisfying
$$\max\limits_{i\in Q_2}y_i-\frac{1}{k}\leq \max\limits_{i\in Q_2}x_i
\leq \min\limits_{i\in Q_1}x_i-\rho\leq \min\limits_{i\in Q_1}y_i-\rho+\frac{1}{k}.$$
Then
$
\max\limits_{i\in Q_2}y_i\leq \min\limits_{i\in Q_1}y_i-\frac{\rho}{2},
$
hence we can find a number $h\in \frac{\rho}{4}\Z$ such that
$$\min\limits_{i\in Q_1}y_i\geq h \quad \qand \quad \max\limits_{i\in Q_2}y_i\leq h-\frac{\rho}{4}.$$
By the definition of $\gncvectors_n$ we also have $|x_{\sigma_x(i)}|\leq \gfn(n/i)$ for all $i\in[n]$.
By the definition of $\bar\Pi_n$, we can find a permutation $\bar\sigma\in \bar\Pi_n$ such that
$$\sigma_x\big(\{\lfloor n/2^{\ell}\rfloor+1,\dots,\lfloor n/2^{\ell-1}\rfloor\}\big)
=\bar\sigma\big(\{\lfloor n/2^{\ell}\rfloor+1,\dots,\lfloor n/2^{\ell-1}\rfloor\}\big)\quad\mbox{ for all }\ell\geq 1.$$
Clearly for such a permutation we have
$|x_{\bar\sigma(i)}|\leq \gfn(2 n/i)$ for every $i\leq n$.
Using (\ref{gfncond}), we obtain
$$
|y_{\bar\sigma(i)}|\leq |x_{\bar\sigma(i)}|+\frac{1}{k}\leq
\gfn(2n/i)+\frac{1}{k}\leq \gfn(6n/i)-2.
$$
Thus
$$
  \forall i\in\bar\sigma^{-1}(Q_1): \, \, h\leq \min\limits_{i\in Q_1}y_i \leq \gfn(6n/i)-2\qand
  \forall i\in\bar\sigma^{-1}(Q_2): \, \, h-\frac{\rho}{4}\geq \max\limits_{i\in Q_2}y_i\geq 2-\gfn(6n/i).
$$
Since $h=\rho q/4$ for some $q\in \Z$, this implies the desired result.
\end{proof}

The following statment, together with Theorem~\ref{p: cf est} and Proposition~\ref{l: stability of bal},
is the main ingredient of the proof of Theorem~\ref{th: gradual}.

\begin{prop}\label{prop: 09582593852}
Let $\varepsilon\in(0,1/8]$, $\rho,\delta\in(0,1/4]$ and let the growth function $\gfn$
satisfies \eqref{gfncond}. There exist
$K_{\text{\tiny\ref{prop: 09582593852}}}=K_{\text{\tiny\ref{prop: 09582593852}}}(\delta,\rho)\geq 1$,
$n_{\text{\tiny\ref{prop: 09582593852}}}=n_{\text{\tiny\ref{prop: 09582593852}}}(\varepsilon,\delta,\rho,K_3)$, and
$C_{\text{\tiny\ref{prop: 09582593852}}}=C_{\text{\tiny\ref{prop: 09582593852}}}(\varepsilon,\delta,\rho,K_3)\in\N$
with the following property.
Let $\sigma\in\Pi_n$, $h\in\R$, and let $Q_1,Q_2\subset[n]$ be
such that $|Q_1|,|Q_2|=\lceil \delta n\rceil$.
Let $8\leq K_2\leq 1/\varepsilon$,
 $n\geq n_{\text{\tiny\ref{prop: 09582593852}}}$, $m\geq C_{\text{\tiny\ref{prop: 09582593852}}}$ with
$n/m\geq C_{\text{\tiny\ref{prop: 09582593852}}}$, $1\leq k\leq\min\big(
(K_2/8)^{m/2},2^{n/C_{\text{\tiny\ref{prop: 09582593852}}}}\big)$,
and let $X=(X_1,\dots,X_n)$ be a random vector
uniformly distributed on $\Lambda_n(k,\gfn,Q_1,Q_2,\rho,\sigma,h)$.
Then
$$
\Prob\big\{\bal_n(X,m,K_{\text{\tiny\ref{prop: 09582593852}}},
K_2)< k m^{1/2}/C_{\text{\tiny\ref{prop: 09582593852}}}\big\}\leq \varepsilon^n.
$$
\end{prop}

Let us describe the proof of Theorem~\ref{th: gradual} informally.
Assume that the hyperplane $H_1$ admits a normal vector $X$ which belongs to $\gncvectors_n(r,\gfn,\delta,\rho)$.
We need to show that with a large probability the u-degree $\bal_n(X,m,K_1,K_2)$ of $X$ is very large,
say, at least $\varepsilon^{-m}$ for a small $\varepsilon>0$.
The idea is to split the collection $\gncvectors_n(r,\gfn,\delta,\rho)$ into about $\log_2(\varepsilon^{-m})$ subsets according
to the magnitude of the u-degree (that is, each subset
$\mathcal T_N$ will have a form
$\mathcal T_N=\big\{x\in \gncvectors_n(r,\gfn,\delta,\rho):\;\bal_n(x,m,K_1,K_2)\in [N,2N)\big\}$
for an appropriate $N$).
To show that for each $N\ll \varepsilon^{-m}$ the probability of $X\in \mathcal T_N$
is very small, we define a discrete approximation $\AA$ of $\mathcal T_N$ consisting of all vectors $y\in \frac{1}{k}\Z^n$
such that $\|y-x\|_\infty\leq 1/k$ for some $x\in \mathcal T_N$ and additionally,
in view of Proposition~\ref{l: stability of bal},
$\bal_n(y,m,c_{\text{\tiny\ref{l: stability of bal}}} K_1,K_2)\leq 2N$ and
$\bal_n(y,m,c_{\text{\tiny\ref{l: stability of bal}}}^{-1} K_1,K_2)\geq N$.
We can bound the cardinality of such set $\AA$ by $(\tilde\varepsilon\,k)^n$, for a small $\tilde\varepsilon>0$,
by combining Proposition~\ref{prop: 09582593852} with Lemma~\ref{l: disc with l} and with the following simple fact.

\begin{lemma}\label{l: Lambda_n card}
Let $k\geq 1$, $h\in\R$, $\rho,\delta\in(0,1)$,  $Q_1,Q_2\subset[n]$ with
$|Q_1|,|Q_2|=\lceil \delta n\rceil$, and  $\gfn$
satisfies \eqref{gfncond} with some $K_3\geq 1$. Then
$|\Lambda_n(k,\gfn,Q_1,Q_2,\rho,\sigma,h)|\leq \big(C_{\text{\tiny\ref{l: Lambda_n card}}}k\big)^n$,
where $C_{\text{\tiny\ref{l: Lambda_n card}}}\geq 1$ depends only on $K_3$.
\end{lemma}

On the other hand, for each fixed vector $y$ in the set $\AA$
we can estimate the probability that it ``approximates'' a normal vector to $H_1$
by using Corollary~\ref{cor: cf tensor}:
$$
\Prob\big\{\mbox{$y$ is an ``approximate'' normal vector to $H_1$}\big\}\leq (C'/k)^n\quad \mbox{for every }y\in \AA,
$$
for some constant $C'\ll \tilde\varepsilon^{-1}$.
Taking the union bound, we obtain
$$
\Prob\big\{X\in \mathcal T_N\big\}\leq \Prob\big\{\mbox{$\AA$ contains an ``approximate''
normal vector to $H_1$}\big\}\leq (C'/k)^n\,(\tilde\varepsilon\,k)^n\ll 1.
$$
Below, we make this argument rigorous.

\bigskip

\begin{proof}[Proof of Theorem~\ref{th: gradual}]
We start by defining parameters. We always assume that $n$ is large enough, so
all statements used below work for our $n$.
Fix any $R\geq 1$, $r>0$ and $s>0$, and set $b:=\lfloor(2p R)^{-1}\rfloor$.
Let $K_2= 32 \exp(16R).$
Note that the function $\gfn(6\, \cdot)$ is a growth function that satisfies condition \eqref{gfncond}
with parameter $K_3'=(K_3)^8$. In particular, choosing $j$ so taht $2^{j-1}\leq 6n\leq 2^j$, we have
$$\gfn(6n)\leq \gfn(2^j)\leq (K_3')^{2^j/j}\leq (K_3')^{12n/\log_2(6n)}\leq K_3^n.$$
For brevity, we denote $$C_{\text{\tiny\ref{cor: norm of centered}}}:=C_{\text{\tiny\ref{cor: norm of centered}}}(s,2R),\,\,
C_{\text{\tiny\ref{l: column supports}}}:=C_{\text{\tiny\ref{l: column supports}}}(2R),\,\,
c_{\text{\tiny\ref{l: stability of bal}}}':=c_{\text{\tiny\ref{l: stability of bal}}}'(K_2),\,\,
c_{\text{\tiny\ref{l: stability of bal}}}:=c_{\text{\tiny\ref{l: stability of bal}}}(K_2),\,\,
C_{\text{\tiny\ref{l: Lambda_n card}}}
=C_{\text{\tiny\ref{l: Lambda_n card}}}(K_3').$$
Set
$$
K_1:=\max\big(K_{\text{\tiny\ref{prop: 09582593852}}}(\delta,\rho/4)/c_{\text{\tiny\ref{l: stability of bal}}},
C_{\text{\tiny\ref{p: low bound on bal}}}(r,\delta,\rho)\big),
$$
and
$$
\varepsilon:= \min\Big(K_2^{-1},\, c_{\text{\tiny\ref{l: stability of bal}}}' \big( 384e K_3\, \exp({C_{\text{\tiny\ref{l: permut}}}})\,
  C_{\text{\tiny\ref{l: Lambda_n card}}}
  C_{\text{\tiny\ref{l: tensor}}}C_{\text{\tiny\ref{p: cf est}}}C_{\text{\tiny\ref{cor: norm of centered}}}
 \big)^{-1}\exp(-3R)\Big)
$$
We will assume that $pn$ is sufficiently large so that
$$5\exp(-2Rpn)\leq\exp(-Rpn)\quad \mbox{
and }\quad \exp(-3Rpn)\leq\frac{1}{2Rpn}\exp(-2Rpn).$$
Moreover, we will assume that
\begin{equation}\label{eq: aux 2-9582-598}
\mbox{$2R C_{\text{\tiny\ref{l: column supports}}}p\leq 1\quad $
and $\quad C_{\text{\tiny\ref{l: column supports}}}\leq p n$}
\end{equation}
and
\begin{align*}
&\frac{1}{8p}\geq \max(C_{\text{\tiny\ref{prop: 09582593852}}}(\varepsilon,\delta,\rho/4,K_3'),
C_{\text{\tiny\ref{p: low bound on bal}}}(r,\delta,\rho));\quad
pn\geq 16C_{\text{\tiny\ref{prop: 09582593852}}}(\varepsilon,\delta,\rho/4,K_3')^2;\\
&e^{2Rp}\leq 2^{1/C_{\text{\tiny\ref{prop: 09582593852}}}(\varepsilon,\delta,\rho/4,K_3')};\quad
c_{\text{\tiny\ref{l: stability of bal}}}'/3\geq \exp(-Rpn);\quad \lfloor \exp(Rpn)/c_{\text{\tiny\ref{l: stability of bal}}}'\rfloor\,n\leq 2^n.
\end{align*}

\medskip

Define two auxiliary random objects as follows.
Set
$$Z:=\mbox{$\{x\in\R^n:\;x^*_{\lfloor rn\rfloor}=1,\,\, \,
\bal_n(x,m,K_1,K_2)\geq \exp(Rpn)\,\,$ for all $\,\, pn/8\leq m\leq 8pn\}$,}$$
and let $X$ be a random vector measurable with respect to $H_1$ and such that
\begin{itemize}

\item $X\in \big(\gncvectors_n(r,\gfn,\delta,\rho)\cap H_1^\perp\big)\setminus Z\quad\mbox{whenever }\;
\big(\gncvectors_n(r,\gfn,\delta,\rho)\cap H_1^\perp\big)\setminus Z\neq\emptyset$;

\item $X\in \big(\gncvectors_n(r,\gfn,\delta,\rho)\cap H_1^\perp\big)\cap Z\quad\mbox{whenever }\;
\big(\gncvectors_n(r,\gfn,\delta,\rho)\cap H_1^\perp\big)\setminus Z=\emptyset$ and
$\gncvectors_n(r,\gfn,\delta,\rho)\cap H_1^\perp\neq \emptyset$;

\item $X={\bf 0}\quad\mbox{whenever }\quad
\gncvectors_n(r,\gfn,\delta,\rho)\cap H_1^\perp= \emptyset$.

\end{itemize}
(Note that $H_1^\perp$ may have dimension larger than one with non-zero probability,
and thus  $\pm X$ is not uniquely defined). Note that to prove the theorem, it is sufficient to show that
with probability at least $1-\exp(-R pn)$ one has either $X={\bf 0}$ or $X\in Z$.

Next, we denote
$$
\xi:=\begin{cases}\min\limits_{8pn \geq m\geq pn/8}\bal_n(X,m,K_1,K_2),&\mbox{whenever }\;X\neq {\bf 0};\\
+\infty,&\mbox{otherwise.}\end{cases}
$$
Then, proving the theorem amounts to showing that $\xi<\exp(R pn)$ with probability at most $\exp(-R pn)$.

We say that a collection of indices $I\subset[n]$ is admissible if $1\notin I$ and $|I|\geq n-b-1$.
For  admissible sets $I$ consider disjoint collection of events $\{\Event_I\}_I$ defined by
\begin{align*}
\Event_I:=\big\{\forall i\in I :\, \, |\supp \col_i(M_n)|\in [pn/8,8pn]
\qand
\forall i\notin I :\, \, |\supp \col_i(M_n)|\notin [pn/8,8pn]
\big\}.
\end{align*}
Further, denote
$$
\widetilde \Event:=\big\{\|M_n-\E M_n\|\leq C_{\text{\tiny\ref{cor: norm of centered}}}\sqrt{pn}\big\}.
$$
According to Corollary~\ref{cor: norm of centered}, $\Prob(\widetilde \Event)\geq 1-\exp(-2Rpn)$, while
 by Lemma~\ref{l: column supports} and  \eqref{eq: aux 2-9582-598},
$$
\Prob\Big(\bigcup_{I}\Event_I\Big)\geq 1-\exp(-n/C_{\text{\tiny\ref{l: column supports}}})\geq 1-\exp(-2Rpn).
$$
Denote by $\mathcal I$ the collection of all admissible $I$ satisfying
$2\Prob(\Event_I\cap\widetilde\Event)\geq \Prob(\Event_I)$. Then for $I\in \mathcal I$, we have
$\Prob(\Event_I)\geq 2\Prob(\Event_I\cap\widetilde\Event^c)$, and, using that events $\Event_I$ are disjoint,
$$
\Prob\Big(\bigcup_{I\in \mathcal I}\Event_I\Big)\geq 1-\exp(-2Rpn)-2\Prob(\widetilde\Event^c)\geq 1-3\exp(-2Rpn).
$$
Hence,
\begin{align*}
\Prob\big\{\xi<\exp(Rpn)\big\}
&\leq \sum\limits_{I\in \mathcal I}\Prob\big(\big\{\xi<\exp(Rpn)\big\}\cap \Event_I\cap \widetilde \Event\big)
+\Prob\Big(\bigcap_{I\in \mathcal I}\Event_I^c\Big)+\Prob(\widetilde \Event^c)\\
&\leq \sum\limits_{I\in \mathcal I}\Prob\big(\big\{\xi<\exp(Rpn)\big\}\;|\;\Event_I\cap \widetilde \Event\big)
\Prob(\Event_I\cap \widetilde \Event)+4\exp(-2Rpn).
\end{align*}
Therefore, to prove the theorem it is sufficient to show that for any $I\in\mathcal I$,
$$
\Prob\big(\big\{\xi<\exp(Rpn)\big\}\;|\;\Event_I\cap \widetilde \Event\big)\leq \exp(-2Rpn).
$$
Fix an admissible $I\in\mathcal I$,
denote by $B_I$ the $|I|\times n$ matrix obtained by transposing
columns $\col_i(M_n)$, $i\in I$, and let $\widetilde B_I$ be the non-random $|I|\times n$ matrix
with all elements equal to $p$.
Note that, in view of our definition of $K_1$, the assumptions on $p$
and Proposition~\ref{p: low bound on bal}, we have a {\it deterministic} relation
$$
\xi\geq \sqrt{pn/8}
$$
everywhere on the probability space.
For each real number $N\in J_p:=[\sqrt{pn/8},\exp(Rpn)/2]$, denote by $\Event_{N,I}$
the event
$$
\Event_{N,I}:=\big\{\xi\in[N,2N)\big\}\cap \Event_I\cap \widetilde \Event.
$$
Splitting the interval $J_p$ into subintervals, we observe that   it is sufficient to show that for every
 $N\in J_p$ we have
$$
\Prob\big(\Event_{N,I}\;|\;\Event_I\cap \widetilde \Event\big)\leq \exp(-3Rpn)\leq\frac{1}{2Rpn}\exp(-2Rpn).
$$

\medskip

The rest of the argument is devoted to estimating probability of $\Event_{N,I}$ for fixed  $N\in J_p$ and
fixed $I\in \mathcal I$.
Set $k:=\lceil 2N/c_{\text{\tiny\ref{l: stability of bal}}}'\rceil$.
Let ${\bf m}:\Event_{N,I}\to [pn/8,8pn]$ be a (random) integer such that
$$\bal_n(X,{\bf m},K_1,K_2)\in [N,2N)\;\,\,\, \mbox{ everywhere on }\,\,\, \;\Event_{N,I}.$$
Since on  $\Event_{N,I}$ we have
$\bal_n(X,{\bf m},K_1,K_2)\leq 2N\leq c_{\text{\tiny\ref{l: stability of bal}}}'k$, applying
Proposition~\ref{l: stability of bal}, we can construct a random vector ${\bf Y}:\Event_{N,I}\to \frac{1}{k}\Z^n$
having the following properties:
\begin{itemize}
\item $\|{\bf Y}-X\|_\infty\leq 1/k$ everywhere on $\Event_{N,I}$,
\item $\bal_n({\bf Y},{\bf m},c_{\text{\tiny\ref{l: stability of bal}}} K_1,K_2)\leq 2N$ everywhere on $\Event_{N,I}$,
\item $\bal_n({\bf Y},m,c_{\text{\tiny\ref{l: stability of bal}}}^{-1} K_1,K_2)\geq N$
for all $m\in [pn/8,8pn]$ and everywhere on $\Event_{N,I}$.
\end{itemize}
 The first condition together with the inclusion $\Event_{N,I}\subset \widetilde \Event$ implies that
$$
\|(B_I-\widetilde B_I)({\bf Y}-X)\|\leq C_{\text{\tiny\ref{cor: norm of centered}}}\sqrt{p}n/k.
$$
Using that $B_I X=0$ and that $\widetilde B_I ({\bf Y}-X)=p(\sum_{i=1}^n ({\bf Y}_i-X_i))\,{\bf 1}_I$,
we observe that there is a random number ${\bf z}:\Event_{N,I}\to [-pn/k, pn/k]\cap \frac{\sqrt{pn}}{k}\Z$
such that everywhere on $\Event_{N,I}$ one has
$$
  \|B_I {\bf Y}-{\bf z}\,{\bf 1}_I\|\leq 2C_{\text{\tiny\ref{cor: norm of centered}}}\sqrt{p}n/k.
$$

Let $\Lambda$ be a subset of
$$
\bigcup\limits_{q=\lfloor-4\gfn(6 n)/\rho \rfloor}^{
\lceil 4\gfn(6 n)/\rho \rceil}\;
\bigcup\limits_{\bar\sigma\in\bar\Pi_n}\bigcup\limits_{|Q_1|,|Q_2|=\lceil\delta n\rceil}
\Lambda_n(k,\gfn(6\, \cdot),Q_1,Q_2,\rho/4,\bar\sigma,\rho q/4),
$$
consisting of all vectors $y$ such that
\begin{itemize}
\item $\bal_n(y,m,c_{\text{\tiny\ref{l: stability of bal}}} K_1,K_2)\leq 2N$ for {\it some } $m\in[pn/8,8pn]$;
\item $\bal_n(y,m,c_{\text{\tiny\ref{l: stability of bal}}}^{-1} K_1,K_2)\geq N$
for all $m\in [pn/8,8pn]$.
\end{itemize}
Note that by Lemma~\ref{l: disc with l} the entire range of ${\bf Y}$
on $\Event_{N,I}$ falls into $\Lambda$.

Combining  the above observations,
$$\Event_{N,I}\subset
\big\{\|B_I y-z{\bf 1}_I\|\leq 2C_{\text{\tiny\ref{cor: norm of centered}}}\sqrt{p}n/k\;\mbox{ for some }y\in\Lambda,
\mbox{ $z\in [-pn/k, pn/k]\cap \frac{\sqrt{pn}}{k}\Z$}\big\},
$$
whence, using that $2\Prob(\Event_I\cap \widetilde\Event)\geq \Prob(\Event_I)$ by the definition of $\mathcal I$,
\begin{align*}
\Prob(\Event_{N,I}\,|\,\Event_I\cap \widetilde\Event)
&\leq 2\Prob\big\{\|B_I y-z{\bf 1}_I\|\leq 2C_{\text{\tiny\ref{cor: norm of centered}}}\sqrt{p}n/k\;\mbox{ for some }y\in\Lambda,
\mbox{ $z\in [-pn/k, pn/k]\cap \frac{\sqrt{pn}}{k}\Z$}\;|\;\Event_I\big\}\\
&\leq 6|\Lambda|\sqrt{pn}\,\max\limits_{z\in \frac{\sqrt{pn}}{k}\Z}\,\,\max\limits_{y\in \Lambda}\Prob\big\{
\|B_I y-z{\bf 1}_I\|\leq 2C_{\text{\tiny\ref{cor: norm of centered}}}\sqrt{p}n/k\;|\;\Event_I\big\}.
\end{align*}
To estimate the last probability, we apply Corollary~\ref{cor: cf tensor}
with $t:= C_{\text{\tiny\ref{cor: norm of centered}}} \sqrt{8pn}/{N}$ (note that $k\geq 2N$, $2|I|\geq n$, and
that $t$ satisfies the assumption of the corollary).
We obtain  that for all admissible $y$ and $z$,
\begin{align*}
\Prob\big\{
\|B_I y-z{\bf 1}_I\|\leq 2C_{\text{\tiny\ref{cor: norm of centered}}}\sqrt{p}n/k\;|\;\Event_I\big\}
&\leq\Prob\bigg\{
\|B_I y-z{\bf 1}_I\|\leq \frac{C_{\text{\tiny\ref{cor: norm of centered}}}\sqrt{8pn}}{N}\,\sqrt{|I|}\;|\;\Event_I\bigg\}\\
&\leq (16C_{\text{\tiny\ref{l: tensor}}}C_{\text{\tiny\ref{p: cf est}}}C_{\text{\tiny\ref{cor: norm of centered}}}/N)^{|I|}.
\end{align*}
On the other hand, the cardinality of $\Lambda$ can be estimated by combining
Lemma~\ref{l: Lambda_n card}, Lemma~\ref{l: permut}
and Proposition~\ref{prop: 09582593852}
(note that our choice of parameters guarantees applicability of these statements):
$$
  |\Lambda|\leq 8pn \varepsilon^n\, (9\gfn(6n)/\rho ) \exp({C_{\text{\tiny\ref{l: permut}}}n})\,2^{2n}
  (C_{\text{\tiny\ref{l: Lambda_n card}}}k)^n
   \leq (72pn/\rho ) \varepsilon^n\, K_3^n\, \exp({C_{\text{\tiny\ref{l: permut}}}n})\,2^{2n}
  (C_{\text{\tiny\ref{l: Lambda_n card}}}k)^n,
$$
where $C_{\text{\tiny\ref{l: Lambda_n card}}}
=C_{\text{\tiny\ref{l: Lambda_n card}}}(K_3')$.
Thus, using our choice of parameters and assuming in addition that $2^n\geq 72pn/\rho $
\begin{align*}
\Prob(\Event_{N,I}\,|\,\Event_I\cap \widetilde\Event)
&\leq   \varepsilon^n\, (8 K_3\, \exp({C_{\text{\tiny\ref{l: permut}}}})\,
  C_{\text{\tiny\ref{l: Lambda_n card}}}k)^n
\,(16C_{\text{\tiny\ref{l: tensor}}}C_{\text{\tiny\ref{p: cf est}}}C_{\text{\tiny\ref{cor: norm of centered}}}/N)^{|I|}\\
&\leq
 \varepsilon^n\, (8 K_3\, \exp({C_{\text{\tiny\ref{l: permut}}}})\,
  C_{\text{\tiny\ref{l: Lambda_n card}}}k)^n\,
(48C_{\text{\tiny\ref{l: tensor}}}C_{\text{\tiny\ref{p: cf est}}}C_{\text{\tiny\ref{cor: norm of centered}}}
/(c_{\text{\tiny\ref{l: stability of bal}}}'k))^{n}N^{1+\lfloor (2pR)^{-1}\rfloor}\\
&\leq  \varepsilon^n\, (384 K_3\, \exp({C_{\text{\tiny\ref{l: permut}}}})\,
  C_{\text{\tiny\ref{l: Lambda_n card}}}
  C_{\text{\tiny\ref{l: tensor}}}C_{\text{\tiny\ref{p: cf est}}}C_{\text{\tiny\ref{cor: norm of centered}}}
/(c_{\text{\tiny\ref{l: stability of bal}}}'))^{n}\, e^n\\
&\leq \exp(-3Rn),
\end{align*}
by our choice of parameters. The result follows.
\end{proof}

\subsection{Anti-concentration on a lattice}

The goal of this subsection is to prove Proposition~\ref{prop: 09582593852}. Thus, in this subsection,
we fix $\rho,\delta\in(0,1/4]$, a growth function $\gfn$ satisfying \eqref{gfncond}, which in particular  
means that $\gfn(n)\leq K_3 ^{2n/\log_2 n}$,
a permutation $\sigma\in\Pi_n$, a number $h\in\R$,  two sets $Q_1,Q_2\subset[n]$
such that $|Q_1|,|Q_2|=\lceil \delta n\rceil$, and we do not repeat these assumptions in  lemmas below.
We also always use short notation $\Lambda_n$ for the set $\Lambda_n(k,\gfn,Q_1,Q_2,\rho,\sigma,h)$
defined in \eqref{eq: param l def}.

\smallskip

We start with  auxiliary probabilistic statements which are just special forms of Markov's inequality.
\begin{lemma}[Integral form of Markov's inequality, I]\label{l: int markov}
For each $s\in[a,b]$, let $\xi(s)$ be a non-negative random variable with $\xi(s)\leq 1$ a.e.
Assume that the random function $\xi(s)$ is integrable on $[a,b]$ with probability one.
Assume further that for some integrable function $\phi(s):\, [a,b]\to\R_+$ and some $\varepsilon>0$ we have
$$
\Prob\big\{\xi(s)\leq \phi(s)\big\}\geq 1-\varepsilon
$$
for all $s\in[a,b]$. Then for all $t>0$,
$$
\Prob\bigg\{\int_a^b \xi(s)\,ds\geq \int_a^b \phi(s)\,ds+t(b-a)\bigg\}\leq \varepsilon/t.
$$
\end{lemma}
\begin{proof}
Consider a random set
$$
I:=\big\{s\in[a,b]:\;\xi(s)> \phi(s)\big\}.
$$
Since $\Prob\{s\in I\}\leq \varepsilon$ for any $s\in[a,b]$,
we have
$
\Exp|I|\leq \varepsilon(b-a).
$
Therefore, by the Markov inequality,
$
\Prob\big\{|I|\geq t(b-a)\big\}\leq \varepsilon/t
$
for all $t>0$.
The result follows by noting that
$$
\int_a^b \xi(s)\,ds\leq |I|+\int_a^b \phi(s)\,ds.
$$
\end{proof}
\begin{lemma}[Integral form of Markov's inequality, II]\label{l: sum markov}
Let $I$ be a finite set, and for each $i\in I$, let $\xi_i$ be a non-negative random variable with $\xi_i\leq 1$ a.e.
Assume further that for some $\phi(i):I\to\R_+$ and some $\varepsilon>0$ we have
$$
\Prob\big\{\xi_i\leq \phi(i)\big\}\geq 1-\varepsilon
$$
for all $i\in I$. Then for all $t>0$,
$$
\Prob\bigg\{\frac{1}{|I|}\sum_{i\in I} \xi_i\geq \frac{1}{|I|}\sum_{i\in I} \phi(i)+t\bigg\}\leq \varepsilon/t.
$$
\end{lemma}
The proof of Lemma~\ref{l: sum markov} is almost identical to that of Lemma~\ref{l: int markov}, and we omit it.

Our next statement will be important in an approximation (discretization) argument used later in the proof.
\begin{lemma}[Lipschitzness of the product $\prod\psi_{K_2}(\cdot)$]\label{l: lip of prod}
Let $y_1,\dots,y_n\in \R$  and set $y:=\max\limits_{w\leq n}|y_w|$.
Further, let $S_1,\dots,S_m$ be some non-empty subsets of $[n]$.
For $i\leq m$ denote
$$
 f_i(s):=\psi_{K_2}\bigg(\Big|\frac{1}{|S_i|}
\sum_{w\in S_i}\exp(2\pi{\bf i}\,y_w s)\Big|\bigg)\quad \mbox{ and let }  \quad
f(s):=\prod\limits_{i=1}^m f_i(s).
$$
Then $f$ (viewed as a function of $s$) is $(8 K_2\pi y\,m)$-Lipschitz.
\end{lemma}
\begin{proof}
By our definition, $\psi_{K_2}$ is $1$-Lipschitz for any $K_2\geq 1$,
hence
$f_i$ (viewed as a function of $s$) is
$2\pi y$-Lipschitz.
Since  $\big|\sum_{w\in S_i}\exp(2\pi{\bf i}\,y_w s)\big|\leq |S_i|$,
by the definition of the function  $\psi_{K_2}$, we have
$1/(2K_2)\leq f_i \leq 1$, hence, for all $s,\Delta s\in\R$,
$$
   \frac{f_i(s)}{f_i(s+\Delta s)} = 1 +  \frac{f_i(s)-f_i(s+\Delta s)}{f_i(s+\Delta s)}
   \leq 1+4 K_2\pi y\,|\Delta s|.
$$
 Taking the product, we obtain that
$$
   \frac{f(s)}{f(s+\Delta s)}  \leq \big(1+4 K_2\pi y\,|\Delta s|\big)^m
   \leq 1+8 K_2\pi y\,m\,|\Delta s|
$$
whenever $8 K_2\pi y\,m\,|\Delta s|\leq 1/2$.
This, together with the bound $f\leq 1$ implies for all $s,\Delta s\in\R$,
$$
    f(s) - f(s+\Delta s)   \leq  8 K_2\pi y\,m\,|\Delta s|,
$$
which completes the proof.
\end{proof}

In the next two lemmas we initiate the study of random variables $\exp(2\pi{\bf i}\,\eta[I_w]\,s_j/k)$,
more specifically, we will be interested in the property that, under appropriate assumptions on $s_j$'s,
the sum of such variables is close to zero on average.

\begin{lemma}\label{l: aux 43098275}
Let $\varepsilon\in(0,1]$,  $k\geq 1$, $\ell\geq 2/\varepsilon$. Let $I$ be an integer interval
and let $s_1,\dots,s_\ell$ be real numbers such that for all $j\neq u$,
$$ \frac{k}{\eps |I|} \leq |s_j-s_u|\leq \frac{k}{2}.$$
Then
\begin{align*}
\Exp\,\Big|\sum_{j=1}^\ell \exp(2\pi{\bf i}\,\eta[I]\,s_j/k)\Big|^2\leq\varepsilon\ell^2.
\end{align*}
\end{lemma}
\begin{proof}
We will determine the restrictions on parameter $R$ at the end of the proof.
We have
\begin{equation}\label{eq: aux 976120975}
\begin{split}
\Exp\,\Big|\sum_{j=1}^\ell \exp(2\pi{\bf i}\,\eta[I]\,s_j/k)\Big|^2
&=\sum_{j=1}^{\ell}\sum_{u=1}^{\ell}\Exp \exp\big(2\pi{\bf i}\,\eta[I]\,(s_j-s_u)/k\big)\\
&\leq \ell+\Big|\sum\limits_{j\neq u}\Exp \exp\big(2\pi{\bf i}\,\eta[I]\,(s_j-s_u)/k\big)\Big|.
\end{split}
\end{equation}
Further, denoting $a=\min I$ and $b=\min I$, we observe for any $j\neq u$,
\begin{align*}
\Exp &\exp\big(2\pi{\bf i}\,\eta[I]\,(s_j-s_u)/k\big)\\
&=\frac{1}{|I|}\sum_{v=a}^{b}\exp\big(2\pi{\bf i}\,v\,(s_j-s_u)/k\big)\\
&=\frac{1}{|I|}\exp\big(2\pi{\bf i}\,a\,(s_j-s_u)/k\big)\cdot
\frac{1-\exp\big(2\pi{\bf i}\,(b-a+1)\,(s_j-s_u)/k\big)}{1-\exp\big(2\pi{\bf i}\,(s_j-s_u)/k\big)}.
\end{align*}
In view of assumptions on  $|s_j-s_u|$
$$\big|1-\exp\big(2\pi{\bf i}\,(s_j-s_u)/k\big)\big|= \big|2\sin(\pi\,(s_j-s_u)/k)
\big| \geq \frac{4|s_j-s_u|}{k}\geq \frac{4}{\eps |I|}.$$
Therefore,
$$
 \big|\Exp \exp\big(2\pi{\bf i}\,\eta[I]\,(s_j-s_u)/k\big)\big|\leq \frac{\eps}{2}.
$$
Using \eqref{eq: aux 976120975}, we complete the proof.
\end{proof}

\begin{lemma}\label{l: aux 9871039481}
For every $\varepsilon\in(0, 1/2]$ there are $R_{\text{\tiny\ref{l: aux 9871039481}}}
=R_{\text{\tiny\ref{l: aux 9871039481}}}(\varepsilon)>0$
and $\ell:=\ell_{\text{\tiny\ref{l: aux 9871039481}}}(\varepsilon)\in\N$, $\ell\geq 1000$,
with the following property.
Let $k\geq 1$, $u\geq \ell$, let $I_w$ ($w=1,2,\dots,u$) be integer intervals, and let
$s_1,\dots,s_\ell$ be real numbers such that $|I_w|\,|s_j-s_q|\geq R_{\text{\tiny\ref{l: aux 9871039481}}}k$,
and $|s_j-s_q|\leq k/2$ for all $j\neq q$ and $w\leq u$.
Then, assuming that random variables $\eta[I_w]$, $w\leq u$, are mutually independent, one has
$$
\Prob\Big\{\Big|\frac{1}{u}\sum_{w=1}^u
\exp(2\pi{\bf i}\,\eta[I_w]\,s_j/k)\Big|\geq \varepsilon\;\mbox{ for at least $\varepsilon\ell$ indices }j\Big\}
\leq \varepsilon^{u}.
$$
\end{lemma}
\begin{proof}
Fix any $\varepsilon\in(0,1/2]$, and set $\varepsilon_1:=2^{-10}e^{-6}\varepsilon^{4+9/\varepsilon}$.
Set $R:=1/\varepsilon_1$
and $\ell:=\lceil 2/\varepsilon_1\rceil$.
Assume that $u\geq \ell$, and let numbers $s_j$ and integer intervals $I_w$ satisfy the assumptions of the lemma.
Denote the event
$$
\Big\{\Big|\frac{1}{u}\sum_{w=1}^u
\exp(2\pi{\bf i}\,\eta[I_w]\,s_j/k)\Big|\geq \varepsilon\;\mbox{ for at least $\varepsilon\ell$ indices }j\Big\}
$$
by $\Event$, and additionally, for any subset $Q\subset[\ell]$ of cardinality $\lfloor \varepsilon\ell/4\rfloor$
and any vector $z\in\{-1,1\}^2$, set
$$
\Event_{Q,z}:=\Big\{\Big\langle\Big(\frac{1}{u}\sum_{w=1}^u
\cos(2\pi\,\eta[I_w]\,s_j/k),
\frac{1}{u}\sum_{w=1}^u
\sin(2\pi\,\eta[I_w]\,s_j/k)\Big),z\Big\rangle\geq \varepsilon\;\mbox{ for all }j\in Q\Big\}.
$$
It is not difficult to see that
$$
\Event\subset\bigcup\limits_{Q,z}\Event_{Q,z},
$$
whence it is sufficient to show that for any admissible $Q,z$,
\begin{equation}\label{entosh}
\Prob(\Event_{Q,z})\leq \frac{1}{4}{\ell\choose \lfloor \varepsilon\ell/4\rfloor}^{-1}\varepsilon^u.
\end{equation}
Without loss of generality, we can consider $Q=Q_0:=\big[\lfloor \varepsilon\ell/4\rfloor\big]$. 
Event $\Event_{Q_0,z}$ is contained inside the event
$$
\Big\{\Big|\sum_{j\in Q_0}\sum_{w=1}^u
\exp(2\pi{\bf i}\,\eta[I_w]\,s_j/k)\Big|\geq 2^{-1/2}\varepsilon u\,\lfloor \varepsilon\ell/4\rfloor\Big\},
$$
while the latter is contained inside the event
$$
\Big\{\Big|\sum_{j\in Q_0}
\exp(2\pi{\bf i}\,\eta[I_w]\,s_j/k)\Big|\geq \frac{\varepsilon}{4}\,\lfloor \varepsilon\ell/4\rfloor
\mbox{ for at least $\varepsilon u/4$ indices $w$}\Big\}.
$$
Thus, taking the union over all admissible choices of $\lceil\varepsilon u/4\rceil$ indices $w\in[u]$, we get
$$
\Prob(\Event_{Q_0,z})
\leq {u\choose \lceil\varepsilon u/4\rceil}\max\limits_{F\subset[u],\,|F|=\lceil\varepsilon u/4\rceil}\Prob\Big\{
\Big|\sum_{j\in Q_0}
\exp(2\pi{\bf i}\,\eta[I_w]\,s_j/k)\Big|\geq \frac{\varepsilon}{4}\,\lfloor \varepsilon\ell/4\rfloor\mbox{ for all $w\in F$}
\Big\}.
$$
To estimate the last probability, we  apply Markov's inequality, together with the bound for the second moment from
 Lemma~\ref{l: aux 43098275} (applied with  $\varepsilon_1$), and using independence of $\eta[I_w]$,
$w\leq u$.
We then get
$$
\max\limits_{F\subset[u]\atop |F|=\lceil\varepsilon u/4\rceil}\Prob\Big\{
\Big|\sum_{j\in Q_0}
\exp(2\pi{\bf i}\,\eta[I_w]\,s_j/k)\Big|\geq \frac{\varepsilon}{4}\,\lfloor \varepsilon\ell/4\rfloor\mbox{ for all $w\in F$}
\Big\}\leq
\bigg(\frac{\varepsilon_1 \ell^2}{(\varepsilon^2 \ell/32)^2}\bigg)^{\lceil \varepsilon u/4\rceil}
\leq
e^{-3\eps u/2}  \varepsilon^{2u}.
$$
In view of (\ref{entosh}) this implies the result, since using  $8\leq \ell \leq u$ and $\eps<1/2$,
we have
$$
  4 {\ell\choose \lfloor \varepsilon\ell/4\rfloor}\eps^{-u} {u\choose \lceil\varepsilon u/4\rceil} e^{-3\eps u/2}  \varepsilon^{2u}
  \leq 4e^{-3\eps u/2}  \bigg(\frac{4e}{\varepsilon }\bigg)^{\varepsilon \ell/4}
   \bigg(\frac{2e}{\varepsilon }\bigg)^{\varepsilon u/2}\varepsilon^{u} \leq 4 (16e^{-3})^{\eps u/4}\,
   \varepsilon^{u/4}\leq 1.
$$
\end{proof}

Our next step is to  show that for the vector $X=(X_1,\dots,X_n)$ uniformly distributed on  $\Lambda_n$
the random product $\prod\limits_{i=1}^m\psi_{K_2}\big(\big|\frac{1}{\lfloor n/m\rfloor}
\sum_{w\in S_i}\exp(2\pi{\bf i}X_w s)\big|\big)$ is, in a certain sense,
 typically small (for most choices of $s$). To do this we first show that given a collection
of distinct numbers $s_1,\dots,s_\ell$ which are pairwise well separated, the above product is small
for at least one $s_j$ with very high probability.

\begin{lemma}\label{l: aux 0876958237}
For any $\varepsilon\in(0,1/2]$ there are $R_{\text{\tiny\ref{l: aux 0876958237}}}
=R_{\text{\tiny\ref{l: aux 0876958237}}}(\varepsilon)\geq 1$
and $\ell:=\ell_{\text{\tiny\ref{l: aux 0876958237}}}(\varepsilon)\in\N$
with the following property.
Let $k,m,n\in\N$ be with $n/m\geq \ell$. Let $1\leq K_2\leq 2/\varepsilon$,
$X=(X_1,\dots,X_n)$ be a random vector
uniformly distributed on $\Lambda_n$, and
let $s_1,\dots,s_\ell$ be real numbers in $[0,k/2]$ such that $|s_j-s_q|\geq R_{\text{\tiny\ref{l: aux 0876958237}}}$ for all $j\neq q$.
Fix disjoint subsets $S_1,\dots,S_m$ of $[n]$, each of cardinality $\lfloor n/m\rfloor$.
Then
$$
\Prob\Big\{\forall j\leq \ell\, : \, \, \prod\limits_{i=1}^m\psi_{K_2}\bigg(\Big|\frac{1}{\lfloor n/m\rfloor}
\sum_{w\in S_i}\exp(2\pi{\bf i}X_w s_j)\Big|\bigg)\geq (K_2/2)^{-m/2}\Big\}
\leq \varepsilon^n.
$$
\end{lemma}
\begin{proof}
Fix any $\varepsilon\in(0,1/2]$ and set $\ell:=\ell_{\text{\tiny\ref{l: aux 9871039481}}}(\varepsilon^5)\geq 1000$
and $R:=R_{\text{\tiny\ref{l: aux 9871039481}}}(\varepsilon^5)$.
Assume that $n/m\geq \ell$.
Note that, by our definition of $\Lambda_n$, the coordinates
of $X$ are independent and, moreover,
each variable $kX_w$ is distributed on an integer interval of cardinality at least $k$.
Thus, it is sufficient to prove that for any collection of integer intervals $I_j$, $j\leq n$, such that
$|I_j|\geq k$, the event
$$
\Event:=\Big\{\forall j\leq \ell\, : \, \, \prod\limits_{i=1}^m\psi_{K_2}\bigg(\Big|\frac{1}{\lfloor n/m\rfloor}
\sum_{w\in S_i}\exp(2\pi{\bf i}\,\eta[I_w]\, s_j/k)\Big|\bigg)\geq (K_2/2)^{-m/2}\Big\}.
$$
has probability at most $\varepsilon^n$, where, as usual, we assume that the variables $\eta[I_w]$, $w\in S_i$, $i\leq m$,
are jointly independent.
Observe that, as $\psi_{K_2}(t)\leq 1$ for all $t\leq 1$, the event $\Event$ is contained inside the event
$$
\Event':=\Big\{\forall j\leq \ell\, : \, \,
a_{ij}\geq 2/K_2\mbox{ for at least $m/2$ indices $i$}\Big\},
$$
where $a_{ij}:=\Big|\frac{1}{\lfloor n/m\rfloor} \sum_{w\in S_i}\exp(2\pi{\bf i}\,\eta[I_w]\, s_j/k)\Big|$,
$i\leq m$, $j\leq \ell$. Denoting $b_{ij}=1$ if $a_{ij}\geq 2/K_2$ and $b_{ij}=0$ otherwise and using
a simple counting argument for the matrix $\{b_{ij}\}_{ij}$, we obtain that
$$
\Event\subset\Event'\subset \Event'':=
\Big\{\Big|\Big\{i\,:\,\,\, a_{ij}\geq 2/K_2\,\,\, \mbox{ for at least $\ell/4$ indices $j$}\Big\}\Big|\geq m/4\Big\}.
$$
To estimate $\Prob(\Event'')$ we use Lemma~\ref{l: aux 9871039481} with $\eps^5$.
Note that $\eps^5\leq \min(2/K_2, 1/2)$, and that by  our choice of $R$, for any $j\neq q$ we have
$ |I_w|\,|s_j-s_q|\geq k\,|s_j-s_q|\geq R_{\text{\tiny\ref{l: aux 9871039481}}}(\varepsilon^5)k$,
while $|s_j-s_q|\leq k/2$.
Thus,
$$
\forall i\leq m \, : \quad \Prob\Big\{ a_{ij} \geq 2/K_2\,\,\, \mbox{ for at least $\ell/4$ indices $j$}
\Big\}\leq \varepsilon^{5\lfloor n/m\rfloor}.
$$
Hence,
$$
\Prob(\Event'')\leq {m\choose \lceil m/4\rceil}\varepsilon^{5\lfloor n/m\rfloor\,m/4}
\leq 2^m\varepsilon^{5\lfloor n/m\rfloor\,m/4}\leq \varepsilon^n,
$$
which completes the proof.
\end{proof}

\begin{lemma}[Very small product everywhere except for a set of measure $O(1)$]\label{l: aux 2398205987305}
For any $\varepsilon\in(0,1/2]$ there are
$R_{\text{\tiny\ref{l: aux 2398205987305}}}
=R_{\text{\tiny\ref{l: aux 2398205987305}}}(\varepsilon)\geq 1$,
$\ell=\ell_{\text{\tiny\ref{l: aux 2398205987305}}}(\varepsilon)\in\N$
and $n_{\text{\tiny\ref{l: aux 2398205987305}}}=n_{\text{\tiny\ref{l: aux 2398205987305}}}(\varepsilon,K_3)\in\N$ with the following property.
Let $k,m,n\in\N$, $n\geq n_{\text{\tiny\ref{l: aux 2398205987305}}}$, $k\leq 2^{n/\ell}$,
$n/m\geq \ell$, and $4\leq K_2\leq 2/\varepsilon$.
Let $X=(X_1,\dots,X_n)$ be a random vector
uniformly distributed on $\Lambda_n$.
Fix disjoint subsets $S_1,\dots,S_m$ of $[n]$, each of cardinality $\lfloor n/m\rfloor$.
Then
$$
\Prob\bigg\{\Big|\Big\{s\in[0,k/2]:\;\prod\limits_{i=1}^m\psi_{K_2}\bigg(\Big|\frac{1}{\lfloor n/m\rfloor}
\sum_{w\in S_i}\exp(2\pi{\bf i}X_w s)\Big|\bigg)\geq (K_2/4)^{-m/2}\Big\}\Big|\leq R_{\text{\tiny\ref{l: aux 2398205987305}}}\bigg\}
\geq 1-(\varepsilon/2)^n.
$$
\end{lemma}
\begin{proof}
Fix any $\varepsilon\in(0,1/2]$, and define $\widetilde\varepsilon:=\varepsilon^{3/2}/32$,
$\widetilde\ell:=\ell_{\text{\tiny\ref{l: aux 0876958237}}}(\widetilde\varepsilon)$,
$\ell:=2\widetilde\ell$,
and $R:=4 R_{\text{\tiny\ref{l: aux 0876958237}}}(\widetilde\varepsilon)
\ell_{\text{\tiny\ref{l: aux 0876958237}}}(\widetilde\varepsilon)> 1$.

Assume that the parameters $k,m,n$ and $S_1,\dots,S_m$ satisfy the assumptions of the lemma.
In particular, we assume that $n$ is large enough so that $(8K_2\pi n)^{\widetilde \ell}\leq 2^n$ and $\gfn(n)^{\widetilde\ell}\leq 2^n$.
Denote
$$
   \beta:=(8K_2\pi m\gfn(n))^{-1}(2K_2)^{-m/2}  \qand
   a_{ij}:=\Big|\frac{1}{\lfloor n/m\rfloor} \sum_{w\in S_i}\exp(2\pi{\bf i}\,\eta[I_w]\, s_j/k)\Big|,
   \, \, i\leq m,\, j\leq \widetilde\ell .
$$
Let $T:=[0,k/2]\cap\,\beta\Z$. By Lemma~\ref{l: aux 0876958237}
 for any collection $s_1,\dots,s_{\widetilde\ell}$ of points from $T$ satisfying $|s_j-s_q|\geq
R_{\text{\tiny\ref{l: aux 0876958237}}}(\widetilde\varepsilon)$
for all $j\neq q$, we have
$$
\Prob\bigg\{\forall j\leq \widetilde\ell \,:\,\, \, \prod\limits_{i=1}^m\psi_{K_2}(a_{ij})
\geq (K_2/2)^{-m/2}\bigg\}
\leq \widetilde\varepsilon^{\,n}.
$$
Taking the union bound over all possible choices of $s_1,\dots,s_{\widetilde\ell}$ from $T$, we get
\begin{equation}\label{eq: aux 2058032958}
\begin{split}
\Prob\bigg\{\prod\limits_{i=1}^m\psi_{K_2}(a_{ij})
&\geq (K_2/2)^{-m/2}\mbox{ for all $j\leq {\widetilde\ell}$ and for
some $s_1,\dots,s_{\widetilde\ell}\in T$}\\
&\mbox{with $|s_p-s_q|\geq R_{\text{\tiny\ref{l: aux 0876958237}}}(\widetilde\varepsilon)$
for all $p\neq q$}\bigg\}
\leq \widetilde\varepsilon^{\,n}|T|^{\widetilde\ell}.
\end{split}
\end{equation}
Further, in view of Lemma~\ref{l: lip of prod},
for any realization of $X_w$'s
the product
$$
 f(s):=\prod\limits_{i=1}^m \psi_{K_2}\bigg(\Big|\frac{1}{\lfloor n/m\rfloor}
 \sum_{w\in S_i}\exp(2\pi{\bf i}X_w s)\Big|\bigg),
$$
viewed as a function of $s$, is $(8 K_2\pi \gfn(n)m)$-Lipschitz.
This implies that for any pair $(s,s')\in\R_+^2$,
satisfying $|s-s'|\leq \beta$, we have
$$
 f(s)\geq (K_2/2)^{-m/2}\quad \quad  \mbox{whenever}\quad
 f(s')\geq (K_2/4)^{-m/2}.
$$
Moreover,  for any collection $s_1',\dots,s_{\widetilde\ell}'$ of numbers from $[0,k/2]$
satisfying $|s_p'-s_q'|\geq 2R_{\text{\tiny\ref{l: aux 0876958237}}}(\widetilde\varepsilon)$ for all
$p\neq q$ there are numbers $s_1,\dots,s_{\widetilde\ell}\in T$ with $|s_q-s_q'|\leq \beta$
$|s_p-s_q|\geq R_{\text{\tiny\ref{l: aux 0876958237}}}(\widetilde\varepsilon)$ for all $p\neq q$
(we used also $2\beta\leq  1\leq R_{\text{\tiny\ref{l: aux 0876958237}}}(\widetilde\varepsilon)$).
This, together with \eqref{eq: aux 2058032958}, yields
\begin{align*}
  \Prob\bigg\{&\prod\limits_{i=1}^m\psi_{K_2}\bigg(\Big|\frac{1}{\lfloor n/m\rfloor}
    \sum_{w\in S_i}\exp(2\pi{\bf i}X_w s_j')\Big|\bigg)\geq (K_2/4)^{-m/2}\mbox{ for all
    $j\leq {\widetilde\ell}$ and some $s_1',\dots,s_{\widetilde\ell}'\in [0,k/2]$}\\
  &\mbox{with $|s_p'-s_q'|\geq 2R_{\text{\tiny\ref{l: aux 0876958237}}}(\widetilde\varepsilon)$
       for all $p\neq q$}\bigg\}\\
  &\hspace{1cm}\leq \widetilde\varepsilon^{\,n}|T|^{\widetilde\ell}\leq \widetilde\varepsilon^{\,n}\,
    (k/\beta)^{\widetilde\ell}
  \leq \widetilde\varepsilon^{\,n} \,2^n\, (8K_2\pi m\gfn(n))^{\widetilde\ell} (2K_2)^{m\widetilde\ell/2}\\
  &\hspace{1cm}\leq \widetilde\varepsilon^{\,n}\, 8^n\,(4/\varepsilon)^{m\widetilde\ell/2}\,
   \leq \widetilde\varepsilon^{\,n}\, \varepsilon^{-n/2} \,16^n\,
    \leq (\varepsilon/2)^n.
\end{align*}
The event whose probability is estimated above, clearly contains the event in question ---
$$
\bigg\{\Big|\Big\{s\in[0,k/2]:\;\prod\limits_{i=1}^m\psi_{K_2}\Big(\Big|\frac{1}{\lfloor n/m\rfloor}
\sum_{w\in S_i}\exp(2\pi{\bf i}X_w s)\Big|\Big)\geq (K_2/4)^{-m/2}\Big\}\Big|
\geq 4 R_{\text{\tiny\ref{l: aux 0876958237}}}(\widetilde\varepsilon){\widetilde\ell}\bigg\}.
$$
This, and our choice of parameters, implies the result.
\end{proof}

\begin{lemma}[Moderately small product for almost all $s$]\label{l: aux 20985059837}
For any $\varepsilon\in(0,1]$ and $z\in(0,1)$ there are $\varepsilon'=\varepsilon'(\varepsilon)\in(0,1/2]$,
$n_{\text{\tiny\ref{l: aux 20985059837}}}
=n_{\text{\tiny\ref{l: aux 20985059837}}}(\varepsilon,z)\geq 10$,
and $C_{\text{\tiny\ref{l: aux 20985059837}}}=C_{\text{\tiny\ref{l: aux 20985059837}}}(\varepsilon,z)\geq 1$
with the following property.
Let $n\geq n_{\text{\tiny\ref{l: aux 20985059837}}}$, $2^n\geq k\geq 1$,
$C_{\text{\tiny\ref{l: aux 20985059837}}}\leq m\leq n/4$, and  $4\leq K_2\leq 1/\varepsilon$.
Let $X=(X_1,\dots,X_n)$ be a random vector
uniformly distributed on $\Lambda_n$.
Fix disjoint subsets $S_1,\dots,S_m$ of $[n]$, each of cardinality $\lfloor n/m\rfloor$.
Then
$$
\Prob\bigg\{ \forall s\in [z,\varepsilon' k] \, : \, \, \,\prod\limits_{i=1}^m\psi_{K_2}\bigg(\Big|\frac{1}{\lfloor n/m\rfloor}
\sum_{w\in S_i}\exp(2\pi{\bf i}X_w s)\Big|\bigg)\leq e^{-\sqrt{m}}\bigg\}
\geq 1-(\varepsilon/2)^n.
$$
\end{lemma}
\begin{proof}
Let $\varepsilon'>0$ will be chosen later. Fix any $s\in [z,\varepsilon ' k]$.
Assume $m\geq (\eps' z)^{-4}\geq 10$.
For $i\leq m$ denote 
$$
 \gamma _i(s):=\Big|\frac{1}{\lfloor n/m\rfloor}\sum_{w\in S_i}\exp(2\pi{\bf i}X_w s)\Big|,
 \quad \quad f_i(s):=\psi_{K_2}\big( \gamma _i(s)\big),
\quad \mbox{ and  }  \quad
f(s):=\prod\limits_{i=1}^m f_i(s)
$$
Observe that  by the definition of $\psi_{K_2}$ for each $i\leq m$ we have
$f_i(s)=\gamma _i(s)$, 
provided $\gamma _i(s)\geq 1/K_2$. 
Next note that if for some complex unit numbers $z_1 ,..., z_N$ their average
$v:=\sum _{i=1}^N z_i /N$ has length $1-\alpha>0$ then, taking the unit complex number
$z_0$ satisfying $\la z_0, v\ra =|v|$ we have 
$$N(1-\alpha) \leq \sum _{i=1}^N Re \la z_i, v\ra \leq N,$$
therefore there are at least $N/2+1$ indices $i$ such that $Re \la z_i, v\ra \geq 1-4\alpha$.
This in turn implies that there exists an index $j$ such that there are at least $N/2$
indices $i$ with  $Re \la z_i, \bar z_j\ra \geq 1-16\alpha$. Thus that the event
$\big\{f_i(s)\geq 1-\frac{2}{\sqrt{m}}\big\}$ is contained in the event
$$
  \Big\{\exists\;\; w'\in S_i :\;\;\; \cos(2\pi  s(X_w- X_{w'}))\geq
  1-\frac{32}{\sqrt{m}}\;\;\mbox{for at least } \, \frac{n}{2m}\,
  \mbox{ indices \, }w\in S_i\setminus\{w'\}\Big\}.
$$
 To estimate the  probability of the later event, we take the union bound over all choices of
$n/(2m)$ indices from $S_i$, and over all choices of $w'$. We then get
\begin{align*}
\Prob\bigg\{ f_i(s)
 \geq 1-\frac{2}{\sqrt{m}}\bigg\}
 &\leq \frac{n}{m}\,2^{\lfloor n/m\rfloor}\,
 \max\limits_{w'\in S_i,\,F\subset S_i\setminus\{w'\},\atop |F|\geq n/(2m)}
 \Prob\bigg\{\forall w\in F :\,\, \dist(s(X_w-X_{w'}),\Z)\leq \frac{2}{m^{1/4}}\bigg\}
\end{align*}
To estimate the probability under maximum we use the definition of $\Lambda_n$ and
 independence of coordinates of the vector $X$. Note that for each fixed $w$ there
 is an integer interval $I_w$ of the length at least $2k$ such that $X_w$ is uniformly
 distributed on $I_w/k$. Therefore, fixing a realization $X_{w'}=b/k$, $b\in \Z$,
 we need to count how many $a\in I_w$ are such that $s(a-b)/k$ is close to an integer.
 This can be done by splitting $I_w$ into subintervals of length $k$ and considering cases
 $z\leq s\leq 1$, $1<s\leq   C'k/m^{1/4}$ (this case can be empty), and
 $C'k/m^{1/4}<s\leq  \eps' k$.
This leads to the following bound with an absolute constant $C''>0$,
\begin{align*}
\Prob\bigg\{
 f_i(s)
 \geq 1-\frac{2}{\sqrt{m}}\bigg\}
 &\leq
 \frac{n}{m}\,2^{n/m}\,\bigg(\max\Big(\frac{C''}{z\,m^{1/4}},C''\varepsilon'\Big)\bigg)^{n/(2m)}
 \leq \frac{n}{m}\,\big(4 C''\varepsilon'\big)^{n/(2m)}
\end{align*}
%

Using this estimate and the fact that
 $\psi_{K_2}(t)\leq 1$ for
$t\leq 1$ (so, each $f_i (s)\leq 1$), we obtain
\begin{align*}
\Prob\bigg\{ f(s)
\geq \Big(1-\frac{2}{\sqrt{m}}\Big)^{3m/4}\bigg\}
&\leq
\Prob\bigg\{  f_i(s)
\geq 1-\frac{2}{\sqrt{m}}\;\;\mbox{for at least $m/4$ indices $i$}
\bigg\}\\
&\leq 2^m  \bigg(\frac{n}{m}\,\big(4 C''\varepsilon'\big)^{n/(2m)}\bigg)^{m/4}=
  \bigg(\frac{16n}{m}\bigg)^{m/4}
\big(4 C''\varepsilon'\big)^{n/8}.
\end{align*}

The last step of the proof is somewhat similar to the one used in the proof of
 Lemma~\ref{l: aux 2398205987305} --- we discretize the interval $[z,\varepsilon ' k]$ 
 and use the Lipschitzness  $f(s)$. Recall that $\gfn(n)\leq 2^n$ and thus, by 
 Lemma~\ref{l: lip of prod}, $f(s)$ is $(8 K_2\pi 2^n\,m)$-Lipschitz. 
 Let 
$$ 
 \beta:= \big(1-2/\sqrt{m}\big)^{3m/4}\big(8 K_2\pi \,2^n m \big)^{-1}
 \quad \qand \quad  T:=[z,\varepsilon ' k]\cap \beta\Z.
$$
Then for any $s,s'\in [z,\varepsilon ' k]$
satisfying $|s-s'|\leq \beta$ we have 
$|f(s)-f(s')|\leq \big(1-2/\sqrt{m}\big)^{3m/4}$ 
deterministically.
%
This implies that 
\begin{align*}
    \Prob\bigg\{\forall s \in[z,\varepsilon'k]\, :\,\,\, 
    f(s)&\leq 2\Big(1-\frac{2}{\sqrt{m}}\Big)^{3m/4}\bigg\}
 \geq\Prob\bigg\{\forall s \in T\, :\,\,\, 
    f(s)\leq \Big(1-\frac{2}{\sqrt{m}}\Big)^{3m/4}\bigg\}
 \\&\geq 1-\frac{k}{\beta}\bigg(\frac{16n}{m}\bigg)^{m/4}
     \big(4 C''\varepsilon'\big)^{n/8}
\geq 1-(\varepsilon/2)^n,
\end{align*}
provided that $\varepsilon':=c''\varepsilon^{8}$ for a sufficiently small universal constant $c''>0$.
\end{proof}

\begin{lemma}\label{l: aux 98508746104921-4}
Let $\rho, \varepsilon\in(0,1]$,  $k\geq 1$, $h\in\R$, $a_1\geq h+1$, $a_2\leq h-\rho-1$. 
Let $Y_1,Y_2$ be independent random variables,
with $Y_1$ uniformly distributed on $[h,a_1]\cap \frac{1}{k}\Z$ and
$Y_2$ uniformly distributed on $[a_2,h-\rho]\cap \frac{1}{k}\Z$.
Then for every $s\in  [-\eps/8, \eps/8]$ one has
$$
\Prob\big\{\big|\exp\big(2\pi{\bf i}Y_{1} s\big)+\exp\big(2\pi{\bf i}Y_{2} s\big)\big|> 2
- 2\pi \rho^2 s^2\big\}\leq \varepsilon.
$$
\end{lemma}
\begin{proof}
Clearly, it is enough to consider $0<s<\eps/8$ only. Note that  
$$
 \big|\exp\big(2\pi{\bf i}Y_{1} s\big)+\exp\big(2\pi{\bf i}Y_{2} s\big)\big|= 
 \big|1+\exp\big(2\pi{\bf i}(Y_1-Y_{2}) s\big)\big|= 2\big|\cos\big(\pi{\bf i}(Y_1-Y_{2}) s\big)\big|.
$$
We consider two cases.

\smallskip
\noindent 
{\it Case 1. } $a_1\leq h+2\varepsilon^{-1}$ and $a_2\geq h-2\varepsilon^{-1}$.
In this case, deterministically, $\rho \leq Y_1-Y_{2}\leq 4/\eps$, therefore, 
using that $\cos t\leq 1-t^2/\pi$ on $[-\pi/2, \pi/2]$,  we have 
for every $s\in(0,\varepsilon/8]$,
$$
 \big|\exp\big(2\pi{\bf i}Y_{1} s\big)+\exp\big(2\pi{\bf i}Y_{2} s\big)\big|\leq 2-2\pi \rho^2s^2.
$$

\smallskip
\noindent
{\it Case 2. }  Either $a_1> h+2\varepsilon^{-1}$ or $a_2< h-2\varepsilon^{-1}$.
Without loss of generality, we will assume the first inequality holds.
We condition on a realization $\widetilde Y_2$ of $Y_2$
(further in the proof, we compute conditional probabilities given $Y_2=\widetilde Y_2$).
For any $s\leq \varepsilon/8$, the event
$$
\big\{\big|1+\exp\big(2\pi{\bf i}(Y_1-\widetilde Y_{2}) s\big)\big|\geq 2-s^2\big\}
$$
is contained inside the event
$$
\big\{\dist\big((Y_1-\widetilde Y_{2}) s,\Z\big)\leq s\big\}.
$$
On the other hand, since $(Y_1-\widetilde Y_{2})s$ is uniformly distributed on a set
$[b_1,b_2]\cap \frac{s}{k}\Z$, for some $b_2\geq b_1+2\varepsilon^{-1} s$,
the probability of the last event is less than $\varepsilon$. The result follows.

\end{proof}

\begin{lemma}[Integration for small $s$]\label{l: aux -29802609872}
For any $\widetilde\varepsilon\in(0,1]$, $\rho\in(0,1/4]$ and $\delta\in (0,1/2]$ there are 
$n_{\text{\tiny\ref{l: aux -29802609872}}}
=n_{\text{\tiny\ref{l: aux -29802609872}}}(\widetilde\varepsilon,\delta,\rho)$,
$C_{\text{\tiny\ref{l: aux -29802609872}}}=C_{\text{\tiny\ref{l: aux -29802609872}}}(\widetilde\varepsilon,\delta,\rho)\geq 1$, and
$K_{\text{\tiny\ref{l: aux -29802609872}}}=K_{\text{\tiny\ref{l: aux -29802609872}}}(\delta,\rho)\geq 1$
with the following property.
Let $A_{nm}$ be defined as in (\ref{anm}),  $n\geq n_{\text{\tiny\ref{l: aux 20985059837}}}$, $k\geq 1$,
$m\in\N$ with $n/m\geq C_{\text{\tiny\ref{l: aux -29802609872}}}$ and $m\geq 2$, and let $X=(X_1,\dots,X_n)$ be a random vector
uniformly distributed on $\Lambda_n$. 
Then for every  $K_2\geq 4$,
\begin{align*}
\Prob\bigg\{&A_{nm} \,
\sum\limits_{S_1,\dots,S_m}\;
\int\limits_{-\sqrt{m}/C_{\text{\tiny\ref{l: aux -29802609872}}}}^{\sqrt{m}/C_{\text{\tiny\ref{l: aux -29802609872}}}}
\prod\limits_{i=1}^{m}\psi_{K_2}\bigg(\Big|\frac{1}{\lfloor n/m\rfloor}
\sum_{w\in S_i}\exp\big(2\pi{\bf i}X_{w} m^{-1/2}\,s\big)\Big|\bigg)\,ds\geq K_{\text{\tiny\ref{l: aux -29802609872}}}
\bigg\}\leq (\widetilde\varepsilon/2)^n,
\end{align*}
where the sum is taken over all disjoint subsets $S_1,\dots,S_m\subset[n]$ of cardinality $\lfloor n/m\rfloor$ each.
\end{lemma}

\begin{proof}
Let $n_\delta,C_\delta,c_\delta$, and $\mathcal S$ be as in  Lemma~\ref{l: aux 2498276098059385-}). 
 A given choice of subsets $(S_1,\dots,S_m)\in\mathcal S$  denote
$$
 \gamma _i(s):=\Big|\frac{1}{\lfloor n/m\rfloor}\sum_{w\in S_i}\exp(2\pi{\bf i}X_w s)\Big|, 
 \quad \quad f_i(s):=\psi_{K_2}\big( \gamma _i(s)\big),
\quad \mbox{ and  }  \quad
f(s):=\prod\limits_{i=1}^m f_i(s) 
$$
(note that functions $\gamma _i(s)$, $f_i(s)$, $f(s)$ depend on the 
choice of subsets $S_i$).

\smallskip 

First, we study the distribution of the variable $f(s)$
for a given choice of subsets $S_i$. 
We  assume that $n\geq n_\delta$ and $n/m\geq C_\delta$.
We also denote  $\varepsilon:=2^{-10/\delta}\, \widetilde\varepsilon^{\, 16/\delta c_\delta}$
and 
\begin{align*}
\mathcal S':=
\Big\{&(S_1,\dots,S_m)\in \mathcal S:\;
\min(|S_i\cap Q_1|,|S_i\cap Q_2|)\geq \delta \lfloor n/m\rfloor/2 
\mbox{ for at least $c_\delta m$ indices $i$}
\Big\}.
\end{align*}
Fix a sequence $(S_1,\dots,S_m)\in \mathcal S'$, and $J\subset[m]$ be a subset of cardinality $\lceil c_\delta m\rceil$
such that 
$$
   \forall i\in J\, : \,\,\min(|S_i\cap Q_1|,|S_i\cap Q_2|)\geq \delta \lfloor n/m\rfloor/2 .
$$

For any $i\in J$, $w_1\in S_i\cap Q_1$, and $w_2\in S_i\cap Q_2$ by Lemma~\ref{l: aux 98508746104921-4} 
we have for $s\in [-\eps/8, \eps/8]$,
$$
\Prob\big\{\big|\exp\big(2\pi{\bf i}X_{w_1} s\big)+\exp\big(2\pi{\bf i}X_{w_2} s\big)\big|\geq 2
-2\pi\rho^2 s^2\big\}\leq \varepsilon.
$$
Within $S_i$, we can find at least $\frac{\delta}{2}\lfloor n/m\rfloor$ disjoint pairs 
of indices $(w_1,w_2)\in Q_1\times Q_2$
satisfying the above condition. Let $T$ be a set of such pairs with $|T|=\frac{\delta}{2}\lfloor n/m\rfloor$. 
Using the independence of coordinates of $X$, and denoting  $z:=\min\big(\sqrt{1/(\pi\rho^2\delta)},  
\eps/8\big)$, we obtain for every $s\in[-z,z]$, 
\begin{align*}
\Prob\bigg\{&\gamma_i(s)\geq 1-\frac{\pi\rho^2\delta s^2}{2}\bigg\}\\
&\leq
\Prob\big\{\big|\exp\big(2\pi{\bf i}X_{w_1} s\big)+\exp\big(2\pi{\bf i}X_{w_2} s\big)\big|\geq 2
-2\pi\rho^2 s^2\mbox{ for at least $\frac{\delta}{4}\lfloor n/m\rfloor$ pairs $(w_1,w_2)\in T$}\big\}\\
&\leq 2^{\delta\lfloor n/m\rfloor/2}\,\varepsilon^{\delta\lfloor n/m\rfloor/4}
\leq (4\varepsilon) ^{\delta n/(4m)}.
\end{align*}
Applying this for all $i\in J$ together with observations $f(s)\leq 1$ and $f_i(s)=\gamma_i(s)$ 
(when $\gamma_i(s)\geq 1/K_2$),  we conclude that for every $s\in[-z,z]$,
\begin{align*}
\Prob\bigg\{f(s)\geq \big(1-\pi\rho^2\delta s^2/2\big)^{|J|/2}\bigg\}
&\leq \Prob\bigg\{
f_i(s)\geq 1-\pi\rho^2\delta s^2/2\, \, 
\mbox{ for at least $|J|/2$ indices $i\in J$}\bigg\}\\
&\leq 2^{|J|}\,(4\varepsilon) ^{\delta |J|n/(8m)}
\end{align*}

At the next step, we apply the Lemma~\ref{l: int markov} with $\xi(s)=f(s)$ to obtain
from the previous relation
$$
\Prob\bigg\{\int\limits_{-z}^{z}
f(s)\,ds
\leq \int\limits_{-z}^{z}
\bigg(1-\frac{\pi\rho^2\delta s^2}{2}\bigg)^{|J|/2}\,ds
+m^{-1/2}\bigg\}\geq 1-2z m^{1/2}\,2^{|J|}\,(4\varepsilon) ^{\delta |J|n/(8m)}.
$$

 Next we apply Lemma~\ref{l: sum markov}) with $I=\mathcal S'$, $\xi_i=f(s)$ 
 (recall that $f(s)$ depends also on the choice of $(S_1,\dots,S_m)\in\mathcal S$). 
 We obtain
\begin{align*}
\Prob\bigg\{A_{nm}\,
\sum\limits_{(S_1,\dots,S_m)\in\mathcal S'}\;\int\limits_{-z}^{z}
 f(s)\,ds
\leq \int\limits_{-z}^{z}
\bigg(1-\frac{\pi\rho^2\delta s^2}{2}\bigg)^{|J|/2}\,ds+2m^{-1/2}\bigg\}
\geq 1-2z m\,2^{|J|}\,(4\varepsilon) ^{\delta |J|n/(8m)}.
\end{align*}
Further, since by Lemma~\ref{l: aux 2498276098059385-} we have 
$|\mathcal S'|\geq (1-e^{-c_\delta n})|\mathcal S|$ and since $f(s)\leq 1$, we observe 
that 
\begin{align*}
&A_{nm}\,
\sum\limits_{(S_1,\dots,S_m)\in \mathcal S\setminus\mathcal S'}\;\int\limits_{-z}^{z}
f(s)\, ds
\leq 2z\,e^{-c_\delta n}
\end{align*}
deterministically.
Recalling that $|J|=\lceil c_\delta m\rceil$,
we obtain
\begin{align*}
\Prob\bigg\{&
A_{nm}\,
\sum\limits_{(S_1,\dots,S_m)\in\mathcal S}\;\int\limits_{-z}^{z}
f(s)\, ds
\leq C'' m^{-1/2} \bigg\}
\geq 1-2z m\,2^{|J|}\,(4\varepsilon) ^{\delta |J|n/(8m)}
\geq 1-(\widetilde \varepsilon/2)^n,
\end{align*}
for some $C''\geq 1$ depending only on $\delta$ and $\rho$, provided that 
$n\geq n_0(\widetilde\varepsilon,\delta,\rho)$. 
The result follows by the  substitution $s= m^{-1/2}u$ in the integral.
\end{proof}

As a combination of Lemmas~\ref{l: aux 2398205987305},~\ref{l: aux 20985059837}
and~\ref{l: aux -29802609872}, we obtain Proposition~\ref{prop: 09582593852}.

\begin{proof}[Proof of Proposition~\ref{prop: 09582593852}]
As we mentioned at the beginning of this subsection, we 
fix $\rho,\delta\in(0,1/4]$, a growth function $\gfn$ satisfying \eqref{gfncond}, 
a permutation $\sigma\in\Pi_n$, a number $h\in\R$,  two sets $Q_1,Q_2\subset[n]$
such that $|Q_1|,|Q_2|=\lceil \delta n\rceil$, and we use 
 $\Lambda_n$ for the set $\Lambda_n(k,\gfn,Q_1,Q_2,\rho,\sigma,h)$
defined in \eqref{eq: param l def}. We also fix $\eps\in (0,1/4]$.

We start by selecting the parameters.
Assume that $n$ is large enough.
Set $\ell:=\ell_{\text{\tiny\ref{l: aux 2398205987305}}}(\varepsilon)$.
Let $\varepsilon'=\varepsilon'(\varepsilon)$ be taken from Lemma~\ref{l: aux 20985059837}.
Set $z:=1/C_{\text{\tiny\ref{l: aux -29802609872}}}(\varepsilon,\delta,\rho)$.
Fix an integer $m\in [C_{\text{\tiny\ref{l: aux 20985059837}}}(\varepsilon,z),
n/\max(\ell,C_{\text{\tiny\ref{l: aux -29802609872}}})]$ satisfying the condition
$R_{\text{\tiny\ref{l: aux 2398205987305}}}\sqrt{m}\,e^{-\sqrt{m}}\leq 1$,
and take $1\leq k\leq \min\big(2^{n/\ell},(K_2/8)^{m/2}\big)$. 
Let $A_{nm}$ be defined as in (\ref{anm}). 
We assume that $h$ is chosen in such a way that the set $\Lambda_n$
is non-empty. As before $X$ denotes the random vector uniformly distributed on 
$\Lambda_n$. 
Let  $\mathcal S$ be as in  Lemma~\ref{l: aux 2498276098059385-}).
 A given choice of subsets $(S_1,\dots,S_m)\in\mathcal S$  denote
$$
 f(s)=f_{S_1,\dots,S_m}(s):=\prod\limits_{i=1}^{m}\psi_{K_2}\bigg(\Big|\frac{1}{\lfloor n/m\rfloor}
\sum_{w\in S_i}\exp\big(2\pi{\bf i}X_{w} m^{-1/2}\,s\big)\Big|\bigg).
$$
We have
\begin{align*}
A_{nm}\sum\limits_{S_1,\dots,S_m}\;
&\int\limits_{-\varepsilon' m^{1/2}k}^{\varepsilon' m^{1/2}k}
f(s)\, ds 
=
A_{nm}\sum\limits_{S_1,\dots,S_m}\;
\int\limits_{-z \sqrt{m}}^{z \sqrt{m}}  f(s)\, ds
+2 A_{nm}\sum\limits_{S_1,\dots,S_m}\;
\int\limits_{z \sqrt{m}}^{\varepsilon'k\sqrt{m}}
f(s)\, ds 
\end{align*}
In view of Lemma~\ref{l: aux -29802609872},
with probability at least $1-(\varepsilon/2)^n$ the first summand is bounded above by
$K_{\text{\tiny\ref{l: aux -29802609872}}}$.
To estimate the second summand, we  combine Lemmas~\ref{l: aux 2398205987305} and~\ref{l: aux 20985059837}
(we assume that $z \leq\varepsilon' k$ as otherwise there is no second summand).
Fix for a moment a collection $(S_1,\dots,S_m)\in\mathcal S$.
By Lemma~\ref{l: aux 2398205987305}, with probability at least $1-(\varepsilon/2)^n$
the function $f$ on $[0,k\sqrt{m}/2]$ is bounded above by $(K_2/4)^{-m/2}$
for all points $s$ outside of some set of measure at most $R_{\text{\tiny\ref{l: aux 2398205987305}}}\sqrt{m}$
(note that we apply variable transformation $s\to m^{-1/2}s$ to use the lemma here).
Further, by Lemma~\ref{l: aux 20985059837}, with probability at least $1-(\varepsilon/2)^n$ we have
that $f$ is bounded above by $e^{-\sqrt{m}}$ for all $s\in[z \sqrt{m},\varepsilon' k\sqrt{m}]$.
Thus, with probability at least $1-2(\varepsilon/2)^n$,
$$
\int\limits_{z \sqrt{m}}^{\varepsilon' k\sqrt{m}}f(s)\,ds\leq \sqrt{m}k\,\Big(\frac{K_2}{4}\Big)^{-m/2}+
R_{\text{\tiny\ref{l: aux 2398205987305}}}\sqrt{m}\,e^{-\sqrt{m}}.
$$
Applying Lemma~\ref{l: sum markov} with $I=\mathcal S$ and $\xi_i=f(s)$, we  obtain that
\begin{align*}
& A_{nm}\sum\limits_{S_1,\dots,S_m}\;
\int\limits_{z \sqrt{m}}^{\varepsilon'k\sqrt{m}}
f(s)\, ds \leq \sqrt{m}k\,\Big(\frac{K_2}{4}\Big)^{-m/2}+
R_{\text{\tiny\ref{l: aux 2398205987305}}}\sqrt{m}\,e^{-\sqrt{m}}+1\leq 3
\end{align*}
with probability at least $1-2(\varepsilon/2)^n$.
Thus,
taking $K_1:=K_{\text{\tiny\ref{l: aux -29802609872}}}+3$, we obtain
$$\Prob\{\bal_n(X,m,K_1,K_2)\geq \varepsilon' m^{1/2}k\}\geq 1-3( \varepsilon/2)^n\geq 1-3 \varepsilon^n.$$
\end{proof}

\section{Complement of gradual non-constant vectors: constant $p$}
\label{steep:constant p}

In this section, we study the problem of invertibility of the \Ber matrix $M$ over the set
$\ST$ defined by (\ref{strvect}) in the case when the parameter $p$ is a small constant.
 This setting turns out to be much simpler
than treatment of the general case $C\ln n/n\leq p\leq c$ given in the next section.
Although the results of Section~\ref{s: steep} essentially absorb the statements of this section,
we prefer to include analysis of the constant $p$ in our work, first, because it provides
a short and relatively simple illustration of our method and, second, because the estimates obtained here
allow to derive better quantitative bounds\ for the smallest singular value of $M$.

\subsection{Spliting of $\R^n$ and main statements}
\label{subs: steep vectors}

We define the following four classes of  vectors $\stt_1, \dots,  \stt_{4}$. For simplicity,
we normalize vectors with respect to the Euclidean norm.
The first class is the set of vectors with one coordinate much larger than the others, namely,
$$
  \stt _{1}= \stt _{1}(p): =\{x\in S^{n-1}\,:\,   x_{1}^{*}> 6 pn\, x_{2}^{*}\}.
$$
For the next sets we fix a parameter $\beta_p = \sqrt{p}/C_0$, where $C_0$ is the
absolute constant  from Proposition~\ref{rogozin}. Recall also that the operator
$Q$ (which  annihilates the maximal coordinate of a given vector) and the set
$U(m, \gamma)$ were introduced in Subsection~\ref{net}.
We also fix a small enough absolute positive constant $c_0$. We don't try to compute
the actual value of $c_0$, the conditions on how small $c_0$ is can be obtained
from the proofs. We further fix an integer $1\leq m\leq n$.

The second class of vectors consist of those vectors for which
the Euclidean norm dominates the maximal coordinate. 
To control cardinalities of nets (discretizations) we intersect this
class with $U(m, c_0)$, specifically, we set
$$
 \stt_2 =\stt_2(p,m):=  \stt_2' \cap U(m, c_0), \quad \mbox{ where }\quad
  \stt_2' := \left\{x\in S^{n-1}\,:\,  x\not\in \stt_{1}
    \,\, \mbox{ and } \, \,  x_{1}^{*}\leq \beta_p   \r\}.
$$

The next set is similar to $\stt _2$, but instead of comparing $x_1^*$ with
the Euclidean norm of the entire vector, we compare $x_2^*$ with $\|Qx\|$.
For a technical reason,
we need to control the magnitude of $\|Qx\|$ precisely; thus we partition the third set into subsets.
Let numbers $\lam_k$, $k\leq \ell$, be defined by
\begin{equation}\label{eq: 0495205965029385}
   \lam_1 = \frac{1}{6pn},\quad \lam_{k+1}= 3 \lam _k, \, \, k<\ell -1,
  \quad 1/3\leq \lam_{\ell-1} <1 \qand \lam _\ell =1.
\end{equation}
Clearly, $\ell \leq \ln n$.
Then for each $k\leq \ell-1$ we define
\begin{align*}
   \stt _{3,k}=\stt _{3,k}(p,m) &: =\left\{x\in S^{n-1}\,:\,  x\not\in \stt_{1} \cup  \stt_{2}', \, \,
  x_{2}^{*}\leq \beta_p  \|Qx\|
    \,\, \mbox{ and } \, \, \lam_k \leq \|Qx\| < \lam_{k+1}\r\}\cap U(m, c_0 \lam_k) .
\end{align*}
To explain the choice of $\lam_1$, note that if $x\not\in \stt_{1} \cup  \stt_{2}'$
and $\|x\|=1$, then
$x_2^*\geq x_1^*/(6pn)\geq \beta_p/(6pn)$. Thus, if in addition
$\beta_p\|Qx\| \geq x_{2}^{*}$, then $\|Qx\| \geq 1/(6pn)=\lam_1$.
We set  $$\stt _{3}=\stt _{3}(p,m):= \bigcup_{k=1}^{\ell-1}  \stt _{3,k}.$$

The fourth set covers the remaining options for vectors having a large almost constant part.
Let numbers $\mu_k$, $k\leq s$, be defined by
\begin{equation}\label{eq: 9635194-9580-98}
   \mu_1 = \frac{\beta_p}{6pn},\quad \mu_{k+1}= 3 \mu _k, \, \, k<s -1,
   \quad 1/3\leq \mu_{s-1} <1
   \qand \mu _s =1.
\end{equation}
Clearly, $s \leq \ln n$.
Then for each $k\leq s-1$ define the set $\stt _{4,k} =\stt _{4,k}(p,m)$ as
$$
  \left\{x\in S^{n-1}\,:\,  x\not\in \stt_{1} \cup  \stt_{2}',  \, \,
   x_{2}^{*}> \beta_p  \|Qx\| \,\, \mbox{ and } \, \,
   \mu_k \leq x^*_2 < \mu_{k+1}\r\}\cap U(m, c_0 \mu_k /\sqrt{\ln(e/p)}).
$$
Note that if $x\not\in \stt_{1} \cup  \stt_{2}'$
and $\|x\|=1$, then $x_2^*\geq x_1^*/(6pn)\geq \beta_p/(6pn)$, justifying the choice of $\mu_1$.
We set  $$\stt _{4}=\stt _{4}(p,m)= \bigcup_{k=1}^{\ell-1}  \stt _{4,k}.$$


\smallskip

Finally define $\stt$ as the union of these four classes,
$
   \stt=\stt(p,m):=\bigcup_{j=1}^{4} \stt_{j}.
$

\medskip

In this section we prove  two following theorems.

\begin{theor} \label{t:steep}
There exists  positive absolute constants $c, C$ such that the following holds.
Let $n$ be large enough, let $m\leq cpn/\ln(e/p)$, and $(30\ln n)/n\leq p<1/20$.
Let $M$ be an $n\times n$ \Ber random matrix. Then
\begin{equation*}
\label{Psteep}
  \Prob \Big\{\exists\;x\in \stt\, \, \, \mbox{ such that }
  \, \, \,
  \|M x\| <  \frac{ 1}{C \sqrt{n \ln(e/p)}} \, \,  \|x\|
  \Big\}   \leq  n(1-p)^n + 4e^{-1.5np},
\end{equation*}
where the set $\stt=\stt(p,m)$ is defined above.
\end{theor}

Recall that the set 
$\gncvectors_n$ was introduced in Subsection~\ref{gradnac}.
The next theorem shows that, after a proper normalization, 
the complement of $\gncvectors_n$ (taken in $\normr_n(r)$) is contained in $\stt$ 
for some choice of $r, \delta, \rho$ and for the growth function
$\gfn(t)=(2t)^{3/2}$ (clearly, satisfying (\ref{gfncond})).

\begin{theor}
\label{compl-1}
There exists an absolute (small) positive constant $c_1$ such that the following holds.
Let $q\in (0, c_1)$ be a parameter. Then there exist $n_q\geq 1$,
$r=r(q), \rho=\rho (q)\in (0,1)$ such that  for $n\geq n_q$,  $p\in (q, c_1)$, $\delta = r/3$,
$\gfn(t)=(2t)^{3/2}$, and $m=\lfloor rn \rfloor$ one has
$$
   \Big\{x/\|x\|\, \, : \, \, x\in \normr_n(r)\setminus \gncvectors_n(r,\gfn,\delta,\rho)\Big\}\subset \stt(p,m).
$$
\end{theor}



\subsection{Proof of Theorem~\ref{t:steep}}

Theorem~\ref{t:steep} is a consequence of four
lemmas that we prove in this section. Each lemma treats
one of the classes $\stt_i$, $i\leq 4$, and   Theorem~\ref{t:steep}
follows by the union bound.
Recall that $U(m, \gamma)$ was introduced in Subsection~\ref{net} and
that given $x$, we fixed one permutation, $\sigma_x$,
such that $x_i^*=|x_{\sigma_x(i)}|$ for $i\le n$. Recall also that the event
$\Event_{nrm}$ was introduced in Proposition~\ref{nettri}.

\begin{lemma} \label{st1}
Let $n\geq 1$ and  $p\in (0, 1/2]$.
Let $\Event_{sum}$ (with $q=p$)
be the event introduced in Lemma~\ref{bennett} and by $\Event_{col}\subset \Mc$
denote the subset of $0/1$ matrices with no zero columns.
Then for every $M\in \Event_{sum}\cap \Event_{col}$ and every $x\in \stt_1$,
$$
 \|M  x\| \geq \frac{1}{3\sqrt{n}}\, \|x\|.
$$
In particular,
$$
 \p\Bigl\{M\in\Mc:\;\exists x\in\stt_1 \,\, \mbox{ with }\,\,
  \|Mx\| \leq \frac{1}{3\sqrt{n}}
 \Bigr\}\leq n (1-p)^n + e^{-1.5np}.
$$
\end{lemma}

\begin{proof}
Let $\delta_{ij}$, $i,j\leq n$ be entries of $M\in \Event_{sum}\cap \Event_{col}$.
Let $\sigma=\sigma_x$.
Denote, $\ell=\sigma(1)$.
Since $M\in \Event_{col}$, there exists $s\leq n$ such that
$\delta_{s\ell}=1$. Then
\begin{align*}
  |\langle R_{s} (M),\, x \rangle|
  &=\Big| x_{\ell}
       +   \sum_{j\ne\ell} \delta _{sj} x_j
 \Big|
    \geq|x_{\ell}|- \sum_{j\ne\ell} \delta _{sj} \, x_{j}\geq
    |x_{\ell}|- \sum_{j=1}^n \delta _{sj} \,x_{n_2}^*  .
\end{align*}
Using that $M\in \Event_{sum}$ we observe that
$\sum_{j=1}^n \delta _{sj} \leq 3.5 pn$. Thus,
$$
 \|Mx\| \geq |\langle R_{s} (M),\, x \rangle| \geq  x_{1}^* -  3.5 pn  x_{n_2}^* \geq x_{1}^*/3.
$$
The trivial bound $\|x\|\leq \sqrt{n} \, x_{1}^*$ completes the first estimate.
The ``in particular" part follows by the ``moreover" part of Lemma~\ref{bennett} and since
$\p(\Event_{col})\leq n(1-p)^n$.
\end{proof}

\begin{lemma} \label{st2}
There exists a (small) absolute positive constant $c$ such that the following holds.
Let $n$ be large enough and $m\leq cn$.
Let   $(4\ln n)/n\leq p<1/2$
and $M$ be \Ber matrix.
Then
$$
 \p\Bigl(M\in \Event_{nrm} \qand \exists x\in\stt_2 \,\, \mbox{ with }\,\,
  \|Mx\| \leq \frac{\sqrt{pn}}{5C_0}
 \Bigr)\leq  e^{-2n}.
$$
\end{lemma}

\begin{proof}
By Lemma~\ref{cardnet} for $\eps\in [8c_0, 1)$ there exists an $(\eps/2)$--net
in
$V(1)\cap  U(m, c_0)$ with respect to the triple norm $|||\cdot|||$, with cardinality at most
$$
  \frac{C n^{2}}{\eps^2} \left(\frac{18 e n }{\eps m}\r)^m.
$$
Since $\stt_2\subset V(1)\cap  U(m, c_0)$,
by a standard ``projection'' trick, we can obtain from it an $\eps$--net $\mathcal{N}$
in $\stt_2$ of the same cardinality.
Let $x\in \stt_2$. 
Let $z\in \mathcal{N}$ be such that $|||x-z|||\leq \eps$.
Since on $\stt_2$ we
have $z_1^*\leq \beta_p \|z\|=\beta _p$, Proposition~\ref{rogozin} implies that
with probability at least $1-e^{-3n}$,
\begin{equation}\label{eq: 49820598207492740329}
   \|Mz\| \geq \frac{\sqrt{pn}}{3\sqrt{2} C_0} .
\end{equation}
Further, in view of Proposition~\ref{nettri}, conditioned on \eqref{eq: 49820598207492740329}
and on $\{M\in \Event_{nrm}\}$, we have
$$
 \|Mx\| \geq \|Mz\| -\|M(x-z)\| \geq \frac{\sqrt{pn}}{3\sqrt{2} C_0}
 - 100 \sqrt{pn} \eps \geq \frac{\sqrt{pn}}{5 C_0} ,
$$
where
we have  chosen $\eps = 1/(5000 C_0 )$.
Using the union bound and our choice of $\eps$, we obtain that
$$
 \p\Bigl(M\in \Event_{nrm} \qand \exists x\in\stt_2 \,\, \mbox{ with }\,\,
  \|Mx\| \leq \frac{\sqrt{pn}}{5C_0}
 \Bigr)\leq e^{-3n}|\mathcal{N}| \leq e^{-2n}
$$
for sufficiently large $n$ and provided that $c_0\leq 1/(40000  C_0)$
and $m\leq cn$ for
small enough absolute positive constant $c$. This completes the proof.
\end{proof}

\begin{rem}
Note that we used Proposition~\ref{rogozin} with the set $A=[n]$. In this
case we could use slightly easier construction for nets than the one in
Lemma~\ref{cardnet} --- we don't need to distinguish the first
coordinate in the net construction, in other words we could have only one
special direction, not two. However this would not lead
to a better estimate and in the remaining lemmas we will need
the fult strength of our construction.
\end{rem}

Next we threat the case of vectors in $\stt_3$. The proof is similar to
the proof of Lemma~\ref{st2}, but we need to remove the maximal coordinate
and to deal with remaining part of the vector. Recall that the operator $Q$
serves this purpose.

\begin{lemma} \label{st4}
There exists a (small) absolute positive constant $c$ such that the following holds.
Let $n$ be large enough, and $m\leq cpn/\ln(e/p)$, $(4\ln n)/n\leq p<1/2$.
Let $M$ be a random \Ber matrix.
Then
$$
 \p\Bigl(M\in \Event_{nrm} \qand \exists x\in\stt_3 \mbox{ with }\,\,
  \|Mx\| \leq  \frac{1}{30C_0\sqrt{pn}}
 \Bigr)\leq  e^{-2 n}.
$$
\end{lemma}

\begin{proof}
Fix $1\leq k \leq \ell - 1$.
By Lemma~\ref{cardnet} for $\eps\in [8c_0 \lam_k, \lam_{k+1})$
there exists an $(\eps/2)$--net in
$V(\lam_{k+1})\cap  U(m, c_0\lam_k)$ with respect to $|||\cdot|||$, with cardinality at most
$$
  \frac{C n^{2}}{\eps^2} \left(\frac{18 e \lam_{k+1} n }{\eps m}\r)^m
  \leq \frac{C n^{2}}{\eps^2} \left(\frac{54 e \lam_{k} n }{\eps m}\r)^m.
$$
Again using a ``projection'' trick, we can construct an $\varepsilon$--net $\mathcal{N}_k$ in $\stt_{3,k}$
of the same cardinality.
Let $x\in \stt_{3,k}$. 
Let $z\in \mathcal{N}_k$ be such that $|||x-z|||\leq \eps$.
Since on $\stt_{3,k}$ we
have $z_2^*\leq \beta_p \|Qz\|$, Proposition~\ref{rogozin} applied
with $A=\sigma_z([2, n])$ implies that
with probability at least $1-e^{-3n}$,
$$
   \|Mz\| \geq \frac{\sqrt{pn}\, \|Qz\|}{3\sqrt{2} C_0}
   \geq \frac{\sqrt{pn}\, \lam_k}{3\sqrt{2} C_0} .
$$
Conditioned on the above inequality and on the event $\{M\in \Event_{nrm}\}$,
Proposition~\ref{nettri} implies that
$$
 \|Mx\| \geq \|Mz\| -\|M(x-z)\| \geq \frac{\sqrt{pn}\, \lam_k}{3\sqrt{2} C_0}
 - 100 \sqrt{pn} \eps \geq \frac{\sqrt{pn}\,\lam_k}{5 C_0} ,
$$
where
we have  chosen $\eps = \lam_k /(5000C_0)$. Using the union bound, our choice
of $\eps$ and $\lam_k\geq 1/(6pn)$, we obtain that
$$
 P_k:=\p\Bigl(\exists x\in\stt_{3,k} \,\, \mbox{ with }\,\,
  \|Mx\| \leq \frac{\sqrt{pn}\,\lam_k}{5C_0}
 \Bigr)\leq e^{-3n}|\mathcal{N}_k | \leq e^{-2.5 n}
$$
for  large enough $n$ and for $m\leq cn$, where $c>0$ is a
small enough absolute  constant (we also assume $c_0\leq 1/(40000 C_0)$).
Since $\ell \leq \ln n$ and $\lam_k\geq \lam_1\geq 1/(6pn)$,
we obtain
$$
 \p\Bigl(\exists x\in\stt_{3} \,\, \mbox{ with }\,\,
  \|Mx\| \leq \frac{1}{30C_0\sqrt{pn}}
 \Bigr)\leq  \sum _{k=1}^{\ell-1 } P_k\leq  e^{-2p n}.
$$
This completes the proof.
\end{proof}

Finally we threat the case of vectors in $\stt_4$.

\begin{lemma} \label{st5}
There exists a (small) absolute positive constant $c$ such that the following holds.
Let $n$ be large enough and let $m\leq cpn/\ln(e/p)$, $(30\ln n)/n\leq p<1/20$.
Let $M$ be a \Ber random matrix.
Then
$$
 \p\Bigl(M\in \Event_{nrm} \qand \exists x\in\stt_4 \,\, \mbox{ with }\,\,
  \|Mx\| \leq \frac{ 1}{60C_0 \sqrt{n \ln(e/p)}}
 \Bigr)\leq  e^{-1.5pn}.
$$
\end{lemma}

\begin{proof}
Fix $1\leq k \leq s-1$.
By Lemma~\ref{cardnet} for $\eps\in [8c_0 \mu_k/\sqrt{\ln(e/p)}, \mu_{k+1})$
there exists an $(\eps/2)$--net in
$$V(\mu_{k+1}/\beta_p)\cap  U(m,  c_0 \mu_k /\sqrt{\ln(e/p)})$$ with respect to $|||\cdot|||$ with cardinality at most
$$
  \frac{C n^{2}}{\eps^2} \left(\frac{18 e \mu_{k+1} n }{\eps m \beta_p}\r)^m
  \leq \frac{C n^{2}}{\eps^2} \left(\frac{54 e \mu_{k} n }{\eps m \beta_p}\r)^m.
$$
By the projection trick, we get an $\eps$--net $\mathcal{N}_k$ in $\stt_{4,k}\subset V(\mu_{k+1}/\beta_p)\cap  U(m,  c_0 \mu_k /\sqrt{\ln(e/p)})$.

Let $x\in \stt_{4,k}$. 
Let $z\in \mathcal{N}_k$ be such that $|||x-z|||\leq \eps$.
Since on $\stt_4$ we
have $z_1^*\geq z_2^*\geq \mu_k$, Proposition~\ref{anti2} implies that
with probability at least $1-e^{-1.6np}$,
$$
    \|Mz\| \geq \frac{  \mu_k \sqrt{pn}}{7\sqrt{\ln(e/p)}}  .
$$
Conditioned on the above and on $\{M\in \Event_{nrm}\}$, Proposition~\ref{nettri} implies that
$$
 \|Mx\| \geq \|Mz\| -\|M(x-z)\| \geq \frac{\mu_k \sqrt{pn}}{7\sqrt{\ln(e/p)}}
 - C_1 \sqrt{pn} \eps \geq \frac{\mu_k \sqrt{pn}}{10\sqrt{\ln(e/p)}}  ,
$$
where
we have  chosen
$$
  \eps = \mu_k/(50C_1\sqrt{\ln(e/p)})
 \geq 8 c_0 \mu_k/ \sqrt{\ln(e/p)} ,
$$
provided that $c_0\leq 1/40000$.
Using the union bound and our choice
of $\eps$ we obtain that
$$
P_k:= \p\Bigl(M\in \Event_{nrm} \qand \exists x\in\stt_{4,k} \,\, \mbox{ with }\,\,
 \|Mx\| \leq \frac{  \mu_k \sqrt{pn}}{10\sqrt{\ln(e/p)}}
 \Bigr)\leq e^{-1.6pn}|\mathcal{N}_k| \leq e^{-1.55 p n}
$$
for large enough $n$ and for  $m\leq cpn/\ln(e/p)$, where $c>0$ is a
small enough absolute  constant.
Since $s \leq \ln n$ and
$\mu_k\geq \mu_1\geq \beta_p/(6pn)= 1/(6C_0 n\sqrt{p})$,
we obtain
$$
 \p\Bigl(M\in \Event_{nrm} \qand \exists x\in\stt_{4} \,\, \mbox{ with }\,\,
  \|Mx\| \leq \frac{ 1}{60C_0 \sqrt{n \ln(e/p)}}
 \Bigr)\leq  \sum _{k=1}^{s-1} P_k\leq  e^{-1.5p n}.
$$
This completes the proof.
\end{proof}

\begin{proof}[Proof of Theorem \ref{t:steep}.]
Lemmas~\ref{st1}, \ref{st2}, \ref{st4}, and \ref{st5} imply that
$$
  \Prob(\Event) \leq  n(1-p)^n + 3e^{-1.5np} +\p(\Event^c_{nrm}),
$$
where  $\Event$ denotes the event from Theorem \ref{t:steep}.
Lemma~\ref{mnorm} applied with $t=30$ and (\ref{normofone}) imply that
$
  \p(\Event^c_{nrm})\leq e^{-10pn},
$
provided that $pn$ is large enough. This completes the proof.
\end{proof}

\subsection{Proof of Theorem \ref{compl-1}}




\begin{proof}
We prove the statement with   $r=r(q)= cq/\ln(e/q)$, where $c$ is the constant from Theorem~\ref{t:steep},
and  $\rho= \rho(q)= c_0 \sqrt{r}\beta_q /(6\sqrt{\ln(e/q)})$. 
Note that under our choice of parameters (and assuming $c_1$ is small), $9\delta/2 \leq c_0\beta_q /\sqrt{\ln(e/q)}
\leq c_0\beta_p /\sqrt{\ln(e/p)}$.

Assume that $x\in \normr_n(r)\setminus \gncvectors_n$.
By $(x_i^{\#})_i$ denote the non-increasing  rearrangement
of $(x_i)_i$ (we would like to emphasize that  we do not take absolute values).
Note that for any $t>0$ there are two subsets $Q_1, Q_2\subset[n]$  with $|Q_1|,|Q_2|\geq \lceil\delta n\rceil$
satisfying $\max\limits_{i\in Q_2} x_i\leq \min\limits_{i\in Q_1}x_i- t$ if and only if
$x_{\lceil\delta n\rceil}^{\#}-x_{n-\lceil\delta n\rceil+1}^{\#}\geq t$. This leads to the two following cases.
\smallskip

\noindent
{\it Case 1. $x_{\lceil\delta n\rceil}^{\#}-x_{n-\lceil\delta n\rceil+1}^{\#}\geq  \rho$. }
Since $x\notin \gncvectors_n$, in this case there exists an index $j\leq n$ with $x_j^*>(2n/j)^{3/2}$. Note that since
$x^*_{\lfloor rn\rfloor}=1$, we have $j<rn=3\delta n$. 



\smallskip

\noindent
{\it Subcase 1a. $1< j<  3\delta n$. }
Since $x_j^*>(2n/j)^{3/2}$ we get
$$
  \|Qx\|^2 \geq \sum _{i=2}^j (x_i^*)^2\geq \sum _{i=2}^j (2n/i)^{3} \geq \frac{j}{2}\, (2n/j)^{3}
  =n (2n/j)^{2}.
$$
Therefore,
$$
 \frac{x^*_{\lfloor rn\rfloor+1}}{\|Qx\|}\leq
   \frac{1}{\sqrt{n}} \, \frac{j}{2n}
   \leq
   \frac{(3\delta/2) }{\sqrt{n}} .
$$
Now let $y=x/\|x\|$. Then
\begin{equation}\label{estym}
  y^*_{\lfloor rn\rfloor+1} = \frac{x^*_{\lfloor rn\rfloor+1}}{\|x\|}\leq \frac{3\delta/2}{\sqrt{n}}  \,
   \frac{\|Qx\|}{\|x\|} =
   \frac{3\delta/2}{\sqrt{n}} \, \|Qy\| .
\end{equation}

Our goal is to show that $y\in \stt(p,m)$ (with $m=\lfloor rn\rfloor$).

If $y\in \stt_1(p)$, we are done.

Otherwise, if $y\in \stt_2'$, then (\ref{estym}) implies
that $y^*_{\lfloor rn\rfloor+1}\leq c_0/\sqrt{n}$, that is, there are at least $n-m$ coordinates
at the distance at most $c_0/\sqrt n$ from zero. Thus $y\in U(m, c_0)$ and
hence $y\in \stt_{2}$.

If
$y\not\in \stt_1\cup \stt_2'$ and
$y^*_2\leq \beta_p\|Qy\|$, then necessarily $\lam _k\leq \|Qy\|<\lam_{k+1}\leq 3\lam_k$ for some $k$, where
$\lam_k,\lam_{k+1}$ are defined according to \eqref{eq: 0495205965029385}. Then
(\ref{estym}) implies that  $y^*_{\lfloor rn\rfloor+1}\leq c_0 \lam_k/\sqrt{n}$, that is,
there are at least $n-m$ coordinates at the distance at most $c_0\lam_k/\sqrt n$ from zero.
Thus $y\in U(m, c_0\lam _k)$ and hence $y\in \stt_{3,k}$.

If $y\not\in \stt_1\cup \stt_2'$ and
$y^*_2> \beta_p\|Qy\|$ then necessarily $\mu _k\leq y_2^*<\mu_{k+1}\leq 3\mu_k$,
where $\mu_k,\mu_{k+1}$ are given by \eqref{eq: 9635194-9580-98}.
Then,
similarly,
$$
  y^*_{\lfloor rn\rfloor+1}
\leq    \frac{3\delta/2}{\sqrt{n}}  \, \|Qy\|
  \leq  \frac{3\delta/2}{\sqrt{n}}  \, \frac{y_2^*}{\beta_p}
  \leq  \frac{9\delta/2}{\beta_p\sqrt{n}}  \,
  \mu_k \leq  \frac{c_0 \mu_k}{\sqrt{\ln(e/p)} \sqrt{n}} .
$$
This implies that $y\in U(m, c_0\mu _k/\sqrt{\ln(e/p)})$ and, thus, $y\in \stt_{4,k}$.

\smallskip

\noindent
{\it Subcase 1b. $j=1$. } In this case $x_1^*\geq (2n)^{3/2}$. Assume $x\not\in \stt_1$, that is
$x_1^*<6pn x_2^*$. Then
$$
 \frac{x^*_{\lfloor rn\rfloor+1}}{\|Qx\|}\leq \frac{1}{x_2^*}\leq \frac{6pn}{(2n)^{3/2}}=
   \frac{6p}{2^{3/2} \sqrt{n}}.
$$
We can now define $y:=x/\|x\|$ and, having noted that $y^*_{\lfloor rn\rfloor+1} \leq
   \frac{6p}{2^{3/2}\sqrt{n}} \, \|Qy\|$, proceed similarly to the Subcase~1a. We will need to
use the condition
$18 p \leq  2^{3/2}  c_0 \beta_p /\sqrt{\ln(e/p)}$, which holds for small enough $p$.

\smallskip

\noindent
{\it Case 2. $x_{\lceil\delta n\rceil}^{\#}-x_{n-\lceil\delta n\rceil+1}^{\#}<\rho$. }
Set $\sigma$ be a permutation of $[n]$ such that $x_{i}^{\#}=x_{\sigma(i)}$, $i\leq n$
(note that $\sigma$ is in general different from the permutation $\sigma_x$ defined in connection with the
non-increasing rearrangement of the absolute values $|x_i|$).
Define the following set, which will play the role of the set in the definition
of $U(m, \gamma)$ (see Subsection~\ref{net}),
$$A:=\{\sigma(i) :\,\, \lceil\delta n\rceil<i\leq n-\lceil\delta n\rceil\}.$$
Then $|A|=n-2\lceil\delta n\rceil$, and $m> 2\lceil\delta n\rceil=2\lceil r n/3\rceil$.
Since $x^*_m=1$, we observe
that either $x_{\lceil\delta n\rceil+1}^{\#}\geq 1$ or $x_{n-\lceil\delta n\rceil}^{\#}\leq -1$ (or both).
Moreover, since $r<1/2$, we necessarily have that $|x^{\#}_i|\leq 1$ for some $\lceil\delta n\rceil<i\leq n-\lceil\delta n\rceil$.
Therefore, there exists an index $j\in A$ such that $|x_j|=1$. Taking $b=x_j$,
we observe that for every
$i\in A$, $|x_i-b|<\rho$.
On the other hand we have
$$
  \|x\|^2\geq \|Qx\|^2\geq \sum_{i=2}^{m}
   x_i^*\geq m -1 \geq m/2 \qand
   \forall i\in A\, : \,
   \frac{|x_i - b|}{\|Q x\|}\leq \frac{\sqrt{2}\, \rho}{\sqrt{m}}\leq
   \frac{1}{\sqrt{n}}\,\, \frac{2 \rho}{\sqrt{r}}.
$$
Now let $y=x/\|x\|$. Then
\begin{equation}\label{setam}
  \forall i\in A\, : \,
  \left|y_i - \frac{b}{\|x\|}\right| =
   \frac{|x_i - b|}{\|Q x\|}\, \frac{\|Qx\|}{\|x\|} \leq
   \frac{1}{\sqrt{n}}\,\, \frac{2 \rho}{\sqrt{r}}  \, \|Qy\|.
\end{equation}
The end of the proof is similar to the end of the proof of Case~1.
If $y\in \stt_1$, we are done.
If $y\in \stt_2'$, then using (\ref{setam}), $\|Qy\|\leq \|y\|=1$,
 and $6\rho/\sqrt{r} \leq c_0$ we obtain
 that $y\in U(m, c_0)$  and, thus, $y\in \stt_{2}$.
If $y\not\in \stt_1\cup \stt_2'$,
$y^*_2\leq \beta_p\|Qy\|$, and $\lam _k\leq \|Qy\|<\lam_{k+1}\leq 3\lam_k$ then,
using (\ref{setam}) and $6\rho/\sqrt{r} \leq c_0$ we obtain
 that $y\in U(m, c_0\lam _k)$  and, thus, $y\in \stt_{3,k}$.
 If $y\not\in \stt_1\cup \stt_2'$,
$y^*_2\geq \beta_p\|Qy\|$, and $\mu _k\leq y_2^*<\mu_{k+1}\leq 3\mu_k$ then,
similarly, using (\ref{setam}) and $6\rho/\sqrt{r} \leq c_0\beta_p /\sqrt{\ln(e/p)}$,
 we obtain  that $y\in U(m, c_0\mu _k/\sqrt{\ln(e/p)})$ and, thus, $y\in \stt_{4,k}$.
This completes the proof.
\end{proof}

\section{Complement of gradual non-constant vectors: general case}
\label{s: steep}

We split $\R^n$ into two classes of vectors. The first class, the class of {\it steep} vectors
$\st$, is constructed in  essentially the same way as in
 \cite{LLTTY first part} and \cite{LLTTY-TAMS}. The proof of bound for this class resembles
corresponding proofs in \cite{LLTTY first part} and \cite{LLTTY-TAMS}, however, due to the
differences of the models of randomness, there are important modifications.
The second class $\Rv$, which we call $\Rv$-vectors, will consist
of vectors to which Proposition~\ref{rogozin} can be applied, therefore dealing with this class is simpler.
To control the cardinality of nets, part of this class will be intersected with
the almost constant vectors.
Then we show that the complement of $\gncvectors_n(r,\gfn,\delta,\rho)$ in $\normr_n(r)$
is contained in $\st\cup \Rv$.


We now introduce the following parameters, which will be used throughout this section.
It will be convenient to denote $d=pn$. We always assume that $p\leq 0.0001$ and
$n$ is large enough (that is, larger than a certain positive absolute constant).
We also always assume that the ``average degree'' $d=pn\geq 200\ln n$.
Fix  a sufficiently small  absolute positive constant $\aaa$ and sufficiently
large absolute positive constant $\ct$
(we do not try to estimate the actual values of $\aaa$ and $\ct$, the conditions
on how small $r$ can be extracted from the proofs, in particular, the condition on $\ct$
comes for (\ref{ctau})).
We also fix two positive integers $\ell_0$ and $s_0$ such that
\begin{equation}\label{eq: l0s0 def}
 \ell_0 = \left\lfloor \frac{pn}{4\ln (1/p)}\r\rfloor
 \qand \ell_0^{s_0-1} \leq \frac{1}{64 p}=\frac{n}{64 d}< \ell_0 ^{s_0} .
\end{equation}
Note that $\ell_0 \geq 50$ and  that $s_0> 1$ implies $p\leq c\sqrt{(\ln n)/n}$.


For $1\leq j\leq s_0$ we set
$$
 n_0:=2, \,\, \,\, \,\, \,\,   n_{j}:= 30 \ell_0^{j-1},  \,\,\,\, \,\, \,\,
   \,\,\,\, \,\, \,\,     n_{s_0+2} :=\left\lfloor   \sqrt{n/p}  \r\rfloor
   =\left\lfloor  \frac{n}{\sqrt{d} } \r\rfloor ,
   \,\,\,\, \,\, \,\,     \mbox{ and }  \,\, \,\, \,\, \,\,     \nn :=\lfloor \aaa n \rfloor.
$$
Then, in the case $\left\lfloor 1/(64p) \r\rfloor \geq 15 n_{s_0}$ we set
$n_{s_0+1}=\left\lfloor 1/(64p) \r\rfloor$. Otherwise,
let $n_{s_0+1}=n_{s_0}$.
Note that with this definition we always have $n_{s_0+2}>n_{s_0+1}$.
{\bf The indices $n_j$, $j\leq s_0+3$, are global parameters which will be used throughout the section.}
Below we provide the proof only for the case
$\left\lfloor 1/(64p) \r\rfloor=n_{s_0+1}\geq 15 n_{s_0}$, the other case is treated similarly (in particular,
in that other case the set $\st_{1 (s_0+1)}$ defined below, will be empty).

We also  will use  another parameter,
\begin{equation}\label{eq: kappa def}
\kappa = \kappa(p):= \frac{\ln (6pn)}{\ln \ell_0}.
\end{equation}
Note that the function $f(p)=\ln (6pn)/(4\ln (1/p))$
is a decreasing function on $(0,1)$, therefore
for $p\geq (100 \ln n)/n$ and sufficiently larg $n$ we have
$1<\kappa\leq \ln \ln n$. Moreover, it is easy to see that if
$p\geq (100 \ln^2 n)/n$, then $\kappa\leq 2$.
We also notice that if  $pn\geq  6(5\ln n)^{1+\gamma}$ for some $\gamma\in (0,1)$ then
$\kappa \leq 1+1/\gamma$ and, using the definition of $\ell_0$ and $s_0$,
\begin{equation}\label{kappain}
  (6d)^{s_0-1} =
  \ell_0^{(s_0-1)\kappa}\leq 1/(64p)^{\kappa}.
\end{equation}


\subsection{Two classes of vectors and main results}
\label{subs: steep vectors}

We first introduce the class of steep vectors. It will be constructed as a union of four subclasses.
Set
$$
  \st _{0} : =\{x\in \R^n\,:\,   x_{1}^{*}> 6 d\, x_{2}^{*}\}
  \qand \st _{11} : =\{x\in \R^n\,:\,  x\not\in  \st_{0}
  \,\, \mbox{ and } \, \,  x_{2}^{*}> 6 d\, x_{n_1}^{*}\} .
$$
Then for $2\leq j\leq s_0+1$,
$$
   \st _{1 j} : =\left\{x\in \R^n\,:\,  x\not\in \st_{0} \cup \bigcup_{i=1}^{j-1} \st_{1 i}
    \,\, \mbox{ and } \, \,  x_{n_{j-1}}^{*}>  6 d\, x_{n_{j}}^{*}\r\}
    \qand \st _{1} : = \bigcup_{i=1}^{s_0+1} \st_{1 i}.
$$
Finally, for $k=2,3$ set $j=j(k)= s_0+k$ and define
$$
  \st _{k} : =\left\{x\in \R^n\,:\,  x\not\in  \bigcup_{i=0}^{k-1} \st_{i}
   \,\, \mbox{ and } \, \, x_{n_{j-1}}^{*}>   \ct\sqrt{d}\,x_{n_{j}}^{*}\r\}.
$$
%
The set of steep vectors is $\st:=\st_0\cup\st_1\cup\st_2 \cup\st_3$.
The ``rules'' of the partition are summarized in the diagram.
%
%
%
%
%
%
\includegraphics[width=1.0\textwidth]{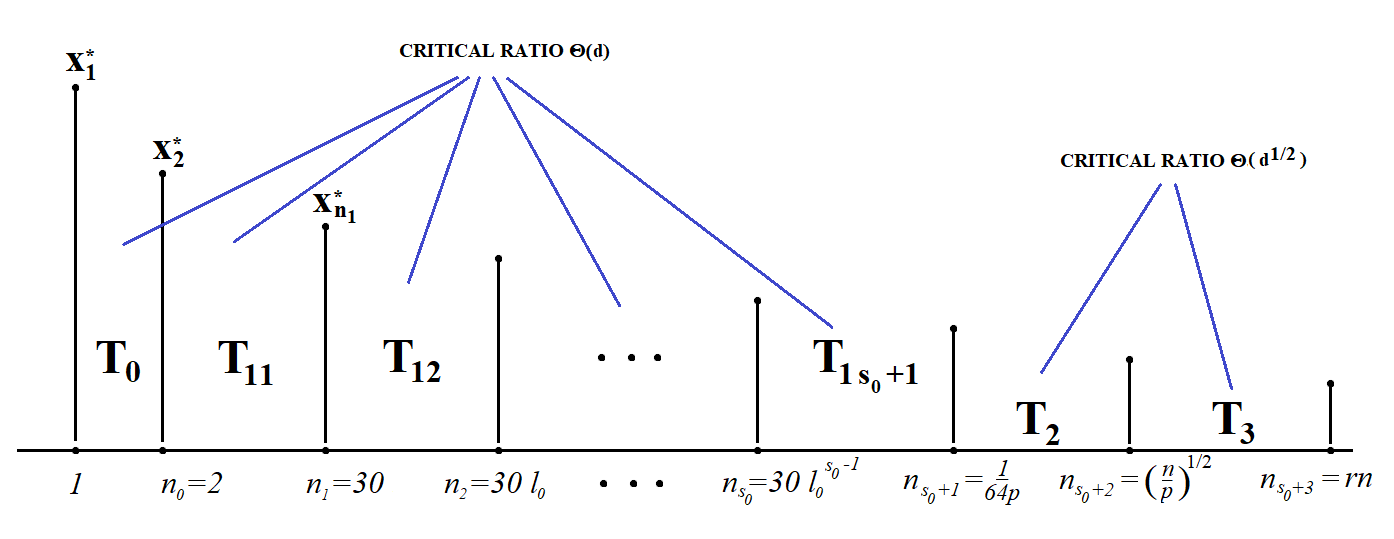}

For this class we prove the following bound.

\begin{theor} \label{steep}
There exist positive absolute constants $c,C>0$ such that the following holds.
Let $n\geq C$, and let $0<p<c$ satisfy $pn\geq  C \ln n$.
Let $M$ be a \Ber random matrix and denote
$$
  \Event_{steep}:=\bigg\{\exists\;x\in \st\, \, \, \mbox{ such that }
  \, \, \,
  \|M x\| < \frac{c(64p)^{\kappa}}{(pn)^2}\, \min\left(  1, \frac{1}{p^{1.5}n} \r) \, \,  \|x\|
  \bigg\},
$$
where as before $\kappa = \kappa(p):= (\ln (6pn))/\ln \ell_0.$
 Then
\begin{equation*}
\label{Psteep}
\Prob(\Event_{steep}) \leq n(1-p)^n + 2 e^{-1.4 pn}.
\end{equation*}
\end{theor}

Next we introduce the class of $\Rv$-vectors, denoted by $\Rv$.
Let $C_0$ be the constant from Proposition~\ref{rogozin} and recall that the class $\BB(\rho)$
of almost constant vectors was defined by (\ref{acv}) in Subsection~\ref{ss: steep overview}.
Given $n_{s_0+1}<k\leq n/\ln^2 d$ denote $A=A(k):=[k, n]$ and consider the sets
$$
 \Rv_k^1  :=\left\{x\in \big(\normr_n(r)\setminus \st\big) \cap \BB(\rho) \, \, : \,  \,
 \frac{\|x_{\sigma_x(A)}\|}{\|x_{\sigma_x(A)}\|_\infty} \geq \frac{C_0}{\sqrt{p}}
 \qand \sqrt{n/2}\leq \|x_{\sigma_x(A)}\| \leq \ct \sqrt{dn}\r\},
$$
and
$$
 \Rv_k^2  :=\left\{x\in\normr_n(r)\setminus \st
  \, \, : \,  \, \frac{\|x_{\sigma_x(A)}\|}{\|x_{\sigma_x(A)}\|_\infty} \geq \frac{C_0}{\sqrt{p}}
 \qand \frac{2\sqrt{n}}{r}\leq \|x_{\sigma_x(A)}\| \leq \ct^2 d \sqrt{n}\r\}.
$$
Define
$
  \Rv:=\bigcup _{n_{s_0+1}<k\leq n/\ln^2 d}\, (\Rv_k^1 \cup \Rv_k^2).
$

\smallskip
The class $\Rv$ should be thought of as the class of {\it sufficiently spread}
vectors, not steep, but possibly without having two subsets of coordinates of size proportional to $n$,
which are separated by $\rho$ (which would allow us to treat those vectors as part of the set $\gncvectors_n$).
Crucially, the sets $\Rv_k^1$ and $\Rv_k^2$ are ``low complexity'' sets because they admit
$\varepsilon$--nets of relatively small cardinalities (see Subsection~\ref{sub-nets}).
For the class $\Rv$ we prove the following bound.


\begin{theor}
\label{classB}
There are absolute constants $r_0,\rho_0,C$ with the following property.
Let $0<r\leq r_0$, $0<\rho\leq \rho_0$, let
$n\geq 1$ and $p\in (0, 0.001]$ be such that
$d=pn\geq C\ln n$.
Then
$$
 \p\left(\left\{\exists x\in \Rv \, : \,
 \|Mx\|\leq \frac{\sqrt{p} n}{12 C_0}\r\}\right)  \leq  e^{-2n}   +  e^{-200pn}.
$$
\end{theor}



Finally we show that together with $\gncvectors_n$ classes $\st$ and $\Rv$ cover
all (properly normalized) vectors for the growth function defined by
\begin{equation}\label{gfn-str}
\gfn(t)=(2t)^{3/2}\,\,\, \mbox{ for }\, 1\leq t<64pn\quad \qand\quad
\gfn(t)=\exp(\ln^2(2t))\,\,\, \mbox{ for }\, t\geq 64pn.
\end{equation}
It is straightforward to check that $\gfn$ satisfies (\ref{gfncond}) with some absolute constant $K_3$.

\begin{theor}
\label{complement}
There are universal constants $c,C>0$ with the following property.
Let $n\geq C$, $p\in (0, c)$, and assume that $d=pn\geq 100 \ln n$.
Let $r\in(0,1/2)$, $\delta \in(0,r/3)$, $\rho\in(0,1)$, and let $\gfn$ be as in (\ref{gfn-str}).
 Then
\label{compl}
$$
 \normr_n(r)\setminus \gncvectors_n(r,\gfn,\delta,\rho) \subset \Rv\cup \st.
$$
\end{theor}

\subsection{Auxiliary lemmas}
\label{twolemmas}

In the following lemma we provide a simple bound on the Euclidean norms of vectors
in the class $\st$ and its complement in terms of their order statistics.

\begin{lemma}\label{euclnorm}
Let $n$ be large enough and $(200 \ln n)/n<p<0.001$.
Consider the vectors $x\in \st_{1j}$ for some $1\leq j\leq s_0+1$,
$y\in \st_2$, $z\in \st_3$ and $w\in \st^c$.  Then
$$
   \frac{\|x\|}{ x_{n_{j-1}}^*}\leq \frac{ 64(pn)^2}{(64p)^{\kappa}}, \quad
   \frac{\|y\|}{ y_{n_{s_0+1}}^*}\leq \frac{ 384(pn)^3}{(64p)^{\kappa}}, \quad
   \frac{\|z\|}{ z_{n_{s_0+2}}^*}\leq \frac{ 384\ct (pn)^{3.5}}{(64p)^{\kappa}},
 \qand   \frac{\|w\|}{ w_{n_{s_0+3}}^*}\leq \frac{ 384\ct^2 (pn)^{4}}{(64p)^{\kappa}}.
$$
\end{lemma}

\begin{proof}
Let $d=pn$. Since $x\in \st_{1j}$, denoting $m=n_{j-1}$, we have
$$
  x_1^*\leq (6d) x^*_{2}\leq (6d)^2 x^*_{n_1}\leq \ldots \leq
  (6d)^j x_{n_{j-1}}^* = (6d)^j x_m^*.
$$
Since $n_i=30\ell_0^{i-1} \leq 30d^{i-1}$, $i\leq s_0$, since $\kappa>1$, and in view of \eqref{kappain}, 
we obtain
\begin{align*}
\|x\|^2 &= (x_{1}^*)^2+ (x_2^*+ \ldots + (x_{n_1}^*)^2) + ((x_{n_1+1}^*)^2 + \dots + (x_{n_2}^*)^2)
+ \ldots
\\&\leq
 ((6d)^{2 j} + n_1 (6d)^{2(j-1)} + n_2 (6d)^{2(j-2)}\ldots + n_{j-1} (6d)^{2} + n) (x_m^*)^2
\\&\leq
\Big((6d)^{2 j} + 5(6d)^{2 j-2}\sum_{i\geq 0}(6d)^{-i} +n \Big) (x_m^*)^2 \leq
\left(2 (6d)^{2 (s_0+1)} +n \r) (x_m^*)^2
\\&\leq
  \left(2 (6d)^4 /(64p)^{2\kappa}+n \r) (x_m^*)^2 \leq \left(3 (6d)^4 /(64p)^{2\kappa} \r) (x_m^*)^2 .
\end{align*}
This implies the first bound. The bounds for $y,z, w$ are obtained similarly.
\end{proof}

The next two Lemmas~\ref{col} and~\ref{c:SJ} will be used to bound from below the norm of the matrix-vector
product $M x$ for vectors $x$ with a ``too large'' almost constant part which does
not allow to directly apply the L\'evy--Kolmogorov--Rogozin anti-concentration inequality together with the
tensorization argument. Lemma~\ref{col} will be used to bound $\|Mx\|$ by a single inner product
$|\langle \row_i(M),x\rangle|$ for a specially chosen index $i$, while Lemma~\ref{c:SJ}
will allow to extract a subset of ``good'' rows having large inner products with $x$.

\begin{lemma}
\label{col}
Let $n\geq 30$ and $0<p<0.001$ satisfy $pn\geq 200 \ln n$.
Let $m, \ell=\ell(m) \geq 2$ be such that either
$$m=2\mbox{ and }\ell = 15,$$
or
$$
  m\geq 30, \quad  \ell m\leq \frac{1}{64 p} \quad \quad \mbox{and}
  \quad \quad \ell \leq \frac{np}{4\ln \frac{1}{pm}}.
$$
Let  $M$ be an $n\times n$ \Ber random matrix.
By $\Event_{col}=\Event_{col}(\ell, m)$ denote the event that
for any choice of two disjoint sets $J_1, J_2\subset [n]$ of cardinality
$|J_1|=m$, $|J_2|=\ell m - m$ there exists a row of $M$ with exactly one $1$ among components
indexed by $J_1$ and no $1$s among components indexed by $J_2$.
Then $\p(\Event_{col} )\geq 1-\exp(-1.5 pn).$
\end{lemma}

\begin{proof}
We first treat the case $m\geq 30$.
Fix two disjoint sets $J_1, J_2\subset [n]$ of required cardinality.  The probability that
a fixed row has exactly one $1$ among components
indexed by $J_1$ and no $1$s among components indexed by $J_2$ equals
$$
  q:=m p(1-p)^{\ell m - 1}\geq m p \exp(- 2p\ell m ) \geq 29mp/30,
$$
where we used $\ell m p\leq 1/64$.
Since the rows are independent, the probability that $M$ does not have such a row
is
$$
 (1-q)^n  \leq \exp(-nq) \leq  \exp(-29 m p n/30).
$$
Note that the number of all choices of $J_1$ and $J_2$ satisfying the conditions of the lemma is
$$
 {n \choose \ell m - m} {n-\ell m +m \choose m} \leq
 \left( \frac{en}{(\ell-1) m}\r)^{\ell m -m} \left(\frac{en}{m}\r)^m
 \leq \left( \frac{3n}{\ell m}\r)^{\ell m } (2\ell) ^m .
$$
Thus union bound over all choices of $J_1$ and $J_2$ implies
$$
 \p((\Event_{col})^c )\leq \left( \frac{3n}{\ell m}\r)^{\ell m } (2\ell) ^m   \exp(-29 m p n/30).
$$
Using that $m\leq 1/(64p)$ and $\ell \leq \frac{np}{4\ln (1/(pm))}$, we observe
$
 \left( \frac{3n}{\ell m}\r)^{\ell m } \leq \exp( m p n/2).
$
Since $np\geq 200 \ln n$, we have $(2\ell) ^m \leq \exp( 2m p n/5)$.
Thus,
$$
 \p((\Event_{col} )^c )\leq  \exp(- m p n/15) \leq \exp(- 2 p n),
$$
which proves this case.

The case $m=2$, $\ell = 15$ is similar.
Fixing two disjoint sets $J_1, J_2\subset [n]$ of the required cardinality,  the probability that
a fixed row has exactly one $1$ among components
indexed by $J_1$ and no $1$s among components indexed by $J_2$ equals
$$
  q:=2p(1-p)^{29}\geq 2p \exp(-29 p).
$$
Since rows are independent, the probability that $M$ does not have such a row
is
$$
 (1-q)^n  \leq (1- 2p \exp(-29 p))^n \leq  \exp(-2pn \exp(-29 p))
  \leq \exp(-1.8 pn ).
$$
Using union bound over all choices of $J_1$ and $J_2$ we obtain
$$
 \p(\Event_{sum}^c )\leq  \frac{n^{30}}{2\cdot 28!}\exp(-1.8p n)\leq
\exp(-1.5p n),
$$
which proves the lemma.
\end{proof}

\smallskip

In the next lemma we restrict a matrix to a certain set of columns
and estimate the cardinality of a set of rows having exactly one $1$.
To be more precise, for any  $J\subset [n]$ and a $0/1$ matrix $M$
denote
$$
   I_J=I(J,M) := \{i\le n :\,  |\supp \row_i(M)\cap J|=1\}.
$$
The following statement is similar to Lemma 2.7 from \cite{LLTTY first part} and
Lemma~3.6 in \cite{LLTTY-TAMS}.

\begin{lemma}\label{c:SJ}
Let $\ell \geq 1$  be an integer and   $p\in (0, 1/2]$ be such that $p\ell \leq 1/32$.
Let $M$ be a \Ber random matrix.
Then with  probability at least
$$
   1- 2{n \choose \ell}\exp\left(-n\ell p/4\r)
$$
for every  $J\subset [n]$ of cardinality $\ell$ one has
$$
   \ell pn/16 \leq |I(J,M)| \leq 2\ell n p.
$$
In particular, if $\ell = 2\lfloor 1/(64p)\rfloor\leq n$, $n\geq 10^5$,  and $p\in [100/n, 0.001]$  then,
denoting
$$
 \Event _{card} =\Event _{card} (\ell) :=\{M\in \Mc \, :\, \forall J\subset [n]\,\,\, \mbox{with}\,\,\, |J|=\ell
 \,\,\, \mbox{one has}\,\,\,   |I(J,M)|\in [\ell p n/16, 2 \ell p n]\},
$$
we have
$$
 \p\left(\Event _{card} \r)\geq  1- 2\exp\left(-n/500\right).
$$
\end{lemma}

\smallskip

\begin{proof}
Fix $J\subset [n]$ of cardinality $\ell$. Denote  $q=\ell p (1-p)^{\ell -1}$. Since $\ell p\leq 1/32$,
$$
     15\ell p/16\leq  \ell p(1-2p\ell )\leq \ell p \exp(-2 p\ell)\leq q\leq \ell  p \leq 1/2.
$$
For every $i\leq n$, let $\xi_i$ be the indicator of the event $\{i\in I(J,M)\}$.
Clearly, $\xi_i$'s are independent \Berq random variables and
$|I(J,M)| = \sum_{i=1}^n \xi_i$. Applying Lemma~\ref{bennett}, we observe
that for every $0<\eps<q$
$$
  \p\left(|I(J,M)|\in [(q-\eps)n, (q+\eps)n]\r)\geq  1- 2\exp\left(-\frac{n\eps^2}{2q(1-q)}  \,
 \left(1-\frac{\eps}{3q}\right)\right).
$$
 Taking $\eps = 14q/15$ we obtain that
 $$
    (q-\eps)n= qn/15\geq \ell pn/16 \quad  \quad \mbox{and} \quad  \quad
    (q+\eps)n\leq 2 qn\leq 2\ell pn,
 $$
 and
$$
  \frac{n\eps^2}{2q(1-q)}  \, \left(1-\frac{\eps}{3q}\r)
  \geq  \frac{98 \cdot 31 nq}{225 \cdot 45}   \geq
  0.3 n\ell p (1-2\ell p) \geq  n\ell p/4.
$$
This implies the bound for a fixed $J$. The lemma follows by the union bound.
\end{proof}

\subsection{Cardinality estimates for $\varepsilon$--nets}
\label{sub-nets}

In this subsection we provide bounds on cardinality of certain discretizations of the sets of vectors introduced earlier.
Recall that $\edv$ denotes the vector  ${\bf 1} /\sqrt{n}$,
$P_\edv$ denotes the projection on $\edv^\perp$, and
$P_\edv^\perp$ is the projection on $\edv$, that is $P_\edv^\perp = \la \cdot , \edv\ra \edv$.
We recall also that given $A\subset [n]$, $x_A$ denotes coordinate projection of $x$ on $\R^A$,
and that given $x\in \R^n$, $\sigma_x$ is a (fixed) permutation corresponding to
non-increasing rearrangement of $\{|x_i|\}_{i=1}^n$.

Our first lemma deals with nets for $\st_2$ and $\st_3$.
We will consider the following normalization:
$$
 \st'_2=\{ x\in \st_2\,:\, x_{n_{s_0+1}}^{*}=1 \}, \qand
  \st'_3=\{ x\in \st_3\,:\, x_{n_{s_0+2}}^{*}=1 \}.
$$
%
The triple norm is defined by
$
  ||| x |||^2 := \|P_\edv x\|^2 + p n \|P_\edv^\perp x\|^2.
$

\begin{lemma}
\label{l:nets}
Let $n\geq 1$, $p\in (0, 0.001]$, and assume that $d=pn$ is sufficiently large. Let  $i\in\{2,3\}$.
Then there exists a set $\Net _i= \Net_i' + \Net_i''$, $\Net_i'\subset \R^n$, $\Net_i''\subset \spn\{{\bf{}1}\}$, with the following properties:
\begin{itemize}

\item $
  |\Net_i| \le \exp\left( 2  n_{s_0+i}\ln d \r).
$

\item For every $u\in \Net_i'$ one has $u_j^*=0$ for all $j\geq n_{s_0+i}$.

\item For every $x\in \st_i'$  there are $u\in \Net_i'$ and $w\in \Net_i''$
  satisfying
$$
         \|x-u\|_\infty \leq \frac{1}{\ct \sqrt{d}},
         \quad \|w\|_\infty \leq\frac{1}{\ct \sqrt{d}}, \qand
         |||x-u-w|||\leq \frac{\sqrt{2n}}{\ct \sqrt{d}}.
$$
\end{itemize}
\end{lemma}

Since the proof of this lemma in many parts repeats the  proofs
of Lemma~3.8 from \cite{LLTTY first part} and of Lemma~\ref{newnet} below,
we only sketch it.

\begin{proof} Fix $\mu =1/(\ct\sqrt{d}$) and $i\in \{2, 3\}$.
We first repeat the proof of Lemma~3.8 from \cite{LLTTY first part} with our choice of parameters.
See also the beginning of the proof of Lemma~\ref{newnet} below --- many definitions, constructions, and
calculations are exactly the same, however note that the normalization is slightly different.
In particular,
the definitions of sets $B_1(x)$, $B_2(x)$ (with $k-1=n_{s_0+i-1}$),
$B_3(x)$ are the same (we do not need the sets $B_0(x)$ and $B_4(x)$).
This will show (for large enough $d$) the existence of a $\mu$-net $\mathcal{N}_i'$ (in the $\ell_\infty$ metric)
for $\st_i'$ such that for every $u\in \Net_i'$ one has $u_j^*=0$ for all $j\geq n_{s_0+i}$ and
$
 |\mathcal{N}_i'|\leq \exp\left( 1.1  n_{s_0+i}\ln d \r)$.

Next given $x\in \st_i'$ let $u=u(x)\in \mathcal{N}_i'$
be such that $\|x-u\|_\infty \leq \mu$. Then $\|P_\edv^\perp (x-u)\|\leq \mu \sqrt{n}$.
Let $\mathcal{N}_i''$ be a $(\mu\sqrt{n/d})$-net in the segment $\mu \sqrt{n} \, [-\edv, \edv]$ of cardinality at most
$2\sqrt{d}$ (note, we are in the one-dimensional setting). Note that every $w\in \mathcal{N}_i''$ is of the form
$w=a\, \edv=a\, {\bf 1}/\sqrt{n}$, $|a|\leq \mu \sqrt{n}$, in particular, $\|w\|_\infty\leq \mu$.
Then for $x$ (and the corresponding $u=u(x)$),
there exists $w\in \mathcal{N}_i''$ such that
$$
   |||x-u-w|||^2=
   \| P_\edv (x - u-w)\|^2 + d\| P_\edv^\perp (x - u -w) \|^2=
   \| P_\edv (x - u)\|^2 + d\| P_\edv^\perp (x - u) -w \|^2\leq 2 \mu^2 n .
$$
Finally, note that $|\mathcal{N}_i' + \mathcal{N}_i''|\leq 2\sqrt{d} \exp\left( 1.1  n_{s_0+i}\ln d \r)\leq  \exp\left( 2  n_{s_0+i}\ln d \r)$.
This completes the proof.
\end{proof}

Let $\Rv_{k}^1$, $\Rv_{k}^2$ be the vector subsets introduced in Subsection~\ref{subs: steep vectors}.
Consider the increasing sequence $\lam _1<\lam_2<\ldots <\lam _m$, $m\geq 1$, defined by
\begin{equation}\label{eq 2498520598207560}
\mbox{$\lam _1 = 1/\sqrt{2}$, $\,\, \, \lam _{i+1}=3\lam _i\,\,\, $ for $1<i<m,\quad $ and
$\quad \lam _{m-1} <\lam _m=\ct^2 d \leq 3\lam _{m-1}$.}
\end{equation}
Clearly $m\leq n$.
For $s\in\{1,2\}$, $n_{s_0+1}<k\leq n/\ln^2 d$ and $i\leq m$ set
$$
  \Rv_{k i}^s  :=
  \left\{x\in\Rv_k^s \, \, : \,  \,  \lam _i \sqrt{n}\leq \|x_{\sigma_x([k,n])}\| \leq \lam _{i+1} \sqrt{n}\r\}.
$$
It is not difficult to see that the union of  $\Rv_{k i}^s$'s over admissible $i$ gives $\Rv_k^s$.
The sets $\Rv_{k i}^s$ are ``low complexity'' sets in the sense that they admit
efficient $\varepsilon$-nets. For $s=1$, the low complexity is a consequence of the condition that
$\Rv_{k i}^1\subset \BB(\rho)$, i.e., the vectors have a very large almost constant part.
For the sets $\Rv_{k i}^2$, we do not assume the almost constant behavior, but instead rely on the assumption
that $\|x_{\sigma_x([k,n])}\|$ is large (much larger than $\sqrt{n}$). This will allow us to pick
$\varepsilon$ much larger than $\sqrt{n}$, and thus construct a net of small cardinality.

\begin{lemma}
\label{newnet}
Let $R\geq 40$ be a (large) constant.
Then there is $r_0>0$
depending on $R$ with the following property.
Let $0<r\leq r_0$, $0<\rho\leq 1/(2R)$, let
$n\geq 1$ and $p\in (0, 0.001]$ so that $d=pn$ is sufficiently large (larger than a constant depending on $R,r$).
%
Let $s\in \{1,2\}$,  $n_{s_0+1}<k\leq n/\ln^2 d$, $t\leq m$, and
$40 \lam_t \sqrt{n}/R\leq  \eps \leq \lam_t\sqrt{n}$,
where $\lam_t$ and $m$ are defined according to relation \eqref{eq 2498520598207560}.
Then there exists an $\eps$-net
$\mathcal{N}_{k t}^s\subset \Rv_{k t}^s$ for  $\Rv_{k t}^s$ with respect to $|||\cdot|||$
of cardinality at most $(e/r)^{3rn}$.
\end{lemma}

\begin{proof}
Note that in case of $s=2$ the set $\Rv_{k t}^2$ is empty whenever $3\lambda_t<\frac{2}{r}$.
So, in the course of the proof we will implicitly assume that $3\lambda_t\geq\frac{2}{r}$ whenever $s=2$.

We follow ideas of the proof of Lemma~3.8 from \cite{LLTTY first part}.
We split a given vector from $\Rv_{k t}^s$ into few parts according to magnitudes of its coordinates
and approximate each part separately.
Then we construct nets for vectors with the
same splitting and take the union over all nets. We now discuss the splitting. For each $x\in \Rv_{k t}^s$
consider the following (depending on $x$) partition of $[n]$.
If $s=2$, set $B_0'(x)=\emptyset$. If $s=1$ then $x\in \BB(\rho)$ and we set
$$
  B_0'(x) :=\sigma _x(\{ j\leq n \, : \,  |x_j-\lam _x|\leq \rho \}),
$$
where $\lam _x=\pm 1$ is from the definition of $\BB(\rho)$ (note that under the normalization
in $\normr_n(r)$ we have  $x^*_{\nn}=1$). Then $|B_0'(x)|> n - \nn$ for $s=1$.
Next, we set
\begin{align*}
B_1(x)&=\sigma_x([n_{s_0+1}]);\\
B_2(x)&= \sigma_x([k-1])\setminus B_1(x);\\
B_3(x)&= \sigma_x([n_{s_0+3}])\setminus (B_1(x)\cup  B_2(x));\\
B_0(x) &= B_0'(x)\setminus (B_1(x)\cup  B_2(x)\cup  B_3(x));\\
B_4(x) &=[n]\setminus (B_0(x)\cup  B_1(x)\cup B_2(x)\cup B_3(x) )
\end{align*}
(one of  the sets $B_0(x)$, $B_4(x)$ could be empty).
Denote $\ell _x:=|B_0(x)|$.
Note that the definition of $B_3(x)$ and $B_4(x)$ imply that $\ell_x \leq n-\nn$,
while the condition $k-1\leq n_{s_0+3}$ and the above observation for $B_0'(x)$ give
$n-2\nn< \ell_x$ for $s=1$. Clearly, $\ell_x=0$ for $s=2$.

Moreover, we have both for $s=1$ and $s=2$:
\begin{equation}\label{cardpart}
  |B_1(x)|=n_{s_0+1}, \quad   |B_2(x)|= k-1 - n_{s_0+1},
  \quad   |B_3(x)|= \nn - k+1,
  \quad   |B_4(x)|= n- \ell_x - \nn.
\end{equation}

Thus, given $\ell \in \{0\}\cup [n-\nn-k+1, n-k+1]$ and a partition of $[n]$ into five sets
$B_i$, $0\leq i\leq 4$,  with cardinalities as in (\ref{cardpart}),
it is enough to construct a net for vectors $x\in \Rv_{k t}^s$ with
$B_i(x)=B_i$,  $0\leq i\leq 4$, $\ell_x=\ell$, and then to take the union of nets over all possible realizations of $\ell$ and all
such partitions $\{B_0,B_1,B_2, B_3, B_4\}$ of $[n]$.

\smallskip

Now we describe our construction. Fix $\ell$ as above and fix two parameters $\mu= 1/(\ct \sqrt{d})$, 
and
$\nu=9\lam_t \sqrt{n}/R$.
We would like to emphasize that for the actual calculations in this lemma, taking $\mu$ to be a small constant multiple of
$R^{-1}$ would be sufficient, however, we would like to run the proof with the above choice of $\mu$ because this
corresponds to the parameter choice in the previous Lemma~\ref{l:nets} whose proof we only sketched.
Note  that for $x\in \Rv_{k t}^s$ we have $x\not\in \st$, hence
 $x^*_{n_{s_0+1}}\leq \ct \sqrt{d}x^*_{n_{s_0+2}}\leq \ct^2 d$ and
\begin{equation}\label{decr2}
  x_1^*\leq (6d) x^*_{2}\leq (6d)^2 x^*_{n_1}\leq \ldots \leq (6d)^{s_0+2} x_{n_{s_0+1}}^*
  \leq \ct^2 d (6d)^{s_0+2}.
\end{equation}

Fix $I_0\subset [n]$ with $|I_0|=n_{s_0+1}$ (which will play the role of $B_1$). We shall construct
a $\mu$-net $\Net_{I_0}$ (in the $\ell_\infty$-metric) for the set
\begin{align*}
   \st _{I_0}:=\big\{P_{B_1(x)}x:\;x\in\Rv_{k t}^s,\;B_1(x)=I_0\big\}.
\end{align*}
Clearly, the nets $\Net_{I_0}$ for various $I_0$'s can be related by appropriate permutations, so
without loss of generality we can assume for now that $I_0=[n_{s_0+1}]$.
First, consider the partition of $I_0$ into sets $I_1, \ldots, I_{s_0+2}$ defined by
$$
  I_1=[2] \qand  \, \, \, I_j=[n_{j-1}]\setminus [n_{j-2}], \, \, \mbox{ for }\,\, 2\leq j \leq s_0+2.
$$
Consider the set
$$\st^*:=\big\{x\in\st_{[n_{s_0+1}]}:\,\sigma_x(I_j)=I_j,\;\;j=1,2,\dots,s_0+2\big\}.$$
By the definition of $\st _{I_0}$, for every $x\in \st^*$, one has  $\|P_{I_j}x\|_\infty\le b_j:=\ct^2 d (6d)^{s_0+3-j}$ for every
$j\le s_0+2$  (where as usual $P_I$ denotes the coordinate projection onto $\R^I$). Define
a $\mu$--net (in the $\ell_\infty$-metric) for $\st^*$ by setting
$$
  \Net^*:=\Net_{1}\oplus\Net_{2}\oplus\cdots\oplus\Net_{s_0+2},
$$
where $\Net_{j}$ is a $\mu$-net (in the $\ell_\infty$-metric) of cardinality at most
$$
  (3 b_j/\mu )^{|I_j|} \leq  (\ct^3 d^{3/2} (6d)^{s_0+3-j})^{n_{j-1}} \leq (\ct^3  (6d)^{s_0+5-j})^{n_{j-1}
}
$$
in the coordinate projection of the  cube $P_{I_{j}}(b_j B_\infty^n)$. 
 Recall that $n_0=2$, $n_j=30\ell_0^{j-1}$, $1\leq j\leq s_0$, where $\ell_0$ and $s_0$ are given by \eqref{eq: l0s0 def}.
 Since $d$ is large enough,
\begin{align*}
  2s_0 +8 + 30 \sum_{j=2}^{s_0+1}  (s_0+5-j) \ell_0^{j-2} &=
  2s_0 +8 + 30 \sum_{m=1}^{s_0-1} (m+3) \ell_0^{s_0-m}
\leq 121 \ell_0^{s_0-1} \leq 4.1 n_{s_0+1},
\end{align*}
which implies
$$
  |\Net^*|\le\prod_{j=1}^{s_0+2}|\Net_{j}| \le \exp( 7.1 n_{s_0+1}  \ln (6 \ct^2 d)).
$$
To pass from the net for $\st^*$ to the net for $\st_{[n_{s_0+1}]}$, let $\Net_{[n_{s_0+1}]}$ be the union
of nets constructed as $\Net^*$ but for arbitrary partitions $I_1',\dots, I_{s_0+2}'$ of $[n_{s_0+1}]$ with $|I_j'|=|I_j|$.
Using that
$$
   \sum _{j=1}^{s_0+1} {n_{j-1}} \le 2 + 30 \sum _{j=0}^{s_0-1} \ell_0^j  \leq  2 + 30 \ell_0^{s_0-1}/(1-1/\ell_0) \leq  2 n_{s_0+1}
$$
 and $e\ell_0\leq d$ we obtain that the cardinality of $\Net_{[n_{s_0+1}]}$ is at most
\begin{align*}
   |\Net^*|\,   \prod_{j=1}^{s_0+1} { n_{j} \choose n_{j-1} }
  & \le  |\Net^*|\,   \prod_{j=1}^{s_0+1} \Big(\frac{e n_{j}}{n_{j-1}}\Big)^{n_{j-1}} \le
   |\Net^*|\, \prod_{j=1}^{s_0+1} (e \ell_0)^{n_{j-1}}
    \le
    \exp(9.1 n_{s_0+1}  \ln (6 \ct^2 d)) .
\end{align*}

 Next we construct a net for the parts of the vectors corresponding to $B_2$. Fix
 $J_0\subset [n]$ with $|J_0|= k-1 - n_{s_0+1}$ (it will play the role
 of $B_2$). We  construct a $\mu$-net  (in the $\ell_\infty$-metric) for the set
$$
   \st^2 _{J_0}:=\{ P_{B_2(x)}x \, : \,  x\in\normr_n(r)\setminus \st,\,\, B_2(x)=J_0  \}.
$$
  Since by \eqref{decr2}, we have  $x^*_{n_{s_0+1}}\leq  \ct^2 d$ for every $x\in \normr_n(r)\setminus \st$,
  it is enough to take a $\mu$-net  ${\mathcal{K}}_{J_0}$  of cardinality  at most
$$
   |{\mathcal{K}}_{J_0}|  \leq (3 \ct^2 d/\mu) ^{|J_0|} \leq
   (3 \ct^3 d^{3/2}) ^{k}
$$
in the coordinate projection of the  cube  $P_{J_{0}}(\ct^2 d B_\infty^n)$.

 Now we turn to the part of the vectors corresponding to $B_3$. Fix  $D_0\subset [n]$ with
 $|D_0|= \nn - k+1$ (it will play the role of $B_3$).
 For this part we use $\ell_2$-metric and  construct a $\nu$-net  (in the {\it Euclidean metric} this time) for the set
$$
   \st^3 _{D_0}:=\{ P_{B_3(x)} x \, : \,  x\in \Rv_{k t}^s,\,\, B_3(x)=D_0  \}.
$$
Since for $x\in \Rv_{k t}^s$ we have $\|x_{B_3(x)}\|\leq \|x_{\sigma_x([k,n])}\|\le 3\lam _t \sqrt{n}$, there exists
a corresponding $\nu$-net   ${\mathcal{L}}_{D_0}$  in the coordinate projection of the Euclidean ball
  $P_{D_{0}}(3\lam_t \sqrt{n} B_2^n)$
of cardinality  at most
$$
   |{\mathcal{L}}_{D_0}|  \leq (9\lam _t \sqrt{n}/\nu) ^{|D_0|} \leq
   R^{\nn}\leq R^{r n}.
$$

 Next we approximate the almost constant part of a vector (corresponding to $B_0$),
 provided that it is not empty (otherwise we skip this step). Fix $A_0\subset [n]$ with
 $|A_0|= \ell$ (it will play the role of $B_0$) and denote
$$
   \st^0 _{A_0}:=\{ P_{B_0(x)} x \, : \,  x\in\big(\normr_n(r)\setminus \st\big) \cap \BB(\rho),\,\, B_0(x)=A_0  \}.
$$
Let ${\mathcal{K}}^0_{A_0} :=\{ \pm P_{A_0} {\bf 1}\}$. Since for every $x\in\normr_n(r)$ we have
either $\lam_x=1$ or $\lam_x=-1$, by the definition of $B_0(x)$, every $z\in \st^0 _{A_0}$
is approximated by one of $\pm P_{A_0} {\bf 1}$ within error $\rho$ in the $\ell_\infty$-metric.

The last part of the vector, corresponding to $B_4$ we just approximate by $0$.
Note that for any $x\in\Rv_{k t}^1$ we have $\|P_{B_4(x)}x\|\leq \sqrt{rn}\leq \sqrt{2r}\lambda_t\sqrt{n}$,
in view of the condition $x\in \BB(\rho)$.
On the other hand, for $x\in\Rv_{k t}^2$ we have $\|P_{B_4(x)}x\|\leq \sqrt{n}\leq \frac{3r}{2}\lambda_t\sqrt{n}$.

Now we combine our nets. Consider the net
$$
 \Net_0 :=\bigcup\limits_{\ell,I_0,J_0,D_0,A_0}\big\{y=y_1+y_2+y_3+y_0:\,y_1\in\Net_{I_0},\,y_2\in\mathcal{K}_{J_0},\,
 y_3\in \mathcal{L}_{D_0}, y_0\in \mathcal{K}^0_{A_0}\big\},
$$
where the union is taken over all $\ell\in \{0\}\cup [n-2\nn, n-\nn]$ and all partitions of $[n]$ into $I_0, J_0, D_0, A_0, B$ with
$|I_0|=n_{s_0+1}$, $|J_0|=k-1 - n_{s_0+1}$, $|D_0|= \nn - k+1$, $|A_0|=\ell$, and
$B=[n]\setminus(I_0\cup J_0 \cup D_0 \cup A_0)$.
Then the cardinality of  $\Net_0$,
\begin{align*}
 |\Net_0|&\le n  {n\choose n_{s_0+1}} {n- n_{s_0+1}\choose k-1 - n_{s_0+1}}
  {n- k+1 \choose \nn - k+1} {n- \nn \choose \ell} \max\limits_{I_0}|\Net_{I_0}| \max\limits_{J_0}|\mathcal{K}_{J_0}|
  \max\limits_{D_0}|\mathcal{L}_{D_0}| \max\limits_{A_0}|\mathcal{K}^0_{A_0}|.
\end{align*}
Using that $n_{s_0+1}\leq n/(64d)$, $k\leq n/\ln^2 d$, $\nn\leq rn$, $\ell =0$ or $\ell\geq n-2\nn$,
the obtained bounds on nets, as well as  that $d$ is large enough and $r$ is small enough
(smaller than a constant depending on $R$),
we observe that the cardinality of $\Net_0$ is bounded by
$$
  n \left(e d\r)^{n/d}\,  \left(2e \ln ^2 d\r)^{n/\ln^2 d}\,  \left(2e /r\r)^{rn}\,
  \left(2e /r\r)^{rn}\, \exp(9.1n \ln (6 \ct^2 d)/(64d))\, (3 \ct^3 d^{3/2}) ^{n/\ln ^2 d}
  R^{rn}\, \cdot 2 \leq \left(e /r\r)^{2.5 rn}.
$$
By construction, for every $x\in \Rv_{k t}^s$  there exists $y=y_1+y_2+y_3+y_0\in \Net_0$ such that
\begin{align*}
   \|x-y\|&\leq \| P_{B_1(x)}x-y_1\|+\| P_{B_2(x)}x-y_2\| +\| P_{B_3(x)}x-y_3\| + \| P_{B_4(x)}x\|
  +\| P_{B_0(x)}x-y_0\|
  \\&\leq
   \mu \sqrt{n_{s_0+1}} + \mu \sqrt{k-1- n_{s_0+1}} + \nu  +\sqrt{2r}\lambda_t\sqrt{n} + \rho \sqrt{n}\leq \frac{2\sqrt{n}}{\ct \sqrt{d}}
    + \rho \sqrt{n}+ \frac{9\lam_t\sqrt{n}}{R}\leq \frac{10\lam_t\sqrt{n}}{R},
\end{align*}
where we used that $\rho \leq 1/(2R)\leq \lam_1 /(\sqrt{2} R)\leq \lam_t /(\sqrt{2} R)$ and that $r$ is sufficiently small.

Finally we adjust our net to $|||\cdot|||$. Note that by Lemma~\ref{euclnorm} for every $x\in \normr_n(r)\setminus \st$,
$$
 |\la x , \edv\ra|= \left|\sum _{i=1}^n \frac{x_i}{\sqrt{n}}\r|
 \leq \|x\|\leq
 \frac{ 384\ct^2 d^{4}}{(64p)^{\ln (6d)}}\leq e^{rn}.
$$
Therefore, there exists an $\eps/(4\sqrt{pn})$-net $\mathcal{N}_*$ in $P_\edv^\perp \Rv_{k t}^s$ of
cardinality $8\sqrt{pn}e^{rn}/\eps$ (note, the rank of $P_\edv^\perp$ is one).
Then, by the constructions of  nets,
 for every $x\in \Rv_{k t}^s$ there exist $y\in \mathcal{N}_0$ and $y_*\in \mathcal{N}_*$ such that
$$
  ||| x - P_\edv y -y_*||| ^2=
   \| P_\edv (x - y)\|^2 + pn \|P_\edv^\perp x  -y_*\|^2\leq \frac{100\lam_t^2n}{R^2} + \eps^2/16\leq \eps^2/8.
$$
Thus the set $\mathcal{N} = P_\edv(\mathcal{N}_0) + \mathcal{N}_*$
is an $(\eps/2)$-net for $\Rv_{k t}^s$ with respect to $|||\cdot|||$ and its cardinality is bounded by $(e/r)^{3rn}$.
Using standard argument we pass to an  $\eps$-net $\mathcal{N}_{k t}^s \subset \Rv_{k t}^s$ for $\Rv_{k t}^s$.
\end{proof}


\subsection{Proof of Theorem~\ref{classB}}


\begin{proof}
Recall that the sets $\Rv_{ki}^s$ were introduced just before Lemma~\ref{newnet} and
the event $\Event_{nrm}$ was defined in Proposition~\ref{nettri}.

Fix $s\in \{1, 2\}$, $k\leq n/\ln^2 d$, $A:=[k, n]$, $i\leq m$.
Set $\eps :=\lam _i \sqrt{n}/(600\sqrt{2} C_0)$, where $\lam_i$ and $m$ are defined according to \eqref{eq 2498520598207560}.
Applying Lemma~\ref{newnet} with $R= 24000 \sqrt{2} C_0$, we
find an $\eps$-net (in the $|||\cdot|||$--norm) $\mathcal{N}_{k i}^s\subset \Rv_{k i}^s$ for $\Rv_{k i}^s$
of cardinality at most $(e/r)^{3rn}$.
Take for a moment any $y\in  \mathcal{N}_{k i}^s$. Note that
$\|y_{\sigma(A)}\| \geq C_0\|y_{\sigma(A)}\|_\infty/\sqrt{p}$, $\|y_\sigma(A)\|\geq \lam_i \sqrt{n}$
(where $\sigma=\sigma_y$).
Then Proposition~\ref{rogozin} implies $\p(\Event_y^c)\leq e^{-3n}$, where
$$
   \Event_y=\left\{ \|My\|> \frac{\sqrt{pn}}{3\sqrt{2} C_0} \, \|y_{\sigma(A)}\| \r\}.
$$
Let us condition on the event $\Event_{nrm}\cap  \bigcap\limits_{y\in \mathcal{N}_{k i}^s}\Event_y$.
Using the definition of $\mathcal{N}_{k i}^s$ and $\Rv_{k i}^s$, the triangle inequality, and
the definition of $\Event_{nrm}$ from
Proposition~\ref{nettri}, we get that for any $x\in \Rv_{k i}^s$ there is $y\in \mathcal{N}_{k i}^s$ such that
$|||x-y|||\leq \varepsilon$, and hence
$$
  \|Mx\|\geq  \|My\| - \|M(x-y)\| > \frac{\sqrt{pn}}{3\sqrt{2} C_0} \, \|y_{\sigma(A)}\| - 100\sqrt{pn} \eps \geq
   \frac{\sqrt{p}\lam _i n }{6\sqrt{2} C_0}  .
$$
Using that $|\mathcal{N}_{k i}^s|\leq (e/r)^{3rn}$, that $\lam_i\geq 1/\sqrt{2}$, and  the union bound, we obtain
$$
 \p\left(\Event_{nrm} \cap \left\{\exists x\in \Rv_{ki}^s \, : \,
 \|Mx\|\leq \frac{\sqrt{p} n}{12 C_0}\r\}\right)
\leq \Prob\Big(\Event_{nrm}\cap  \bigcup\limits_{y\in \mathcal{N}_{k i}^s}\Event_y^c\Big)
  \leq  e^{-3(1-r\ln(e/r))n}.
$$
Since $\Rv=\bigcup_{k,i}\, (\Rv_{ki}^1 \cup \Rv_{ki}^2)$ and $r$ is small enough,
the result follows by the union bound and by  Lemma~\ref{bdd} applied with $t=30$ in order to estimate $\Prob(\Event_{nrm})$.
\end{proof}

\subsection{Lower bounds on $\|Mx\|$ for  vectors from $\st_0\cup \st_1$}

The following lemma provides a lower  bound on the ratio $\|Mx\|/\nx _2$  for vectors $x$ from $\st_0\cup \st_1$.

\begin{lemma} \label{l:T0}
Let $n\geq 1$, $0<p<0.001$, and assume that  $d=pn\geq 200 \ln n$. Then
$$
  \p\left(\Big\{\exists\;x\in \st_0 \cup  \st_1 \, \, \, \mbox{ such that }
  \, \, \, \|M  x\| \leq    \frac{(64p)^{\kappa}}{ 192(pn)^2}\, \|x\| \Big\}\r)
  \leq n(1-p)^n+e^{-1.4np},
$$
where $\kappa$ is defined by \eqref{eq: kappa def}.
\end{lemma}

\begin{proof}
Let $\delta_{ij}$, $i,j\leq n$ be entries of $M$.
Let $\Event$ be the event that there are no zero columns in $M$.
Clearly, $\p(\Event)\geq 1-n(1-p)^n$.

Also, for each $1\leq j\leq s_0+1$, let $\Event_j= \Event_{col} (\ell_0, n_{j-1})$ be the event
introduced in Lemma~\ref{col} (with $s_0,\ell_0$ defined in \eqref{eq: l0s0 def}), and observe that, according to Lemma~\ref{col},
$\p(\Event_j)\geq 1-e^{-1.5np}$ for every  $j$.

Recall that $\sigma_x$ denotes a permutation $[n]$ such that $x_i^*=|x_{\sigma(i)}|$ for $i\le n$.
Pick any $x\in \st_0 \cup  \st_1$.
In the case $x\in \st_{0}$ set $m=m_1=1$ and $m_2=2$.
In the case $x\in \st_{1j}$ for some $1\leq j\leq s_0+1$ set $m=m_1=n_{j-1}$ and $m_2=n_j$.
Then by the definition of sets $\st_{0}, \st_1$ we have
$x^*_m>6d x^*_{m_2}$.
Let
$$
  J^\ell=J^\ell(x)=\sigma_x([m]), \quad  J^r=J^r(x)=\sigma_x([m_2-1]\setminus[m]), \quad \mbox{ and } \quad
  J(x)=(J^\ell\cup J^r)^c
$$
(if $x\in \st_{0}$ then $J^r=\emptyset$).
Note that by our definition we have $|x_i|>6d |x_u| $ for any $i\in J^\ell(x)$ and $u\in J(x)$, and that
$\max_{i\in J(x)}|x_i|\le x^*_{m_2}$.
Denote by $\il(x)$ the (random) set of rows of $M$ having exactly one 1 in $J^\ell(x)$ and no 1's in $J^r(x)$.
Now we recall that the event $\Event_{sum}$ was introduced in Lemma~\ref{bennett}
(we use it with $q=p$) and set
$$
\Event':= \Event\cap \Event_{sum}\cap \bigcap_{j=1}^{s_0+1} \Event_j.
$$
Clearly, conditioned on $\Event'$, the set $\il(x)$ is not empty for any $x\in \st_0 \cup  \st_1$.
By definition, for every $s\in \il(x)$  there exists
$j(s)\in J^\ell(x)$ such that
$$
   \supp R_{s}(M)\cap J^\ell(x)=\{j(s)\},\quad \supp R_{s}(M)\cap J^r(x)= \emptyset. 
$$
Since
$j(s)\in J^{\ell}(x)$
(which implies $|x_{j(s)}|\geq x^*_m> 6 d x^*_{m_2}$), we obtain
\begin{align*}
  |\langle R_{s} (M),\, x \rangle|
  &=\Big| x_{j(s)}
       +   \sum_{j\in J(x)} \delta _{sj} x_j \Big|
   \geq|x_{j(s)}|- x_{m_2}^* \sum_{j\in J(x)} \delta _{sj}
  \geq  x_{m}^* -  \frac{x_{m}^*}{6d} \sum_{j\in J(x)} \delta _{sj} .
\end{align*}
Observe that conditioned on $\Event_{sum}$ we have
$\sum_{j\in J(x)} \delta _{sj} \leq \sum_{j=1}^n \delta _{sj} \leq 3.5 pn=3.5 d$. Thus,
everywhere on $\Event'$ we have for all $x\in \st_0 \cup  \st_1$,
$$
  \|M x \| \geq |\langle R_{s} (M),\, x\rangle|
\geq x_{m}^*/3,\quad s\in \il(x).
$$
Finally, in the case $x\in \st_{0}$ we have $m=1$ and $\|x\|\leq  \sqrt{n}x^*_1$.
In the case $x\in \st_{1j}$ by Lemma~\ref{euclnorm} we have
$$
 \|x\|\leq   \frac{ 64(pn)^2}{(64p)^{\kappa}}\, x^*_{m},
$$
This proves the lower bound on $\|Mx\|/\|x\|$  conditioned on  $\Event'$.
The probability bound
 follows  by the union bound, Lemmas~\ref{bennett} and \ref{col}, and since $s_0\leq \ln n$,
 indeed
$$
 \p\left(\Event\cap \Event_{sum}\cap \bigcap_{j=1}^{s_0+1} \Event_j\r) \geq
 1-  n(1-p)^n - (s_0+2)e^{-1.5np} \leq 1-  n(1-p)^n - e^{-1.4 np} .
$$
\end{proof}

\subsection{Individual bounds for vectors  from $\st_2 \cup \st_3$}
\label{subs: nets}

In this section we provide  individual probability bounds for vectors from the nets constructed in
Lemma~\ref{l:nets}.
To obtain the lower bounds on $\|M x\|$, we consider
the behavior of the inner products $\la \row_i(M), x \ra$, more specifically, of the
L\'evy concentration function for $\la \row_i(M), x\ra$. To estimate this
function, we will consider $2m$ columns of $M$ corresponding
to the $m$ biggest and $m$ smallest (in absolute value) coordinates of $x$, where $m=n_{s_0+1}$ or $m=n_{s_0+2}$.
In a sense, our anti-concentration estimates will appear in the process of swapping $1$'s and $0$'s within a specially
chosen subset of the matrix rows. A crucial element in this process is to extract a pair of subsets of indices on which the chosen
matrix rows have only one non-zero component. This will allow to get anti-concentration bounds by ``sending'' the
non-zero component into the other index subset from the pair.
The main difficulty in this scheme comes from the restriction $2m p \leq 1/32$ from
Lemma~\ref{c:SJ}, which guarantees existence of sufficiently many required subsets (and rows)
but which cannot be directly applied to $m=n_{s_0+2}$. To resolve this problem we use idea from
\cite{LLTTY-TAMS}. We split
the initially fixed set of $2m$ columns into smaller subsets of columns of size at most $1/(64 p)$ each,
and create independent random variables corresponding to this splitting. 
Then we apply Proposition~\ref{prop: esseen},
allowing to deal with the L\'evy concentration function for sums of independent random variables.

We first describe subdivisions of $\Mc$ used in \cite{LLTTY-TAMS}.  Recall that $\Mc$ denotes
the class of all $n\times n$ matrices with $0/1$ entries.
We recall also that the probability measure $\Prob$ on $\Mc$ is always assumed to be induced by a
\Ber random matrix.
Given $J\subset [n]$ and $M\in \Mc$
denote
$$
  I (J, M) = \{i \leq n \, : \, |\supp \row_i(M) \cap J |       =1\}.
$$
By $\MSet _J$ we denote
the set of $n\times |J|$ matrices with $0/1$ entries and with columns indexed
by $J$.
Fix $q_0\leq n$ and a partition $J_0$, $J_1$, ..., $J_{q_0}$ of $[n]$.
Given subsets $I_1, \dots,I_{q_0}$ of $[n]$ and $V=(v_{ij})\in \MSet _{J_0}$, denote
$\ii = (I_1, \ldots, I_{q_0})$ and  consider the class
$$
   \f (\ii, V) =  \left\{M=(\mu_{ij})\in \Mc \, :\,  \forall  q\in [q_0] \quad I (J_q, M) =  I_q \,\,
     \mbox{ and } \,\,\forall i\leq n\, \forall j\in J_0  \,\,\,  \mu_{ij}= v_{ij} \right\}.
$$
In words, we fix the columns
indexed by $J_0$ and for each $q\in [q_0]$ we fix the row indices having exactly one $1$ in columns indexed by $J_q$.
Then, for any fixed partition $J_0$, $J_1$, ..., $J_{q_0}$, $\Mc$ is the disjoint union of classes $\f (\ii, V)$ over all $V\in \MSet _{J_0}$ and
all $\ii \in (\mathcal{P} ([n]))^{q_0}$, where  $\mathcal{P} (\cdot)$ denotes the power set.

\smallskip 

The following is an important, but simple observation.
\begin{lemma}\label{l: indep 20598}
Let $\f (\ii, V)$ be a non-empty class (defined as above), and denote by $\Prob_{\f}$ the induced probability measure on
$\f (\ii, V)$, i.e., let
$$
\Prob_{\f}(B):=\frac{\Prob(B)}{\Prob(\f (\ii, V))},\quad B\subset \f (\ii, V).
$$
Then the matrix rows for matrices in $\f (\ii, V)$ are mutually independent with respect to $\Prob_{\f}$,
in other words, a random matrix distributed according to $\Prob_{\f}$ has mutually independent rows.
\end{lemma}

%
%

Finally, given a vector $v\in \R^n$, a class $\f (\ii, V)$, indices $i\leq n$, $q\leq q_0$,
define
\begin{equation}\label{xiq}
   \xi _q(i) = \xi_q (M,v,i) := \sum _{j\in J_q} \delta _{ij} v_j,\quad M=(\delta_{ij})\in \f (\ii, V).
\end{equation}
We will view $\xi_q(i)$ as random variables on $\f (\ii, V)$ (with respect to the measure $\Prob_{\f}$).
It is not difficult to see that for every fixed $i$, the variables $\xi_1(i),\dots\xi_{q_0}(i)$ are mutually independent,
and, moreover, whenever $i\in I_q$, the variable $\xi_q(i)$ is uniformly distributed on the multiset $\{v_j\}_{j\in J_q}$.
Thus, we may apply Proposition~\ref{prop: esseen} to
$$
   \left|\la \row_i(M), v \ra\right| = \Big| \sum _{q=0}^{q_0} \xi_q(i)\Big|
$$
with some $\alpha >0$ satisfying  $\cf (\xi_q(i) , 1/3)\leq \alpha$ for every $i\in I_q$. This  gives
\begin{equation}\label{conc-inner}
\Prob_{\f} \left\{\left|\la \row_i(M), x+y \ra \right| \leq 1/3\right\} \leq \frac{C_0 \alpha }{\sqrt{(1-\alpha) |\{q\geq 1:\,i\in I_q\}|}},
\end{equation}
where $C_0$ is a positive absolute constant.

\smallskip

We are ready now to estimate individual probabilities.

\begin{lemma}[Individual probabilities]
\label{individual}
There exist absolute constants $C, C'>1>c_1>0$ such that the following holds.
Let $p\in (0, 1/64]$, $d=pn\geq 2$,
Set $m_0=   \lfloor 1/(64 p)\rfloor$ and
let $m_1$ and $m_2$ be such that
$$1\leq m_1<m_2\leq n-m_1.$$
 Let  $y\in \spn\{{\bf{1}}\}$ and assume that
$x\in \R^n$ satisfies
$$
   x^*_{m_1}> 2/3 \quad \mbox{ and } \quad x^*_i = 0 \, \, \, \mbox{ for every }\, \, i>  m_2.
$$
 Denote $m=\min(m_0, m_1)$ and consider the event
$$
  E(x, y) = \left\{
     M \in \Mc\, :\, \|M (x+y) \|\leq   \sqrt{ c_1 m d}   \r\}.
$$
Then in the case $m_1 \leq m_0$ one has
$$
   \Prob(E(x, y)\cap \Event_{card})\leq  2^{-m d/20},
$$
and in the case $m_1> C'  m_0$ one has
$$
   \Prob(E(x, y)\cap\Event_{card} )\leq  \left(\frac{C n}{m_1 d}\r)
   ^{m d/20},
$$
where $\Event_{card}$ is the event introduced in Lemma~\ref{c:SJ} with $\ell=2m$.
\end{lemma}

\begin{rem}\label{rem-ind}
We apply this lemma below for sets $\st_i$ with the following choice of parameters.
For $i=2$ we set
$$
 m_1 =m_0=  n_{s_0+1}=\max(30 \ell_0^{s_0-1}, \left\lfloor 1/(64p) \r\rfloor),
 \quad  m_2=n_{s_0+2}, \quad \mbox{and} \quad
p\leq 0.001,
$$
obtaining
\begin{equation*} \label{individual-one}
   \Prob(E(x, y)\cap \Event_{card} )\leq   2^{-n_{s_0+1} d/20}.
\end{equation*}
For $i=3$, we set
$$
m_1=n_{s_0+2}=\lfloor  n/\sqrt{d} \rfloor > m_0=n_{s_0+1}, \quad
m_2=n_{s_0+3}, \quad \mbox{and} \quad  p\leq 0.001,
$$
obtaining for large enough $d$,
\begin{equation*}\label{individual-two}
  \Prob(E(x, y)\cap \Event_{card} )\leq  \left(\frac{C n}{n_{s_0+2} d}\r)
   ^{n_{s_0+1} d/20} \leq \left( \sqrt{d}/(2C)\r)
   ^{-n_{s_0+1} d/20} .
\end{equation*}
\end{rem}

\medskip

To prove Lemma~\ref{individual} it will be convenient to use the same notation as in
Lemma~\ref{l:T0}.
Given two disjoint subsets $J^\ell$, $J^r\subset[n]$
and a matrix $M\in \Mc$, denote
$$
\il=\il(M):=\{i\le n :\,|\supp \row_i(M)\cap J^\ell|=1 \, \, \text{ and }\, \,\supp \row_i(M)\cap J^r=\emptyset\},
$$
and
$$
 \ir=\ir(M):=\{i\le n :\,\supp \row_i(M)\cap J^\ell=\emptyset\,\,\text{ and }\,\,|\supp \row_i(M)\cap J^r|=1\}.
$$
Here the upper indices $\ell$ and $r$ refer to {\it left} and {\it right}.

\begin{proof}
Let $d=pn$ and  fix $\gamma = mp/72= md/(72n)$.

Fix $x\in \R^n$ and $y\in \spn\{{\bf 1}\}$ satisfying the conditions of the lemma. Let $\sigma=\sigma_x$, that is,
 a permutation of $[n]$ such that
$x_i^*=|x_{\sigma(i)}|$ for all $i\le n$. Denote $q_0=m_1/m$ and without loss of generality
assume that either $q_0=1$ or that $q_0$ is a large enough integer. Let
$J^{\ell}_1, J_2^\ell, \ldots,  J^\ell_{q_0}$ be a partition
of $\sigma ([m_1])$ into sets of cardinality $m$ each, and
let $J^{r}_1, J_2^r, \ldots,  J^r_{q_0}$ be a partition
of $\sigma ([n-m_1+1, n])$ into sets of cardinality $m$ each. Denote
$$
 J_q:=J^\ell_q\cup J^r_q  \, \, \, \mbox{ for }\,\,\, q\in [q_0] \quad \mbox{ and } \quad
 J_0:= [n]\setminus \bigcup _{q=1}^{q_0} J_q.
$$
Then $J_0$, $J_1$, ..., $J_{q_0}$ is a partition of $[n]$, which we fix in this proof.
Let $M$ be a $0/1$ $n\times n$ matrix.
 For every pair  $J^\ell_q$, $J^r_q$, let the sets $\il_q(M)$ and  $\ir_q(M)$
be defined as
after Remark~\ref{rem-ind} and let $I_q(M)= \il_q(M) \cup \ir_q(M)$.
Since
$$
 |J_q|=2m \le 2m_0\le 1/(32p),
$$
 and by the definition of  the event $\Event_{card}$ (see Lemma~\ref{c:SJ} with $\ell=2m$), we have
 \begin{equation}\label{cond-card}
 |I_q(M)|\in[md/8,\,4md]
 \end{equation}
everywhere on $\Event_{card}$.
Now we represent $\Mc$ as a disjoint union of classes $\f (\ii, V)$ defined at the beginning of this subsection
with $V\in \MSet _{J_0}$ and  $\ii =(I_1, \ldots, I_q)$.
%
Since it is enough to prove a uniform upper bound for classes $\f (\ii, V)\cap \Event_{card}$
and since for every such non-empty class $\ii$ must satisfy (\ref{cond-card}) for every $q\leq q_0$,
we have
$$
  \Prob(E(x, y)\cap  \Event_{card} )\leq \max  \p (E(x, y) \cap \Event_{card}\,|\, \f (\ii, V))
  \leq \max  \p (E(x, y) |\, \f (\ii, V)),
$$
where the first maximum is taken over all  $\f (\ii, V)$
with $\f (\ii, V) \cap \Event_{card} \ne \emptyset$
and
the second maximum is taken over all $\f (\ii, V)$ with $I_q$'s satisfying
condition (\ref{cond-card}).

Fix any class $\f (\ii, V)$, where $\ii$ satisfies \eqref{cond-card}, and denote the corresponding induced probability measure on the class
by $\p_\f$, that is
$$
 \p_\f (\cdot) = \p( \cdot \, | \,  \f (\ii, V)).
$$
 Let
 $$
  I: = \bigcup _{q=1}^{q_0} I_q.
 $$
 Note that $|I|\leq 4 q_0 md$.
We first show that the set of $i$'s which belongs to many $I_q$'s is large.
More precisely, denote
$$
 A_i = \{ q\in[q_0]\, : \, i\in I_q\},\;\;i\in[n], \quad \quad \mbox{ and }\quad\quad
 I_{0}=\{i\leq n\, : \, |A_i|\ge \gamma   q_0\}.
$$
Then, using bounds on cardinalities of $I_q$'s, one has
$$
  m d q_0 /8 \leq \sum_{q=1}^{q_0} |I_q| = \sum_{i=1}^n |A_i| \leq |I_{0}| q_0 + (n-|I_{0}|) \gamma q_0
  \leq |I_{0}| q_0 + n \gamma q_0.
$$
Thus,
$$
   |I_0|\geq    m d/8 - n\gamma    \geq md/9.
$$
 Without loss of generality
we assume that $I_0=\{1, 2, \ldots |I_0|\}$ and  only consider the first
$k:=\lc  md/9 \rc$  indices from it. Then $[k]\subset I_0$.

Now, by definition, for matrices $M\in E(x, y)$ we have
$$
  \|M(x+y)\|^2 = \sum _{i=1}^n | \la \row_i(M), x+y \ra|^2 \le c_1\, md .
$$
Therefore there are at most $9 c_1 md$ rows  with
$| \langle \row_i(M), x+y) \rangle|\ge 1/3$. Hence,
$$
 |\{i\leq k\, : \,|  \langle \row_i(M), x+y \rangle|< 1/3\}|\ge
   md/9 - 9 c_1 md \ge (1/9- 9c_1)md.
$$
Let
 $
   k_0:= \lc (1/9- 9c_1) md\rc
 $
and for every $i\leq k$ denote
$$
  \Omega_i:=\{M\in\f (\ii, V) \, :\, |\la \row_i(M), x+y \ra|< 1/3 \}
  \quad \mbox{ and } \quad \Omega_0= \f (\ii, V) .
$$
Then
\begin{align*}
  \Prob_{\f}(E(x,y)) &\le  \sum _{B\subset [k]\atop |B|=k_0 } \,
  \Prob_{\f}\Big(\bigcap_{i\in B}\Omega_i\Big)
  \le   {k \choose k_0 }\, \max _{B\subset [k]\atop |B|=k_0 } \, \Prob_{\f}\Big(\bigcap_{i\in B} \Omega_i\Big).
\end{align*}
Without loss of generality we assume that the maximum above is attained at $B=[k_0]$. Then
\begin{equation} \label{ptensor}
  \Prob_{\f}(E(x, y))  \le
 \left(e/(81 c_1)\r)^{9c_1 md} \,\, \, \prod_{i=1}^{k_0} \, \Prob_{\f}(\Omega_{i}|\,\Omega_1\cap\ldots\cap \Omega_{i-1})
=\left(e/(81 c_1)\r)^{9c_1 md} \,\, \, \prod_{i=1}^{k_0} \, \Prob_{\f}(\Omega_{i}),
\end{equation}
where at the last step we used mutual independence of the events $\Omega_{i}$ (with respect to measure $\Prob_{\f}$),
see Lemma~\ref{l: indep 20598}.

 Next we estimate the factors in the product. Fix $i\leq k_0$ and $A_i= \{ q\, : \, i\in I_q\}$. Since, by our assumptions,
$i\in I_0$, we have
 $|A_i|\geq \gamma q_0$.
Consider the random variables $\xi_q(i)=\xi_q(M,x+y,i)$, $q\in A_i$,
defined in (\ref{xiq}). Then by
(\ref{conc-inner}) we have
\begin{align*}
  \Prob_{\f}(\Omega_{i}) &= \Prob_{\f}\big\{|\la \row_i(M), x+y \ra | < 1/3\big\}
\leq \cf_{\f}\Big(\sum _{q=0}^{q_0} \xi_q(i),1/3\Big)\\
&\leq\cf_{\f}\Big(\sum _{q\in A_i} \xi_q(i),1/3\Big)
  \leq \frac{C_0 \alpha}{\sqrt{(1-\alpha) |A_i|}}
\leq\frac{C_0 \alpha}{\sqrt{(1-\alpha) \gamma q_0}}
\end{align*}
where
 $\alpha = \max _{q\in A_i} \cf_{\f} (\xi_q (i), 1/3)$.
Moreover, in the case $q_0=1$ we just have
$$
  \Prob_{\f}(\Omega_{i}) \leq \alpha = \cf (\xi_1 (i), 1/3).
$$
Thus it remains to estimate $\cf_{\f} (\xi_q(i) , 1/3)$ for $q\in A_i$. Fix $q\in A_i$, so that $i\in I_q$.
Recall that, by construction, the intersection of the support of $\row_i(M)$ with $J_q$
is a singleton everywhere on $\f (\ii, V)$. Denote the corresponding index by $j(q,M)=j(q,M,i)$. Then
$$
  \xi _q(i) = \xi_q(M, x+y,i) = \sum _{j\in J_q} \delta _{ij} (x_j+y_1) = x_{j(q,M)}+y_1,
$$
and note that $|x_{j(q,M)}|>2/3$ whenever $j(q,M)\in J^\ell_q$ and
$x_{j(q,M)} =0$ whenever $j(q,M)\in J^r_q$.
Observe further that $\Prob_{\f}\big\{j(q,M)\in J^r_q\big\}=\Prob_{\f}\big\{j(q,M)\in J^\ell_q\big\}=1/2$.
Hence, we obtain
$$
   \cf_{\f} (\xi_q (i), 1/3) \leq 1/2:=\alpha.
$$

Combining the probability estimates starting with  (\ref{ptensor}) and
using that $\gamma = md/(72n)$,
we obtain in the case $q_0=m_1/m\geq C'$,
\begin{align*}
   \Prob_{\f}(E(x, y))&\leq  \left( \frac{e}{81 c_1}\r)^{9 c_1 md} \,\, \,    \left(\frac {C_0}{\sqrt{2\gamma q_0}}\r)^{(1/9- 9c_1) md}
   \\& =
    \left(\frac{e}{81c_1}\r)^{9 c_1 md} \,\, \,    \left(\frac{6 C_0 \sqrt{n}}{\sqrt{ m_1 d}}\r)^{(1/9-9 c_1)md} \leq
   \left(\frac{C_1 n}{m_1 d}\r)^{md/20},
\end{align*}
provided that $c_1$ is small enough and $C_1=36C_0^2$. Note that the bound is
meaningful only if $C'$ is large enough.
In the case $q_0=1$ we have
$$
   \Prob_{\f}(E(x,y))\leq \left( \frac{e}{81 c_1}\r)^{9 c_1 md}   \,\, \, \left( \frac{1}{2}\r)^{(1/9-9 c_1) md}   \leq
   \left( \frac{1}{2}\r)^{md/20},
$$
provided that $c_1$ is small enough.
This completes the proof.
\end{proof}

\subsection{Proof of Theorem \ref{steep}}

We are ready to complete the proof.
Denote
$$m= m_0=n_{s_0+1}:=\max(30 \ell_0^{s_0-1}, \left\lfloor 1/(64p) \r\rfloor)\in [n/(64 d), n/(2d)].$$
Lemma~\ref{l:T0} implies that
$$
  \p\left(\Big\{\exists\;x\in \st_0 \cup  \st_1 \, \, \, \mbox{ such that }
  \, \, \, \|M  x\| \leq    \frac{(64p)^{\kappa}}{ 192(pn)^2}\, \|x\|\Big\}\r)
  \leq n(1-p)^n+e^{-1.4np}.
$$

\smallskip

We now turn  to the remaining cases. Fix $j\in \{2, 3\}$.
Let
\begin{align}
 &\Event_{j}:=\Big\{M\in\Mc\,:\,\exists\, x\in \st_j \,\, \, \mbox{such that}\,\,\,\|M x\|\le
 \frac{\sqrt{ c_1 m d}}{2 \, b_j }\, \|x\|\Big\},\notag
\end{align}
where $c_1$ is the constant from Lemma~\ref{individual},
and $b_{2}= 384(pn)^3/(64p)^{\kappa}$, $b_{3}=384\ct (pn)^{3.5}/(64p)^{\kappa}$.

Recall that $\Event_{nrm}$ was defined in Proposition~\ref{nettri}.
For any matrix $M\in \Event_{j}\cap \Event_{nrm}$ there exists $x=x(M)\in \st_j$ satisfying
$$
 \|M x\|\le  \frac{\sqrt{ c_1 m d}}{2 \, b_j }\, \|x\|.
$$
Normalize $x$ so that $x_{n_{s_0+j-1}}^{*}=1$, that is, $x\in \st_j'$.
By Lemma~\ref{euclnorm} we have  $\|x\|\leq  b_j$.

Let $\Net_j=\Net_j'+\Net_j''$ be the net constructed in Lemma~\ref{l:nets}. Then there exist
$u\in \Net_j'$ with
$$u_{{s_0+j-1}}^{*}\geq 1-1/(\ct\sqrt{d})>2/3$$
and $u_{\ell}^{*}=0$ for $\ell> n_{s_0+j}$, and $w\in \Net_j''\subset \spn\{{\bf 1}\}$, such that
$|||x-(u+w)|||\leq \sqrt{2n}/(\ct \sqrt{d}).$
 Applying Proposition~\ref{nettri} (where $\Event_{nrm}$ was introduced), and using that $\ct$ is large enough,
 we obtain that
 for every matrix $M\in \Event_{j}\cap \Event_{nrm}$ there exist
$u=u(M)\in \Net_j'$ and $w=w(M)\in \Net_j''\subset \spn\{{\bf 1}\}$ with
\begin{equation}
\label{ctau}
   \|M (u+w)\|\le \|Mx\| + \|M (x-u-w)\| \leq
  \sqrt{ c_1 m d}/2 + 200 \sqrt{2n}/\ct \leq
  \sqrt{ c_1 m d}.
\end{equation}
Using   our choice of $n_{s_0+1}$, $n_{s_0+2}$, $n_{s_0+3}$, Lemma~\ref{l:nets}, and
Lemma~\ref{individual} twice ---
first  with $m_1=m_0=n_{s_0+1}$, $m_2=n_{s_0+2}$,
then with $m_1=n_{s_0+2}>m_0=n_{s_0+1}$, $m_2=n_{s_0+3}$  (see Remark~\ref{rem-ind}),
we obtain that for small enough $\aaa$ and large enough $d$ the probability
$\p\left(\Event_{2} \cap \Event _{nrm}\cap  \Event_{card}\r)$ is bounded by
$$
  \exp \left( 2 n_{s_0+2} \ln d\r) 2^{-n_{s_0+1} d/20} \leq
 \exp \left(- n_{s_0+1} d/30 \r)\leq  \exp \left(- n /2000 \r)
$$
and that the probability
$\p\left(\Event_{3} \cap \Event _{nrm}\cap  \Event_{card}\r)$ is bounded by
$$
 \exp \left(2 n_{s_0+3} \ln d\r) \left(\sqrt{d}/(2C)\r)
   ^{-n_{s_0+1} d/20} \leq
 \exp \left(-  n \ln d/10000\r),
$$
where  $\Event_{card}$ is the event
introduced in Lemma~\ref{c:SJ} with $\ell=2m$.

Combining all three cases we obtain that the desired bound holds
for all $x\in \st$ with
probability at most
$$
  2\exp \left(- n /2000 \r)
   + \p\left( \Event _{norm}^c\r)
   + \p\left( \Event_{card}^c\r).
$$
It remains to note that since $np$ is large, by Lemma~\ref{bdd} (applied with $t=30$)
 and by Lemma~\ref{c:SJ}, 
$$
 \p\left( \Event _{nrm}^c\r)
   + \p\left( \Event_{card}^c\r)\leq   4 e^{-225n p}+2\exp(-n/500) \leq  \exp(-10pn) .
$$
\kkk

\subsection{Proof of Theorem~\ref{complement}}



\begin{proof}
Clearly, it is enough to show that
$\normr_n(r)\setminus (\gncvectors_n(r,\gfn,\delta,\rho) \cup \st) \subset \Rv.$
Let $x\in\normr_n(r)\setminus \st$ and set $\sigma:=\sigma_x$. Note that $|x_{n_{s_0+2}}|\leq \ct \sqrt{d}$,
where $s_0$ was defined in \eqref{eq: l0s0 def}.
Denote $m_0=\lfloor n/\ln^2 d\rfloor> 2{n_{s_0+2}}$. 

\smallskip

Assume first that $x$ does not satisfy (\ref{cond2}).
Then by Lemma~\ref{a-c-cond2},
 $x\in \BB(\rho)$.
If $x_{m_0}^*\leq \ln^2 d$ then denoting
$k=m_0$, $A=[k, n]$, and using the definition of $\BB(\rho)$, we observe
$$
 \|x_{\sigma(A)}\|\geq\sqrt{(n-n_{s_0+3}-k)(1-\rho)}\geq \sqrt{n/2},
$$
whence
$$
  \frac{\|x_{\sigma(A)}\|}{\|x_{\sigma(A)}\|_\infty} \geq \frac{\sqrt{n/2}}{\ln^2 d}\geq \frac{C_0}{\sqrt{p}}.
$$
On the other hand,  $x_{m_{0}}^*\leq |x_{n_{s_0+2}}|\leq \ct \sqrt{d}$, hence $\|x_{\sigma(A)}\|\leq \ct \sqrt{dn}$.
This implies that $x\in \Rv_k^1\subset \Rv$.

Now, if $x_{m_0}^*> \ln^2 d$ then denoting
$k=n_{s_0+2}$, $A=[k, n]$,  we get
$$
 \|x_{\sigma(A)}\|^2\geq \sum _{i=n_{s_0+2}}^{m_{0}} (x_i^*)^2 \geq
 (m_0/2)\ln^4 d \geq (n/4)\ln^2 d,
$$
whence
$$
  \frac{\|x_{\sigma(A)}\|}{\|x_{\sigma(A)}\|_\infty} \geq \frac{\sqrt{n}\ln  d}{2\ct \sqrt{d}}\geq \frac{C_0}{\sqrt{p}}.
$$
As in the previous case  we have $\|x_{\sigma(A)}\|\leq \ct \sqrt{dn}$, which
 implies that $x\in \Rv_k^1\subset \Rv$.

\medskip

 Next we assume that $x$ does  satisfy (\ref{cond2}). Then, by the definition of the set $\gncvectors_n(r,\gfn,\delta,\rho)$
 and our function $\gfn$,
 $x$ does not satisfy the following condition:
\begin{align*}
\label{cond1}
  \forall i\leq \frac{1}{64p} : \, \, \, x^*_i\leq  \exp (\ln ^2 (2n/i))
  \qand
  \forall \frac{1}{64p} < i\leq n : \, \, \, x^*_i\leq (2n/i)^{3/2}.
\end{align*}
 We fix the smallest value of $j\geq 1$
 which breaks this condition  and consider several cases.
 Note that since $x\in \normr_n(r)$, we must have $j\leq rn$.

\smallskip

\noindent
 {\it Case 1. }  $2 m_0 \leq j\leq rn$. In this case by the conditions and by
 minimality of $j$, we have
 $x_{m_0}^*\leq (2n/m_0)^{3/2}$ and $x_j^*\geq (2n/j)^{3/2}$. Take $k=m_0$ and $A=[k, n]$.
 Then we have
 $$
   \|x_{\sigma(A)}\|\geq \sqrt{j - m_0+1 } \, x_{j}^* \geq
   \sqrt{j/2} \, (2n/j)^{3/2}\geq \sqrt{rn/2} \, (2/r)^{3/2} =  2\sqrt{n} /r ,
 $$
 hence
 $$
  \frac{\|x_{\sigma(A)}\|}{\|x_{\sigma(A)}\|_\infty} \geq \left(\frac{2}{r}\r)\,
  \frac{ \sqrt{n}}{(2n/m_0)^{3/2}}\geq\left(\frac{2}{r}\r) \frac{ \sqrt{n}}{(2\ln d)^{3}}
  \geq \frac{C_0}{\sqrt{p}}.
$$
As above  we have $\|x_{\sigma(A)}\|\leq \ct \sqrt{dn}$, which
 implies that $x\in \Rv_k^2\subset \Rv$.

\smallskip

\noindent
 {\it Case 2. } $16 C_0^2 n/d \leq j\leq 2 m_0$. Take $k=\lceil j/2\rceil$ and $A=[k, n]$.
 Then we have $x_{k}^*\leq (2n/k)^{3/2}\leq (4n/j)^{3/2}$, $x_j\geq (2n/j)^{3/2}$, and
$$
  \|x_{\sigma(A)}\|\geq \sqrt{j - k +1 } \, x_{j}^* \geq
   \sqrt{j/2} \, (2n/j)^{3/2} \geq (2/r)\, \sqrt{n}.
$$
Therefore,
$$
  \frac{\|x_{\sigma(A)}\|}{\|x_{\sigma(A)}\|_\infty} \geq \left(\frac{j}{2}\r)^{1/2}\frac{ (2n/j)^{3/2}}{(4n/j)^{3/2}}\geq \frac{C_0}{\sqrt{p}}.
$$
Since $x\not\in \st$, we observe $x_k^*\leq \ct^2 d$, hence $\|x_{\sigma(A)}\|\leq \ct^2 d\sqrt{n}$ and
$x\in \Rv_k^2\subset \Rv$.

\medskip

In the rest of the proof we show that we must necessarily have  $j\geq 16 C_0^2 n/d$.

\smallskip

\noindent
 {\it Case 3. } $n_{s_0+1}\leq j <C_1  n/d $, where $C_1=16 C_0^2$.
 Using that $x\not\in \st$,
 in this case we have
$$
   \ct^2 d \geq  x_j^*\geq  \left(\frac{2n}{j}\r)^{3/2}\geq \left(\frac{2 d}{C_1}\r)^{3/2},
$$
which is impossible for  large enough $d$.

\smallskip

\noindent
 {\it Case 4. } $n_{s_0}\leq j < n_{s_0+1}$.
 Using that $x\not\in \st$ and that
 $n_{s_0+1}=\left\lfloor 1/(64p) \r\rfloor=\left\lfloor n/(64d) \r\rfloor$, in this case we have
$$
   (6d) \ct^2 d  \geq x_j^*\geq
   \exp (\ln ^2 (2n/j))\geq \exp (\ln ^2 (2n/n_{s_0+1}))\geq
   \exp (\ln ^2 (128 d))
$$
which is impossible for  large enough $d$.

\smallskip

\noindent
 {\it Case 5. } $n_k\leq j < n_{k+1}$ for some $1\leq k\leq s_0-1$.
 Recall that $n_k=30\ell _0^{k-1}$ and recall also that if $s_0>1$
 (as in this case) then $p\leq c\sqrt{n\ln n}$.
 Using that $x\not\in \st$, in this case we have
\begin{equation*}\label{case5in}
   (\ct^2 d) (6d)^{s_0-k+1} \geq   x_j^*\geq \exp (\ln ^2 (2n/j))\geq
   \exp (\ln ^2 (2n/(30\ell_0^k))),
\end{equation*}
hence
\begin{equation}\label{case5in}
  (\ct^2 d) (6d)^{s_0+1} \geq   (6d)^{k}  \exp (\ln ^2 (2n/(30\ell_0^k))).
\end{equation}
Considering the function $f(k):= k \ln (6d) +   \ln ^2 (2n/(30\ell_0^k)$,
we observe that its derivative is linear in $k$, therefore $f$ attains its maximum
either at $k=1$ or at $k=s_0-1$. Thus, to  show that (\ref{case5in})
is impossible it is enough to consider $k=1, s_0-1$ only. Let $k=1$.
 By (\ref{kappain}),
$(6d)^{s_0}\leq (6d) \,1/(64p)^{\kappa}$, where $\kappa = \frac{\ln (6d)}{\ln \ell_0}$.
Therefore, the logarithm of the left hand side of (\ref{case5in}) is
\begin{equation}\label{boundln}
   \ln ((\ct^2 d) (6d)^{s_0+1})\leq 4\ln d + \frac{\ln (6d)}{ \ln \ell_0}\, \ln (1/64p) .
\end{equation}

On the other hand, $n/\ell_0 \geq (4\ln (1/p) )/p$, therefore the
logarithm of the left hand side of (\ref{case5in}) is larger than
$\ln ^2 (\ln (1/p)/(4p ))$.  Thus,
it is enough to check that
$$
 (1/2) \ln ^2 (\ln (1/p)/(4p )) \geq 4 \ln d
 \qand
 (1/2) \ln ^2 (\ln (1/p)/(4p )) \ln \ell_0 \geq \ln (6d)\, \ln (1/64p) .
$$
Both inequalities follows since $p\leq c\sqrt{n\ln n}$, $d=pn$, $d$ and $n$
are large enough, and since $\ell_0\geq 25$.
Next assume that  $k=s_0-1$. Note that in this case $\ell_0^k\leq n/(64 d)$.
Thus, to disprove (\ref{case5in}) it is enough to show that
$$
  \ln ^2 (64d/15) \geq \ln (36 \ct^2 d^3),
$$
which clearly holds for large enough $d$.

\smallskip

\noindent
 {\it Case 6. } $2\leq j< 30$. In this case we have
 $$
   (\ct^2 d) (6d)^{s_0+1} \geq x_j^* \geq \exp (\ln ^2 (2n/j))\geq
   \exp (\ln ^2 (2n/30)),
 $$
By (\ref{boundln}) this implies
$$
   4\ln d + \frac{\ln (6d)}{ \ln \ell_0}\, \ln (1/64p)\geq \ln ^2 (2n/30),
$$
which is impossible.

\smallskip

\noindent
 {\it Case 7. } $j=1$. In this case we have
 $
   (\ct^2 d) (6d)^{s_0+2} \geq x_1^*\geq  \exp (\ln ^2 (2n))
 $
and we proceed as in Case~6.
\end{proof}

\section{Proof of the main theorem}\label{s: main th}

In this section, we combine the results of Sections~\ref{s: unstructured}, \ref{steep:constant p},
and \ref{s: steep}, as well as Subsection~\ref{subs: lower b}
to prove the main theorem, Theorems~\ref{th: main}, and the following improvement
for the case of constant $p$:

\begin{theor}\label{const-p-th}
There exists an absolute positive constant $c$ with the following property.
Let $q\in (0, c)$ be a parameter (independent of $n$).
Then there exist $C_q$ and $n_q\geq 1$ (both depend only  on $q$),
 such that  for every $n\geq n_q$ and  every $p\in (q, c)$  a
 \Ber $n\times n$ random matrix $M_n$ satisfies
$$\Prob\big\{\mbox{$M_n$ is singular}\big\}=(2+o_n(1))n\,(1-p)^n,$$
and, moreover, for every $t>0$,
$$
\Prob\big\{s_{\min}(M_n)\leq C_q\, n^{-2.5}\, t \big\}\leq t+(1+o_n(1))\Prob\big\{\mbox{$M_n$ is singular}\big\}
=t+(2+o_n(1))n\,(1-p)^n.
$$
\end{theor}

At this stage, the scheme of the proof to a large extent follows the approach of Rudelson and Vershynin
developed in \cite{RV}.
However, a crucial part of their argument --- ``invertibility via distance'' (see \cite[Lemma~3.5]{RV}) ---
will be reworked in order to keep sharp probability estimates
for the matrix singularity and to be able to bind this part of the argument with the previous
sections, where we essentially condition on row- and column-sums of our matrix.

 We start by restating main results of Sections~\ref{steep:constant p} and \ref{s: steep}
 using the vector class $\gncvectors_n(r,\gfn,\delta,\rho)$ defined by \eqref{eq: gnc def},
 together with Lemma~\ref{l:closure}.


\begin{cor}\label{cor: steep}
There are universal constants $C\geq 1$,
$\delta,\rho\in(0,1)$ and $r\in(0,1)$ with the following property. Let $M_n$
be a random matrix satisfying \eqref{eq: assumptions} with $C$
and let the growth function $\gfn$ be given by \eqref{gfn-str}.
Then
\begin{equation}\label{singval}
\Prob\Big\{\|M_n x\|\leq a_n^{-1}
\|x\|\,\,\, \mbox{ for some }\,\,\, x\notin \bigcup\limits_{\lambda\geq 0}\big(\lambda\,\gncvectors_n(r,\gfn,\delta,\rho)\big)\Big\}
=(1+o_n(1))n\,(1-p)^n,
\end{equation}
where $a_n=\frac{(pn)^2}{c(64p)^{\kappa}}\, \max\left(1, p^{1.5} n\r)$,
$\kappa = \kappa(p):= (\ln (6pn))/\ln \big\lfloor\frac{pn}{4\ln (1/p)}\big\rfloor.$
\end{cor}

Further, Theorems~\ref{t:steep},~\ref{compl-1} and Lemma~\ref{l:closure} are combined as follows.

\begin{cor}\label{cor: steep2}
There are universal positive constants $c, C$ with the following property.
Let $q\in (0, c)$ be a parameter. Then there exist $n_0=n_0(q)\geq 1$,
$r=r(q), \rho=\rho (q)\in (0,1)$ such that  for $n\geq n_0$,  $p\in (q, c)$, $\delta = r/3$,
$\gfn(t)=(2t)^{3/2}$, 
the random \Ber $n\times n$ matrix $M_n$ satisfies (\ref{singval}) with
$a_n=C \sqrt{n \ln(e/p)}$.
\end{cor}

Below is our version of ``invertibility via distance,'' which deals with {\it pairs} of columns.
\begin{lemma}[Invertibility via distance]\label{l: inv via dist}
Let $r,\delta,\rho\in(0,1)$, and let $\gfn$ be a growth function.
Further, let $n\geq 6/r$ and let $A$ be an $n\times n$ random matrix. Then for any $t>0$ we have
\begin{align*}
\Prob\big\{&\|A x\|\leq t\,\|x\|\quad \mbox{ for some }\quad x\in \gncvectors_n(r,\gfn,\delta,\rho)\big\}\\
&\leq \frac{2}{(rn)^2}\sum\limits_{i\neq j}\Prob\big\{\dist(H_i(A),\Col_i(A))\leq t\, b_n
\qand \dist(H_j(A),\Col_j(A))\leq t\, b_n \big\},
\end{align*}
where the sum is taken over all ordered pairs $(i,j)$ with $i\neq j$ and $b_n=\sum_{i=1}^n \gfn(i)$.
\end{lemma}
\begin{proof}
For every $i\neq j$, denote by ${\bf 1}_{ij}$ the indicator of the event
$$
\Event_{ij}:= \big\{\dist(H_i(A),\Col_i(A))\leq t\, b_n\qand \dist(H_j(A),\Col_j(A))\leq t\, b_n\big\}.
$$
The condition
$$
\|A x\|\leq t\,\|x\|
$$
for some $x\in \gncvectors_n=\gncvectors_n(r,\gfn,\delta,\rho)$
implies that for every $i\leq n$,
$$
 |x_i|\, \dist(H_i(A),\Col_i(A))\leq\|Ax\|\leq  t\,b_n,
$$
where the last inequality follows from the definition of $\gncvectors_n$.
Since $x^*_{\lfloor rn\rfloor }=1$, we get that everywhere on the event
$\{\|A x\|\leq t\,\|x\|\mbox{ for some }x\in \gncvectors_n\}$ there are at least
$\lfloor rn\rfloor\,(\lfloor rn\rfloor-1)\geq (rn)^2/2$ ordered pairs of indices $(i,j)$ such that for each pair
the event $\Event_{ij}$ occurs.
Rewriting this  assertion in terms of indicators, we observe
$$
\{\|A x\|\leq t\,\|x\|\mbox{ for some }x\in \gncvectors_n\}
\subset\Big\{\sum\limits_{i\neq j}{\bf 1}_{ij}\geq (rn)^2/2\Big\}.
$$
Applying Markov's inequality in order to estimate probability of the event on the right hand side, we obtain
the desired result.
\end{proof}

\begin{proof}[Proof of Theorems~\ref{th: main} and \ref{const-p-th}]
The proofs of both theorems are  almost the same, the only difference is that
Theorem~\ref{th: main} uses  Corollary~\ref{cor: steep2} while  Theorem~\ref{th: main}
uses Corollary~\ref{cor: steep}.
Let parameters $\delta,\rho,r,\gfn, a_n$ be taken from Corollary~\ref{cor: steep} or
from Corollary~\ref{cor: steep2} correspondingly. We always
write $\gncvectors_n$ for $\gncvectors_n(r,\gfn,\delta,\rho)$.
Let  $b_n=\sum_{i=1}^n \gfn(i)$.
Without loss of generality, we can assume that $n\geq 6/r$.
 Fix  $t\in (0, 1]$, and denote by $\Event$ the complement of the event
$$
\Big\{\|M_n x\|\leq a_n^{-1}\|x\|\;
\mbox{ or }\;\|M_n^\top x\|\leq a_n^{-1} \|x\|\quad
\mbox{ for some }\quad x\notin\bigcup\limits_{\lambda\geq 0}\big(\lambda\,\gncvectors_n\big)\Big\}.
$$
For $i=1,2$ denote
$$
  \Event_i:= \big\{\dist(H_i(M_n),\Col_i(M_n))\leq a_n^{-1} \, t\big\}.
$$
Applying Corollary~\ref{cor: steep} (or Corollary~\ref{cor: steep2}), Lemma~\ref{l: inv via dist} and
 the invariance of the conditional distribution of $M_n$ given $\Event$
under permutation of columns, we obtain
\begin{align*}
\Prob&\big\{s_{\min}(M_n)\leq (a_n b_n)^{-1} t\big\}
\\&\leq (2+o_n(1))n\,(1-p)^n
 +\Prob\big(\big\{\|M_n x\|
\leq (a_n b_n)^{-1}\, t\|x\|\quad \mbox{ for some }\quad x\in \gncvectors_n\big\}\cap \Event\big)
\\&\leq (2+o_n(1))n\,(1-p)^n+ \frac{2}{r^2}\, \Prob\big(\Event\cap \Event_1\cap \Event_2\big)  .
\end{align*}

At the next step, we consider events
$$
  \Omega_i:= \big\{|\supp\Col_i(M_n)|\in [pn/8, 8pn]\big\},\, i=1,2,\qand
   \Omega:= \Omega_1\cup \Omega_2.
$$
Since columns of $M$ are independent and  consist of i.i.d. 
\Ber variables, applying Lemma~\ref{bennett}, we observe 
$$
 \Prob\big(\Omega^c\big)= \Prob\big(\Omega_1^c\big)\Prob\big(\Omega_2^c\big) 
 \leq    (1-p)^{n}.
$$
Therefore, in view of equidistribution of the first two columns, we get
\begin{align*}
\Prob&\big(\Event\cap \Event_1\cap \Event_2\big)
\leq (1-p)^n+\Prob\big(\Event\cap \Event_1\cap \Event_2\cap \Omega \big)
\leq (1-p)^n+2\Prob\big(\Event\cap \Event_1\cap \Omega_1\big).
\end{align*}
Denote by ${\bf Y}$ a random unit vector orthogonal to (and measurable with respect to) $H_1(M_n)$.
Note that on the event $\Event_1$
the vector ${\bf Y}$ satisfies
$$
 |\langle{\bf Y},\Col_1(M_n)\rangle|= \|M_n^\top {\bf Y}\|\leq a_n^{-1}\, t \, \|{\bf Y}\|,
$$
which implies that on the event $\Event\cap \Event_1$ we also have
${\bf Y}^*_{\lfloor r n\rfloor}\neq 0$, and
${\bf Z}:={\bf Y}/{\bf Y}^*_{\lfloor r n\rfloor}\in \gncvectors_n$.
By the definition of $\gncvectors_n$, we have $\|{\bf Z}\|\leq b_n$,
therefore,
\begin{align*}
&P_0:=\Prob\big(\Event\cap \Event_1\cap \Omega_1\big)\leq \Prob\big(\Omega_1 \cap \big\{\mbox{There is $Z\in H_1(M_n)^\perp\cap \gncvectors_n$}:\;
|\langle Z,\Col_1(M_n)\rangle|\leq a_n^{-1}\, b_n\,  t\big\}\big).
\end{align*}
On the other hand, applying Theorem~\ref{th: gradual} with $R=2$, we get that
for some constants $K_1\geq 1$ and $K_2\geq 4$, with probability at least $1-\exp(-2pn)$,
\begin{align*}
&H_1(M_n)^\perp\cap \gncvectors_n
\subset\big\{
x\in\normr_n(r):\;\bal_n(x,m,K_1,K_2)\geq \exp(2pn)\,\,\, \mbox{ for any }\,\,\, m\in [pn/8, 8pn]
\big\}.
\end{align*}
Combining the last two assertions and applying Theorem~\ref{p: cf est}, we observe
\begin{align*}
P_0\leq \exp(-2pn)+\Prob\big(&\Omega_1\cap\big\{\mbox{There is $Z\in H_1(M_n)^\perp\cap \gncvectors_n$}:\;
|\langle Z,\Col_1(M_n)\rangle|\leq a_n^{-1}\, b_n\,  t,\mbox{ and }\\
&\bal_n(Z,m,K_1,K_2)\geq \exp(2pn)\mbox{ for any }m\in [pn/8, 8pn]\big\}\big)\\
&\hspace{-3.6cm}\leq
\exp(-2pn)+\sup\limits_{\substack{m\in [pn/8, 8pn],\,y\in \normr_n(r),\\
\bal_n(y,m,K_1,K_2)\geq \exp(2pn)}}\Prob\big\{|\langle y,\Col_1(M_n)\rangle|
\leq a_n^{-1}b_n\, t\,\,\big|\,\, |\supp\Col_1(M_n)|=m\big\}
\\&\hspace{-3.6cm}\leq
 (1+C_{\text{\tiny\ref{p: cf est}}})\exp(-2pn)+  \frac{C_{\text{\tiny\ref{p: cf est}}}b_n}{ a_n\sqrt{pn/8}} \, t.
\end{align*}
Thus
$$
  \Prob\big\{s_{\min}(M_n)\leq (a_n b_n)^{-1} t\big\}
  \leq  (2+o_n(1))n\,(1-p)^n  + \frac{8 C_{\text{\tiny\ref{p: cf est}}}b_n}{ r^2 \, a_n\sqrt{pn}} \, t.
$$
By rescaling $t$ we obtain
$$
\Prob\Big\{s_{\min}(M_n)\leq \frac{r^2 \,\sqrt{pn}}{(8 C_{\text{\tiny\ref{p: cf est}}}b_n^2)}\, t\Big\}
\leq (2+o_n(1))n\,(1-p)^n+t,\quad 0\leq t\leq \frac{8 C_{\text{\tiny\ref{p: cf est}}}b_n}{ r^2 \, a_n\sqrt{pn}}.
$$

\medskip

In the case of constant $p$ (applying Corollary~\ref{cor: steep2}) we have $a_n=C \sqrt{n \ln(e/p)}$ and $b_n\leq 2\sqrt{3}n^{3/2}$,
and we get the small ball probability estimate of Theorem~\ref{const-p-th}.

\medskip

In the case of ``general'' $p$ (with the application of Corollary~\ref{cor: steep}) we have
$a_n=\frac{(pn)^2}{c(64p)^{\kappa}}\, \max\left(1, p^{1.5} n\r)$ and $b_n\leq \exp(1.5\ln^2(2n))$.
Therefore, $$\frac{r^2 \,\sqrt{pn}}{(8 C_{\text{\tiny\ref{p: cf est}}}b_n^2)}\geq \exp(-3\ln^2(2n))$$ for large enough $n$,
and the $s_{\min}$ estimate follows.

\medskip

In both cases the upper bound on $t$, $\frac{8 C_{\text{\tiny\ref{p: cf est}}}b_n}{ r^2 \, a_n\sqrt{pn}}$,
is greater than $1$, so we may omit it.

Finally, applying the argument of Subsection~\ref{subs: lower b}, we get the matching lower bound for the singularity probability.
This completes the proof.
\end{proof}

\section{Open questions}\label{s: further}

The result of this paper leaves open the problem of estimating the singularity probability
for Bernoulli matrices in two regimes: when $n p_n$ is logarithmic in $n$ and when $p_n$ is larger
than the constant $C^{-1}$ from Theorem~\ref{th: main}.

For the first regime, we recall that the singularity probability of $M_n$, with $n p_n$ in a (small) neighborhood of $\ln n$,
was determined up to the $1+o_n(1)$ multiple in the work of Basak--Rudelson \cite{BasRud-sharp}.
Definitely, it would be of interest to bridge that result and the main theorem of this paper.
\begin{Problem}[A brigde: Theorem~\ref{th: main} to Basak--Rudelson]
Let $p_n$ satisfy
$$1\leq \liminf np_n/\ln n\leq \limsup np_n/\ln n<\infty,$$
and for each $n$ let $M_n$
be the $n\times n$ matrix with i.i.d.\ Bernoulli($p_n$) entries.
Show that
$$
\Prob\big\{
M_n\mbox{ is singular}
\big\}=(1+o_n(1))\Prob\big\{M_n\mbox{ has a zero row or a zero column}\big\}.
$$
\end{Problem}
Note that the main technical result for unstructured (gradual non-constant) vectors,
Theorem~\ref{th: gradual} proved in Section~\ref{s: unstructured},
 remains valid for these values of $p_n$. It may be therefore expected that the above problem can be positively resolved
by finding an efficient treatment for the structured vectors (the complement of gradual non-constant vectors),
which would replace (or augment) the argument from Section~\ref{s: steep}.

\medskip

On the contrary, the second problem --- singularity of random Bernoulli matrices with large values of $p_n$
--- seem to require
essential new arguments for working with the unstructured  vectors as the
basic idea of Section~\ref{s: unstructured} --- gaining on anti-concentration estimates by grouping
together several components of a random vector --- does not seem to be applicable in this regime.
\begin{Problem}[Optimal singularity probability for dense Bernoulli matrices below the $1/2$ threshold]
Let the sequence $p_n$ satisfy
$$0< \liminf p_n\leq \limsup p_n<1/2.$$ Show that
\begin{align*}
\Prob\big\{
M_n\mbox{ is singular}
\big\}
&=(1+o_n(1))\Prob\big\{M_n\mbox{ has a zero row or a zero column}\big\}
=(2+o_n(1))n\,(1-p_n)^n.
\end{align*}
\end{Problem}

\subsection*{Acknowledgments}

K.T.\ was partially supported by the Sloan Research Fellowship.

\address

\end{document}